\documentclass[11pt]{article}

\usepackage[a4paper,margin=1in]{geometry}
\usepackage[T1]{fontenc}
\usepackage[utf8]{inputenc}
\usepackage{lmodern}
\usepackage{graphicx}
\usepackage{tikz}
\usetikzlibrary{arrows.meta,calc}
\usepackage{subcaption}
\usepackage{float}
\usepackage{microtype}
\usepackage{amsmath,amssymb,amsthm,mathtools}
\usepackage{bm}
\usepackage{enumitem}
\usepackage{xcolor}
\usepackage{hyperref}
\usepackage[nameinlink,noabbrev]{cleveref}
\usepackage{mathrsfs}
\usepackage{setspace}
\usepackage{titlesec}
\usepackage{fancyhdr}
\usepackage{orcidlink}
\usepackage[misc]{ifsym}
\usepackage[affil-it]{authblk}

\hypersetup{
  colorlinks=true,
  linkcolor=blue!50!black,
  citecolor=blue!50!black,
  urlcolor=blue!50!black,
  pdftitle={Thin Domains, Reduction, and Slow Manifolds},
  pdfauthor={Christian Kuehn and Jan-Eric Sulzbach}
}

\newcommand\blfootnote[1]{%
	\begingroup
	\renewcommand\thefootnote{}\footnote{#1}%
	\addtocounter{footnote}{-1}%
	\endgroup
}

\newtheorem{theorem}{Theorem}[section]
\newtheorem{lemma}[theorem]{Lemma}
\newtheorem{proposition}[theorem]{Proposition}
\newtheorem{corollary}[theorem]{Corollary}
\theoremstyle{definition}

\newtheorem{assumption}[theorem]{Assumption}
\theoremstyle{remark}
\newtheorem{remark}[theorem]{Remark}
\newtheorem{example}[theorem]{Example}

\numberwithin{equation}{section}

\newcommand{\eps}{\varepsilon}
\newcommand{\R}{\mathbb{R}}
\newcommand{\N}{\mathbb{N}}
\newcommand{\Qe}{Q_\varepsilon}
\newcommand{\Q}{Q}
\newcommand{\Om}{\Omega}

\newcommand{\Id}{\mathrm{Id}}
\newcommand{\mc}[1]{\mathcal{#1}}

\newcommand{\norm}[1]{\left\lVert#1\right\rVert}
\newcommand{\set}[1]{\left\{#1\right\}}
\newcommand{\pr}{\textnormal{pr}}

\begin{document}
\title{\bf Thin Domains, Reduction, and Slow Manifolds}

\author{Christian Kuehn \orcidlink{0000-0002-7063-6173}$^{1,2,3}$
\& Jan-Eric Sulzbach \orcidlink{ 0000-0002-9446-2366}$^{1,4}$ \thanks{JES acknowledges partial support from DFG grant 571660837.}
}

\date{
	\small{$^1$\textit{Technical University of Munich, School of Computation, Information and Technology, \\ Department of Mathematics, Boltzmannstraße 3, 85748 Garching, Germany} \\
	$^2$\textit{Munich Data Science Institute (MDSI), Garching, Germany } \\
	$^3$\textit{Munich Center for Machine Learning (MCML), München, Germany }\\
    $^4$\textit{Leiden University, Mathematical Institute, Einsteinweg 55, 2333 CC Leiden, Netherlands}}
}

\maketitle


\begin{abstract}
We propose a unified framework for dimension reduction of partial differential equations posed on thin domains.
Our approach combines three complementary ingredients: a careful boundary-condition analysis, an averaging-based splitting for general thin geometries, and a slow-manifold viewpoint for the resulting fast-slow system.
Homogeneous Neumann conditions on the thin faces emerge as the most relevant and physical regime because they preserve the transverse zero mode and therefore lead to a genuine lower-dimensional reduced equation. 
For general thin domains we derive the averaged fast-slow system and isolate the geometry-induced correction term.
We then formulate a splitting-based Lyapunov-Perron construction for an exact slow manifold when a suitable spectral decomposition of the slow variable is available, and we construct approximate slow manifolds and corrected reduced dynamics directly from the invariance equation by asymptotic expansion.
Moreover, we propose the Schnakenberg reaction-diffusion system as a canonical test problem for comparing the full thin-domain dynamics, the averaged model, and the manifold-corrected reduced dynamics.
Finally, we also extend the framework to thin tubular domains and derive the corresponding rescaled fast-slow system in
curved geometry.
\end{abstract}

\vspace{2mm}
	
	
	\vspace{1mm}
	
	\vspace{3mm}
\blfootnote{\textcolor{white}{.}\\[-2.5mm]
	\hspace{-5.4mm}\Letter \; ckuehn@ma.tum.de (Christian Kuehn)  \\[0.8mm]
	\Letter\; j.e.sulzbach@math.leidenuniv.nl (Jan-Eric Sulzbach)
	}

{\hypersetup{linkcolor=black}
    \tableofcontents}

\section{Introduction}

\textbf{Background:} Partial differential equations on thin domains arise when one or more spatial directions are strongly confined, so that the physical region is narrow compared to the characteristic tangential length scale. Typical examples include thin films and coatings in lubrication theory \cite{OronDavisBankoff1997}, channels, tubes, and waveguides \cite{Grieser2008}, and reaction-diffusion processes on growing, evolving, or strongly confined biological domains \cite{CrampinGaffneyMaini1999,li2017limiting,Miura2017,Miura2023,keller2026asymptotic}. In all these settings one expects that the geometry should enforce an effective lower-dimensional description, but the mechanism by which this reduction occurs depends delicately on the interaction between diffusion, geometry, and boundary conditions. From a mathematical point of view, thin-domain problems therefore provide a natural testing ground for dimension reduction, singular perturbation theory, homogenization, spectral analysis, and long-time dynamics.

A basic observation underlies the present work. After rescaling the thin variable to a fixed reference interval, the transverse diffusion appears with the singular prefactor $\varepsilon^{-2}$, or, in variable-thickness coordinates, with the leading coefficient $\varepsilon^{-2}g(x)^{-2}$. Hence the thin direction is associated with a fast relaxation mechanism, while, in the homogeneous Neumann or compatible-flux regimes, the cross-sectional average evolves on an $\mathcal O(1)$ scale. This suggests that thin-domain problems should be viewed as fast-slow systems: the averaged component is the candidate slow variable, whereas the mean-zero transverse fluctuations form the fast part. Making this heuristic rigorous, however, requires more than a formal averaging step. 

\begin{figure}[H]
\centering

\begin{tikzpicture}[x=1cm,y=1cm,>=Latex]
  \coordinate (B0) at (0,0);
  \coordinate (B1) at (0.9,0);
  \coordinate (B2) at (1.8,0);
  \coordinate (B3) at (2.7,0);
  \coordinate (B4) at (3.6,0);
  \coordinate (B5) at (4.5,0);
  \coordinate (B6) at (5.4,0);
  \coordinate (B7) at (6.3,0);
  \coordinate (B8) at (7.2,0);

  \coordinate (T0) at (0,1.34);
  \coordinate (T1) at (0.9,1.42);
  \coordinate (T2) at (1.8,1.57);
  \coordinate (T3) at (2.7,1.80);
  \coordinate (T4) at (3.6,1.60);
  \coordinate (T5) at (4.5,1.73);
  \coordinate (T6) at (5.4,2.04);
  \coordinate (T7) at (6.3,1.84);
  \coordinate (T8) at (7.2,1.93);

  \fill[blue!10]
    plot[smooth] coordinates {(B0) (B1) (B2) (B3) (B4) (B5) (B6) (B7) (B8)}
    --
    plot[smooth] coordinates {(T8) (T7) (T6) (T5) (T4) (T3) (T2) (T1) (T0)}
    -- cycle;

  \draw[thick]
    plot[smooth] coordinates {(B0) (B1) (B2) (B3) (B4) (B5) (B6) (B7) (B8)};

  \draw[thick]
    plot[smooth] coordinates {(T0) (T1) (T2) (T3) (T4) (T5) (T6) (T7) (T8)};

  \draw[thick] (B0) -- (T0);
  \draw[thick] (B8) -- (T8);

  \draw[->] (-0.35,0.0) -- (7.75,0.0) node[right] {$x$};
  \draw[->] (0.0,-0.1) -- (0.0,2.45) node[above] {$y$};

  \draw[<->] (5.55,0.16) -- (5.55,2.00) node[midway,right] {$\varepsilon g(x)$};

  \node at (3.6,-0.23) {$\Omega$};
  \node at (1.35,1.05) {$Q_\varepsilon$};
\end{tikzpicture}
\caption{A schematic thin domain $Q_\varepsilon=\{(x,y):x\in\Omega,\ 0<y<\varepsilon g(x)\}$ with a tangential variable $x$ on the base domain and a strongly confined transversal direction $y$. After the rescaling $Y=y/(\varepsilon g(x))$, the problem is posed on the fixed reference strip $\Omega\times(0,1)$; the transverse diffusion then carries the singular $\varepsilon^{-2}$ scaling, while the variable thickness enters the coefficients and transformed boundary conditions.}
\label{fig:thin-domain-sketch}
\end{figure}
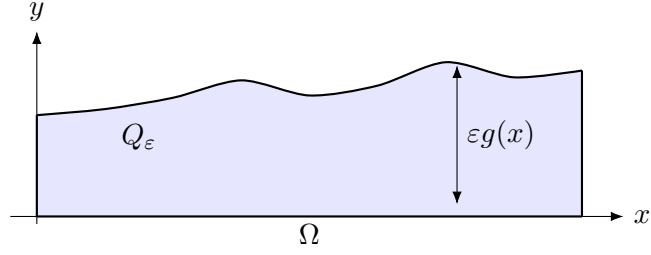

\textbf{Mathematical overview:}
To see why this thin domain problem is not merely solvable by standard averaging, it is useful to display the concrete model that will be analyzed throughout this work. This part of the introduction also serves as a high-level formal overview of the strategy we apply, and the reader can refer back to it as necessary. Let 
\[ Q_\varepsilon := \{(x,y)\in\mathbb R^2:\ x\in\Omega,\ 0<y<\varepsilon g(x)\} \]
be the thin domain, where \(\Omega=(0,L)\) and \(g>0\); see also \Cref{fig:thin-domain-sketch}.
We consider the reaction-diffusion equation
\[ \partial_t \widehat u = \Delta \widehat u + f(\widehat u) \quad \text{in } Q_\varepsilon, \qquad \partial_\nu \widehat u =0 \quad \text{on } \partial Q_\varepsilon . \]
After introducing a fixed transverse variable and rescaled domain
\[ Y=\frac{y}{\varepsilon g(x)}, \qquad Q=\Omega\times(0,1), \]
and writing \(u(x,Y,t)=\widehat u(x,\varepsilon g(x)Y,t)\), the evolution equation becomes 
\[ \partial_t u = \frac{1}{\varepsilon^2 g(x)^2}\partial_Y^2 u + \mc D^2u + f(u)\quad \text{in } Q, \quad \text{where }\quad\mc D:=\partial_x-Y\frac{g'(x)}{g(x)}\partial_Y . \]
The homogeneous Neumann condition on the physical thin domain is transformed into the conormal boundary conditions 
\[ \partial_Y u=0 \quad \text{on } Y=0, \qquad \partial_Y u-\varepsilon^2 gg'\mc Du=0 \quad \text{on } Y=1, \qquad \mc Du=0 \quad \text{on } \partial\Omega\times(0,1). \] 
Thus the singular transverse diffusion and the boundary condition are coupled at order \(\varepsilon^2\). 
This coupling is the source of the trace terms that appear in the averaged equation and is the reason why the analysis cannot be reduced to a purely formal projection onto the transverse zero mode. Next, let 
\[ Mu(x):=\int_0^1 u(x,Y)\,dY, \qquad v:=Mu, \qquad w:=(\mathrm{Id}-M)u . \]
Then \(v\) is the cross-sectional average and \(w\in\ker M\) is the mean-zero transverse fluctuation. Then one observes that the rescaled equation can be written in the fast-slow form 
\[ \partial_t v = B_g v+R_g w+Mf(v+w), \qquad B_g v:=\frac1g\partial_x(g\partial_xv), \]
and 
\[ \partial_t w = \frac{1}{\varepsilon^2g^2}\partial_Y^2 w + \mc D^2w - \mc E\left(\frac{g'}g\partial_xv+R_gw\right) + (\mathrm{Id}-M)f(v+w), \]
where 
\[ R_gw:=-\frac1g\partial_x\bigl(g'w(\cdot,1)\bigr), \qquad \mc E\phi(x,Y):=\phi(x). \]
Note that the boundary conditions for \(w\) are no longer homogeneous: 
\[ \partial_Yw=0 \quad \text{on } Y=0, \qquad \partial_Yw-\varepsilon^2gg'\mc Dw = \varepsilon^2gg'\partial_xv \quad \text{on } Y=1, \]
and 
\[ \mc Dw=-\partial_xv \quad \text{on } \partial\Omega\times(0,1). \]
Averaging the lateral condition yields the induced slow boundary compatibility condition 
\[ \partial_xv-\frac{g'}g w(\cdot,1)=0 \quad \text{on } \partial\Omega . \] 
This is now the point at which the boundary geometry, the singular limit, and the choice of the underlying function spaces interact. To tackle the problem geometrically and construct a slow-manifold we therefore do not work directly with the fast variable \(w\). Instead, we introduce a conormal lifting \(\Lambda_\varepsilon^1 v\) and set 
\[ z:=w-\Lambda_\varepsilon^1v . \]
The variable \(z\) satisfies homogeneous conormal boundary conditions and belongs to the domain of the projected fast operator. 
After a suitable spectral splitting in the slow variable
\[ v=v_F+v_S,\qquad Y=Y_F^\varepsilon\oplus Y_S^\varepsilon, \]
the lifted thin-domain system takes the abstract split form 
\[ \partial_t z^\varepsilon = \varepsilon^{-2}A_\varepsilon z^\varepsilon + F_\varepsilon(z^\varepsilon,v_F^\varepsilon+v_S^\varepsilon), \]
\[ \partial_t v_F^\varepsilon = \varepsilon^{-2}\widetilde B_{g,\varepsilon}v_F^\varepsilon + \Pi_F^\varepsilon G_\varepsilon(z^\varepsilon,v_F^\varepsilon+v_S^\varepsilon), \qquad \widetilde B_{g,\varepsilon}:=\varepsilon^2 B_g|_{Y_F^\varepsilon}, \] 
and 
\[ \partial_t v_S^\varepsilon = B_gv_S^\varepsilon + \Pi_S^\varepsilon G_\varepsilon(z^\varepsilon,v_F^\varepsilon+v_S^\varepsilon). \]
The fast operator \(A_\varepsilon\) is generated on \(\mc X=\ker M\) by the conormal form 
\[ \mathfrak a_\varepsilon(z,\phi) = \int_Q \frac1{g(x)^2}\partial_Yz\,\partial_Y\phi\,g(x)\,dx\,dY + \varepsilon^2 \int_Q \mc Dz\,\mc D\phi\,g(x)\,dx\,dY = \mathfrak a_0(z,\phi)+\varepsilon^2 \mathfrak a_1(z,\phi). \]
At the same time, the trace term \(R_gw\) forces the fast topology to control boundary traces; this motivates the trace-compatible space 
\[ \mc Z_\theta^\varepsilon := \mc X_\theta^\varepsilon\cap H^{2\theta}(Q)\cap \mc X, \qquad \theta>\frac34 . \]
Thus the central analytical issue is already visible at the level of the rescaled model: the transverse diffusion is singular, the transformed Neumann condition is conormal and \(\varepsilon\)-dependent, the averaged equation contains a boundary trace of the fast variable, and the slow-manifold formulation requires a lifted homogeneous fast variable in a topology strong enough to control that trace.\\

\textbf{Main contributions:} The aim of this paper is to develop a reduction framework organized around three steps. One must first determine which boundary conditions preserve a nontrivial slow variable, then derive the correct averaged equations for general thin geometries, and finally justify the reduced dynamics geometrically in an infinite-dimensional setting.

Our first contribution is a systematic analysis of the
boundary conditions on the thin faces. Among the standard local homogeneous
thin-face conditions considered here, homogeneous Neumann conditions form the
distinguished regime for genuine dimension reduction: they preserve the
transverse zero mode and thereby allow a nontrivial lower-dimensional variable
to persist in the thin-domain limit. By contrast, homogeneous Dirichlet
conditions eliminate this mode and enforce fast transverse decay, while general
inhomogeneous data naturally give rise to liftings, singular forcing, or
nontrivial transverse profiles.

Building on this observation, we derive an averaging-based splitting for
variable-thickness thin strips. The splitting separates the averaged dynamics
from the mean-zero transverse fluctuation and makes the geometry-induced trace
terms explicit. This reveals the fast-slow structure that underlies the
reduction. We then connect this structure with infinite-dimensional
slow-manifold theory. At the exact level, we formulate a Lyapunov-Perron
construction after a spectral splitting of the slow variable. This identifies
the main dynamical obstruction: the slow evolution has to be decomposed into a
part that can be continued backward and a high-mode part that decays forward.

For the thin-domain application, the exact construction requires a
trace-compatible fast topology and a uniform parameter-elliptic estimate for the
projected conormal fast operator. These estimates are verified, or explicitly
isolated as elliptic regularity inputs, in \Cref{subsec:thin-domain-application}
and \Cref{app:uniform-parameter-elliptic-estimate}. The asymptotic construction
in \Cref{sec:reduced} is developed first in an abstract split setting. For the
thin-domain system we then identify the corresponding lifted first-order
correction in the homogeneous fast variable. The formal expansion is upgraded to
a finite-order thin-domain approximation theorem in
\Cref{subsec:thin-domain-approximation-theorem}: using the form identity for the
fast operator, the conormal coefficient hierarchy, and a final boundary
correction, one obtains an \(O(\varepsilon^{2N})\) finite-time tracking estimate
for graph-prepared data. The additional finite-order trace liftings needed for
this construction are recorded in the appendix. Finally, we illustrate the
reduction mechanism with a heterogeneous Schnakenberg system.\\

\textbf{Embedding into the literature:} Research on thin domains is broad, and the present paper lies at the intersection of several strands. A foundational direction is the asymptotic analysis of elliptic boundary value problems in singularly perturbed and thin geometries, as developed in the work of Maz'ya, Nazarov, and Plamenevskii \cite{MazyaNazarovPlamenevskii2000,MazyaNazarovPlamenevskii2000b}. A second direction, closer to nonlinear dynamics, was shaped by Hale and Raugel, who derived limit equations for dissipative PDEs on thin domains, proved convergence properties for the associated attractors, and highlighted the decisive role of the transverse spectrum and of boundary conditions \cite{HaleRaugel1992,HaleRaugel1995,Raugel1995}. Related thin-domain singular limits for fluid equations were studied, for example, by Raugel and Sell and by Iftimie \cite{RaugelSell1993,Iftimie2001}. From the viewpoint of curved and squeezed domains, Prizzi, Rinaldi, and Rybakowski developed an influential dynamical-systems theory for parabolic equations on thin domains, including the characterization of limit phase spaces and the behavior of attractors under squeezing \cite{PrizziRybakowski1999,PrizziRybakowski2001,PrizziRybakowski2002}. In parallel, spectral and waveguide aspects of thin structures were clarified in the work surveyed by Grieser \cite{Grieser2008,Grieser2017}.

Thin-domain limits also fit naturally into the broader theory of domain perturbations for PDE, where one studies how spectra, semiflows, and attractors change under deformations of the underlying region. In the regular regime this leads to continuity questions for semilinear parabolic problems under smooth or Lipschitz perturbations of the domain \cite{Arrieta2004,PereiraPereira2007,LeeNguyen2023}. In more singular regimes, rapidly oscillating boundaries, concentrated terms near the boundary, small holes, and squeezing toward lower-dimensional sets can change the effective equation itself and therefore connect thin-domain analysis with homogenization and boundary-layer phenomena \cite{JimenezCasas2020,TavaresLimaLozada2024,AragaoArrietaBruschi2025}.

Recent work has pushed the thin-domain theory toward geometrically more complex and analytically more delicate settings. Arrieta and collaborators studied spectral convergence, nonlinear dynamics under domain perturbation, and thin domains with oscillatory boundaries, thereby linking thin-domain limits with homogenization and attractor continuity \cite{Arrieta2004,Arrieta2024,ArrietaNakasatoVillanueva2025,BarbosaVillanueva2024}. Miura analyzed moving and curved thin domains, where the limit problem lives on an evolving surface and geometric effects enter the reduced equation explicitly \cite{Miura2017,Miura2023}, with a particular application to the Navier-Stokes equations on thin domains in a series of papers \cite{miura2022navier,miura2021navier,miura2020navier}. 
These works make clear that thin-domain problems cannot be treated by a single universal argument: geometry, boundary conditions, and the choice of functional-analytic framework strongly influence both the limit equation and the quality of the approximation.

A natural language for this separation of scales is geometric singular perturbation theory.
In the finite-dimensional setting, Fenichel's theory \cite{fenichel1979geometric} shows that normally hyperbolic slow manifolds persist under perturbation and carry reduced dynamics.
Later works such as \cite{Jones1995} and \cite{kuehn2015multiple} helped establish this viewpoint as a standard framework for multiple-scale problems.
The thin-domain problem studied here is infinite-dimensional, so the classical ODE theory does not apply verbatim. We refer to \cite{hummel2022slow,kuehn2024infinite,kuehn2025fast,kuehn2026approximate} for the development of a generalized Fenichel theory in infinite-dimensional Banach spaces.
Nevertheless, the same organizing idea remains:
under suitable spectral and regularity hypotheses, rapid transverse relaxation should lead to an invariant or approximately invariant slow object on which the effective lower-dimensional dynamics evolve.

Against this background, the perspective adopted in this paper is to connect thin-domain reduction with recent slow-manifold theory for infinite-dimensional evolution equations.
Rather than interpreting the reduced model merely as an averaged limit, we view it as the effective dynamics induced on a slow manifold associated with the rapid transverse relaxation.
This viewpoint offers two advantages for the present work. 
First, it places the reduction procedure into a single conceptual framework: the boundary conditions determine whether a nontrivial slow variable survives, the averaging step reveals the underlying fast-slow structure, and the invariance equation yields corrected reduced dynamics. 
Second, it provides a systematic route to higher-order approximations, which becomes essential in situations where the leading averaged equation captures the dominant tangential behavior but does not recover the small transverse profile of the full solution.\\

\textbf{Structure of the paper:}
In \Cref{sec:boundary} we analyze the role of the boundary conditions on the thin faces and identify homogeneous Neumann data as the standard homogeneous regime leading to genuine
lower-dimensional reduction.
\Cref{sec:general-thin} derives the averaged fast-slow splitting for variable-thickness thin strips over a one-dimensional base and isolates the geometry-induced correction term.
\Cref{sec:manifolds} formulates the splitting-based Lyapunov-Perron construction and verifies the trace-compatible hypotheses for the thin-strip setting, under the stated elliptic regularity assumptions.
\Cref{sec:reduced} proves the asymptotic graph construction in an abstract split setting, derives the lifted first-order thin-domain correction, and proves the finite-order thin-domain approximation theorem in \Cref{subsec:thin-domain-approximation-theorem}.
\Cref{sec:example} applies the formal and numerical reduction procedure to a transversely heterogeneous Schnakenberg system.
\Cref{sec:tubular domains} derives the corresponding rescaled fast-slow system for thin tubular domains and records the further verification steps needed in curved geometry. 
\Cref{sec:discussion} concludes with a discussion of the scope of the framework and of several directions for future work.

\section{Boundary conditions on the thin faces}\label{sec:boundary}

In this section we study the role of the boundary conditions imposed on the thin part of the boundary. We deliberately emphasize the homogeneous Neumann case, because it is both the most natural physical regime for impermeable thin faces and the only standard homogeneous choice that leads to a genuine nontrivial lower-dimensional slow variable.

Let $\Om\subset \R^{n-1}$ be a bounded domain with sufficiently smooth boundary and consider the flat thin domain
\[
\Qe := \Om\times (-\eps,\eps).
\]
On $\Qe$ we study the semilinear parabolic problem with forcing term
\begin{equation}\label{eq:flat-physical}
\partial_t  u^\eps = \Delta  u^\eps + f( u^\eps)+\ell,
\qquad  u^\eps(0)= u_0^\eps,
\end{equation}
with homogeneous Neumann conditions on the lateral boundary $\partial\Om\times(-\eps,\eps)$ and varying conditions on the thin faces
\[
\Gamma_\eps^{\pm}:=\Om\times\{\pm\eps\}.
\]
Here, $f$ is a nonlinear reaction term depending on $u$ and $\ell$ a forcing term. The exact regularity classes are detailed below.

After the rescaling $Y=y/\eps\in(-1,1)$ the equation becomes
\begin{equation}\label{eq:flat-rescaled}
\partial_t u^\eps = \Delta_x u^\eps + \frac{1}{\eps^2}\partial_Y^2u^\eps + f(u^\eps)+\ell
\qquad \text{in }\Q:=\Om\times(-1,1).
\end{equation}
The behavior as $\eps\to 0$ is therefore governed by the transverse operator $\partial_Y^2$ together with the boundary conditions at $Y=\pm1$.

\subsection{The transverse operators}

We first isolate the one-dimensional operators on $I:=(-1,1)$:
\begin{align*}
A_N\phi &:= \phi'', & D(A_N)&:=\set{\phi\in H^2(I): \phi'(-1)=\phi'(1)=0},\\
A_D\phi &:= \phi'', & D(A_D)&:=\set{\phi\in H^2(I): \phi(-1)=\phi(1)=0}.
\end{align*}
The corresponding rescaled operators on $\Q$ are
\[
A_{\mathrm{bc}}^\eps := \Delta_x + \frac{1}{\eps^2}A_{\mathrm{bc}},\quad \text{with } bc=\{D,N\},
\]
with homogeneous Neumann conditions on the lateral boundary.

\begin{lemma}[Transverse spectra]\label{lem:transverse-spectra}
The spectra of the transverse operators are given by
\[
\sigma(-A_N)=\set{0}\cup \set{\bigl(\tfrac{k\pi}{2}\bigr)^2:k\in\N},
\qquad
\sigma(-A_D)=\set{\bigl(\tfrac{k\pi}{2}\bigr)^2:k\in\N}.
\]
In particular,
\[
\ker A_N = \mathrm{span}\set{1},
\qquad
\ker A_D = \set{0},
\qquad
-A_D\ge \frac{\pi^2}{4}
\]
in the sense of quadratic forms.
\end{lemma}

\begin{proof}
This is the standard Sturm-Liouville computation on $(-1,1)$. The homogeneous Neumann problem admits the constant eigenfunction for eigenvalue $0$, whereas the homogeneous Dirichlet problem has first eigenvalue $\pi^2/4$.
\end{proof}

\begin{remark}
\Cref{lem:transverse-spectra} contains the main idea of the section. Among the standard local homogeneous face conditions considered here, the homogeneous Neumann problem is the distinguished case because it preserves a transverse zero mode. It is therefore the case in which averaging in the thin direction produces a nontrivial slow variable. Homogeneous Dirichlet data remove this mode and force the whole solution onto the fast time scale.
\end{remark}

\subsection{The zero Neumann case}

We now impose
\begin{equation}\label{eq:zero-neumann-faces}
\partial_Yu^\eps(x,\pm1,t)=0.
\end{equation}
Define the averaging operator
\begin{equation}\label{eq:M-def}
Mu(x):=\frac12\int_{-1}^1 u(x,Y)\,dY.
\end{equation}
Set
\[
v:=Mu,
\qquad
w:=(\Id-M)u.
\]
Then $v$ is independent of $Y$ and $Mw=0$.

\begin{proposition}[Averaging for zero Neumann data]\label{prop:neumann-averaging}
Let $u\in H^2(\Q)$ satisfy \eqref{eq:zero-neumann-faces}. Then
\[
M\partial_Y^2u = 0.
\]
Consequently, under \eqref{eq:zero-neumann-faces} the decomposition $u=v+w$ is precisely the decomposition into the kernel part and the mean-zero part of the transverse Neumann operator.
\end{proposition}

\begin{proof}
Integrating by parts in $Y$ and using \eqref{eq:zero-neumann-faces},
\[
M\partial_Y^2u = \frac12\int_{-1}^1 \partial_Y^2u\,dY
= \frac12\bigl[\partial_Yu(x,Y)\bigr]_{Y=-1}^{Y=1}=0.
\]
The remaining statements follow immediately from the definition of $M$.
\end{proof}

Let $u^\varepsilon\in W^{1,2}(0,T;L^2(Q)) \cap L^2(0,T;D(A^\varepsilon_N))$ be a strong solution of the rescaled equation \eqref{eq:flat-rescaled}. 
Since \(M\) is a bounded, time-independent linear operator from \(L^2(Q)\) to \(L^2(\Omega)\) we can directly apply it to \eqref{eq:flat-rescaled}.
Thus we obtain
\[
  M(\partial_t u^\varepsilon)=\partial_t(Mu^\varepsilon),
\]
and, using \Cref{prop:neumann-averaging},
\begin{equation}\label{eq:avg-flat}
  \partial_t(Mu^\varepsilon)   =   \Delta_x(Mu^\varepsilon)   +M f(u^\varepsilon)   +M\ell .
\end{equation}
For weak or mild solutions, equation \eqref{eq:avg-flat} is understood in the
distributional sense in time. 
More precisely, since \(M\) is bounded and does not depend on \(t\), one has
\[
  \partial_t(Mu^\varepsilon)=M(\partial_tu^\varepsilon)
  \quad\text{in } \mathcal D'(0,T;H^{-1}(\Omega)).
\]
Equivalently, in the weak formulation one may choose test functions of the form \(\psi(x,Y,t)=\varphi(x,t)\), independent of \(Y\). 
This choice projects the full equation onto the transverse zero mode and gives precisely
\eqref{eq:avg-flat}.
Note that the term \(\varepsilon^{-2}\partial_Y^2u^\varepsilon\) contributes only the Neumann
boundary traces, which vanish under \eqref{eq:zero-neumann-faces}.

This is the fundamental mechanism behind the reduction theory developed later. The fast transverse diffusion does not affect the cross-sectional mean directly; it only forces the mean-zero component $w$ to relax rapidly.

\subsection{The zero Dirichlet case}

We next impose
\begin{equation}\label{eq:zero-dirichlet-faces}
u^\eps(x,\pm1,t)=0.
\end{equation}
In this regime there is no transverse kernel.
For bounded forcing, the full solution relaxes on the fast time scale to an $\mathcal O(\eps^2)$ remainder after the initial layer.

\begin{theorem}[Fast decay under zero Dirichlet data]\label{thm:dirichlet-decay}
Assume that $f$ is globally Lipschitz with Lipschitz constant $L_f$ and that
\[
\ell\in L^\infty(0,T;L^2(\Q)).
\]
Let $u^\eps$ be the mild solution of \eqref{eq:flat-rescaled} subject to \eqref{eq:zero-dirichlet-faces}. Then there exist $\eps_0>0$ and $C>0$ such that for all $\eps\in(0,\eps_0]$ and all $t\in[0,T]$,
\begin{equation}\label{eq:dirichlet-est}
\norm{u^\eps(t)}_{L^2(\Q)}
\le e^{-\frac{\pi^2}{8\eps^2}t}\norm{u_0^\eps}_{L^2(\Q)} + C\eps^2.
\end{equation}
In particular, if
\[
\sup_{0<\eps\le \eps_0}\norm{u_0^\eps}_{L^2(\Q)}<\infty,
\]
then, for every $\tau\in(0,T]$,
\begin{equation}\label{eq:dirichlet-conv}
\sup_{t\in[\tau,T]}\norm{u^\eps(t)}_{L^2(\Q)}\longrightarrow 0
\qquad \text{as }\eps\to0.
\end{equation}
\end{theorem}

\begin{proof}
Let $A_D^\eps$ denote the realization of $\Delta_x+\eps^{-2}\partial_Y^2$ with homogeneous Neumann conditions on the lateral boundary and homogeneous Dirichlet conditions on $Y=\pm1$. By \Cref{lem:transverse-spectra}, the lowest point of the spectrum of $-A_D^\eps$ is at least $\pi^2/(4\eps^2)$. Hence
\[
\norm{e^{tA_D^\eps}}_{\mc L(L^2(\Q))}
\le e^{-\mu_\eps t},
\qquad
\mu_\eps:=\frac{\pi^2}{4\eps^2}.
\]
The variation-of-constants formula gives
\[
u^\eps(t)
=
e^{tA_D^\eps}u_0^\eps
+
\int_0^t e^{(t-s)A_D^\eps}\bigl(f(u^\eps(s))+\ell(s)\bigr)\,ds.
\]
Since $f$ is globally Lipschitz,
\[
\norm{f(u^\eps(s))}_{L^2(\Q)}
\le L_f\norm{u^\eps(s)}_{L^2(\Q)}
+ |f(0)|\,|\Q|^{1/2}.
\]
Thus, with
\[
C_0:=|f(0)|\,|\Q|^{1/2}+\norm{\ell}_{L^\infty(0,T;L^2(\Q))},
\]
we obtain
\[
\norm{u^\eps(t)}
\le
e^{-\mu_\eps t}\norm{u_0^\eps}
+
\int_0^t e^{-\mu_\eps(t-s)}
\bigl(L_f\norm{u^\eps(s)}+C_0\bigr)\,ds.
\]
The convolution form of Gronwall's inequality yields
\[
\norm{u^\eps(t)}
\le
e^{-(\mu_\eps-L_f)t}\norm{u_0^\eps}
+
\frac{C_0}{\mu_\eps-L_f}
\]
whenever $\mu_\eps>L_f$. Choosing $\eps_0>0$ so small that
\[
\mu_\eps-L_f\ge \frac{\pi^2}{8\eps^2}
\qquad\text{for }0<\eps\le\eps_0
\]
gives \eqref{eq:dirichlet-est}. The convergence statement follows from this estimate and the assumed uniform boundedness of the initial data.
\end{proof}

\begin{remark}
Note that under the Dirichlet conditions one generally expects an initial layer of thickness $\mc O(\eps^2)$ in time. 
Thus, the conclusion of \Cref{thm:dirichlet-decay} is not that the mean survives in a modified form, but rather that the transverse constraint removes the slow mode altogether.
\end{remark}

\subsection{General boundary data}\label{subsec:general-bc}

In this final subsection, we outline how to treat general boundary data.

\paragraph{Inhomogeneous Neumann data.}
Assume on the physical thin domain that
\[
\partial_\nu  u^\eps\big|_{y=-\eps}=q_-(x,t),
\qquad
\partial_\nu  u^\eps\big|_{y=\eps}=q_+(x,t).
\]
After rescaling, these become
\[
\partial_Yu^\eps(x,-1,t)=-\eps q_-(x,t),
\qquad
\partial_Yu^\eps(x,1,t)=\eps q_+(x,t).
\]
A quadratic lifting
\[
\chi^\eps(x,Y,t)=\frac\eps4(Y+1)^2q_+(x,t)+\frac\eps4(Y-1)^2q_-(x,t)
\]
realizes these traces exactly.
We now set
\begin{equation*}
w^\varepsilon:=u^\varepsilon-\chi^\varepsilon.
\end{equation*}
Then $w^\varepsilon$ satisfies homogeneous Neumann conditions on $Y=\pm1$, and substitution back into \eqref{eq:flat-rescaled} yields
\begin{equation}\label{eq:remainder-neumann-raw}
\partial_t w^\varepsilon
=
\Delta_x w^\varepsilon+\frac1{\varepsilon^2}\partial_Y^2w^\varepsilon
+f(w^\varepsilon+\chi^\varepsilon)+\ell-\partial_t \chi^\varepsilon+\Delta_x \chi^\varepsilon
+\frac{q_++q_-}{2\varepsilon}.
\end{equation}
The singular term $\frac{q_++q_-}{2\varepsilon}$ is the key new feature of the inhomogeneous Neumann problem.
With this setup in place, averaging across the thin direction yields the effective equation for the cross-sectional mean.
The averaged equation now acquires the forcing term
\begin{equation}\label{eq:neumann-general-avg}
\partial_t(Mu^\eps)=\Delta_x(Mu^\eps)+M f(u^\eps)+M\ell+\frac{q_+(x,t)+q_-(x,t)}{2\eps}.
\end{equation}
Thus the sum of the fluxes on the two thin faces enters the averaged equation as a forcing term of size $1/\varepsilon$.
A finite limit therefore requires the compatibility scaling
\[
q_+(x,t)+q_-(x,t)=\mc O(\eps).
\]
If the combined flux is larger than order $\eps$, one should expect singular forcing, transition layers, or look for a different renormalization of the solution. The mechanism becomes transparent in the following explicit elliptic example on a thin rectangle.
\begin{example}[Singular scaling from inhomogeneous Neumann flux]
    Let
\[
Q_\varepsilon=(0,1)\times(0,\varepsilon).
\]
Consider the linear elliptic problem
\begin{equation}\label{eq:ex-1-linear}
\begin{cases}
\partial_x^2u+\partial_y^2u-u=0, & \text{in }Q_\varepsilon,\\[0.3em]
\partial_yu\big|_{y=0}=0,\qquad \partial_yu\big|_{y=\varepsilon}=\dfrac{c}{\varepsilon^\alpha}, &\\[0.7em]
\partial_xu\big|_{x=0}=\partial_xu\big|_{x=1}=0,&
\end{cases}
\end{equation}
where $c>0$ and $\alpha\in\mathbb R$ are parameters.

\begin{proposition}[Explicit solution and scaling thresholds]\label{prop:explicit-neumann-example}
The unique solution of \eqref{eq:ex-1-linear} is independent of $x$ and is given by
\begin{equation}\label{eq:explicit-neumann-solution}
u^\varepsilon(x,y)=\frac{c}{\varepsilon^\alpha}\frac{\cosh(y)}{\sinh(\varepsilon)}.
\end{equation}
Moreover,
\begin{equation}\label{eq:explicit-neumann-l2-asymptotic}
\|u^\varepsilon\|_{L^2(Q_\varepsilon)}\simeq \varepsilon^{-\alpha-\frac12}
\qquad\text{as }\varepsilon\to0,
\end{equation}
and, after rescaling $y=\varepsilon Y$,
\begin{equation}\label{eq:explicit-neumann-rescaled-solution}
U^\varepsilon(x,Y):=u^\varepsilon(x,\varepsilon Y)
=
\frac{c}{\varepsilon^\alpha}\frac{\cosh(\varepsilon Y)}{\sinh(\varepsilon)}
\sim c\,\varepsilon^{-\alpha-1}
\qquad\text{uniformly on }(0,1)^2.
\end{equation}
In particular, the rescaled solution exhibits the following trichotomy:
\begin{equation}\label{eq:trichotomy}
U^\varepsilon\to
\begin{cases}
0, & \alpha<-1,\\
c, & \alpha=-1,\\
+\infty, & \alpha>-1,
\end{cases}
\qquad\text{uniformly on }(0,1)^2.
\end{equation}
\end{proposition}


\begin{remark}[Physical and rescaled thresholds]\label{rem:interpret-example-l2}
Proposition~\ref{prop:explicit-neumann-example} displays two different thresholds, corresponding to two different ways of measuring the solution. In the physical norm one has
\[
\|u^\varepsilon\|_{L^2(Q_\varepsilon)}\simeq \varepsilon^{-\alpha-1/2},
\]
so the threshold for blow-up in $L^2(Q_\varepsilon)$ is $\alpha=-\tfrac12$. By contrast, on the fixed rescaled domain the relevant amplitude is $\varepsilon^{-\alpha-1}$, so the threshold for an $O(1)$ limit is $\alpha=-1$.
\end{remark}

\begin{remark}[Consistency with the averaged equation]\label{rem:interpret-example-average}
The example is fully consistent with the averaged equation. Indeed, if one averages \eqref{eq:ex-1-linear} in the thin direction, then the $x$-independence and the boundary conditions give
\[
-\bar u^\varepsilon + \frac1\varepsilon\bigl(\partial_yu^\varepsilon(\varepsilon)-\partial_yu^\varepsilon(0)\bigr)=0,
\]
that is,
\[
\bar u^\varepsilon = \frac{c}{\varepsilon^{\alpha+1}},
\]
where
\[
\bar u^\varepsilon(x):=\frac1\varepsilon\int_0^\varepsilon u^\varepsilon(x,y)\,dy.
\]
Thus the average itself already scales like $\varepsilon^{-\alpha-1}$, exactly as indicated by \eqref{eq:neumann-general-avg}. The three regimes $\alpha<-1$, $\alpha=-1$, and $\alpha>-1$ therefore correspond, respectively, to vanishing forcing, finite forcing, and singular forcing in the averaged equation.
\end{remark}

\end{example}

\paragraph{Inhomogeneous Dirichlet data.}
Assume instead that
\[
u^\eps(x,-1,t)=d_-(x,t),
\qquad
u^\eps(x,1,t)=d_+(x,t).
\]
The natural affine lifting
\[
b(x,Y,t)=\frac{1-Y}{2}d_-(x,t)+\frac{1+Y}{2}d_+(x,t)
\]
removes the boundary traces and leaves a homogeneous Dirichlet problem for \(w^\eps:=u^\eps-b\).
The shifted equation for $w^\eps$ contains the additional forcing terms generated by $b$, namely time derivatives, tangential derivatives, and the shifted nonlinearity. If $d_\pm$ are sufficiently regular and these forcing terms remain bounded in the relevant norms, then the same fast-decay estimate as in \Cref{thm:dirichlet-decay} shows that $w^\eps$ is of order $\mc O(\eps^2)$ after the initial layer. Thus the leading profile is the lifting $b$ itself.

In general this profile depends on $Y$, and therefore the limit is not genuinely lower-dimensional unless $d_+=d_-$. Indeed,
\[
Mb=\frac{d_-+d_+}{2},
\qquad
b-Mb=\frac{Y}{2}(d_+-d_-).
\]
Hence any order-one mismatch in the two traces survives in the thin direction. In physical variables the corresponding affine profile has transverse derivative of size
\[
\partial_y b_{\mathrm{phys}}
=
\frac{d_+-d_-}{2\eps},
\]
and therefore
\[
\|\partial_y b_{\mathrm{phys}}\|_{L^2(\Omega\times(-\eps,\eps))}
=
\frac{1}{\sqrt{2\eps}}\|d_+-d_-\|_{L^2(\Omega)}.
\]
Thus general inhomogeneous Dirichlet data naturally produce large transverse gradients, rather than a purely lower-dimensional limit.

\section{Variable-thickness thin strips and the averaging process}\label{sec:general-thin}

We now pass from flat thin domains to two-dimensional thin strips with
variable thickness over a one-dimensional base. 
This is the setting in which the rigorous slow-manifold verification and approximation in \Cref{subsec:thin-domain-application} and \Cref{subsec:thin-domain-asymptotic-expansion} are carried out.
This introduces two new features. 
First, after rescaling the thin direction to a fixed reference interval, the physical homogeneous Neumann condition is no longer a pure coordinate Neumann condition on the entire boundary of the fixed strip.
Instead, the upper boundary condition becomes a geometry-dependent conormal condition.
Second, because of this transformed boundary condition, the singular transverse diffusion term and the transformed tangential operator interact in the averaged equation.

The purpose of this section is therefore twofold. 
In the first step we derive the equation and boundary conditions on the fixed strip. 
In the second step we use the cross-sectional average to split the solution into an averaged component and
a mean-zero fluctuation.
This yields a fast-slow system in which the averaged variable evolves on the base domain and the mean-zero component is driven by the strong transverse diffusion.
The variable thickness enters the reduced equation through a geometry-induced correction term.

Let $\Om\subset\R$ be a bounded interval and let $g\in C^2(\overline\Om)$ satisfy $g(x)>0$. Consider the thin domain
\[
\Qe:=\set{(x,y)\in\R^2: x\in\Om,\ 0<y<\eps g(x)}.
\]
On $\Qe$ we study the reaction-diffusion equation
\begin{equation}\label{eq:general-physical}
\partial_t \hat u = \Delta \hat u + f(\hat u)
\qquad\text{in }\Qe,
\end{equation}
with homogeneous Neumann boundary conditions on $\partial \Qe$.
Next, we introduce the rescaled thin variable
\[
Y:=\frac{y}{\eps g(x)}\in(0,1),
\]
and the rescaled domain
\[
\Q:=\Om\times(0,1).
\]

\begin{proposition}[Rescaled equation and conormal boundary conditions]\label{prop: rescaled thin domain}
 Let $\widehat u$ be a sufficiently smooth solution of \eqref{eq:general-physical}, and define
\[
u(x,Y,t):=\widehat u(x,\varepsilon g(x)Y,t).
\]
Set
\[
\mathcal D:=\partial_x-b(x,Y)\partial_Y,
\qquad
b(x,Y):=Y\frac{g'(x)}{g(x)} .
\]
Here and below \(\mc D^2\) denotes the composition \(\mc D\circ \mc D\).
Then, the transformed equation is given by
\begin{equation}\label{eq:rescaled-operator-general}
\partial_t u = \frac{1}{\eps^2 g(x)^2}\partial_Y^2u + \mathcal D^2 u + f(u)
\qquad\text{in }\Q.
\end{equation}
The pulled-back homogeneous Neumann conditions are
\begin{align*}
\partial_Yu=0
\qquad\text{on } \Omega \times \{0\},\\
\partial_Yu-\varepsilon^2g(x)g'(x)\mathcal D u=0
\qquad\text{on } \Omega\times \{1\},
\intertext{and}
\mathcal D u=0
\qquad\text{on }\partial\Omega\times(0,1).
\end{align*}  
\end{proposition}

\begin{remark}
    Note that the physical homogeneous Neumann condition on \(\partial Q_\varepsilon\) does not transform into the coordinate condition \(\partial_Y u=0\) on the whole boundary of the fixed strip.
\end{remark}

\begin{proof}
We first record the transformed boundary conditions. Since
\[
u(x,Y)=\widehat u(x,\varepsilon g(x)Y),
\]
the chain rule gives
\[
\partial_Yu=\varepsilon g\,\partial_y\widehat u,
\]
and
\[
\partial_xu
=
\partial_x\widehat u+\varepsilon g'Y\partial_y\widehat u
=
\partial_x\widehat u+b\partial_Yu.
\]
Hence
\[
\partial_x\widehat u
=
(\partial_x-b\partial_Y)u
=
\mathcal Du.
\]
Moreover,
\[
\partial_y\widehat u
=
\frac1{\varepsilon g}\partial_Yu.
\]
Thus the physical Laplacian becomes
\[
\Delta\widehat u
=
\partial_x^2\widehat u+\partial_y^2\widehat u
=
\mathcal D^2u+\frac1{\varepsilon^2g^2}\partial_Y^2u.
\]
This gives the transformed equation
\[
\partial_tu
=
\frac1{\varepsilon^2g^2}\partial_Y^2u+\mathcal D^2u+f(u).
\]

The boundary conditions transform as follows. On \(Y=0\), the physical outward normal is vertical,
so \(\partial_\nu\widehat u=0\) gives
\[
\partial_Yu=0.
\]
On the upper boundary \(y=\varepsilon g(x)\), an outward normal is
\[
(-\varepsilon g'(x),1).
\]
Therefore
\[
0
=
-\varepsilon g'\partial_x\widehat u+\partial_y\widehat u
=
-\varepsilon g'\mathcal Du+\frac1{\varepsilon g}\partial_Yu.
\]
Multiplying by \(\varepsilon g\) gives
\[
\partial_Yu-\varepsilon^2gg'\mathcal Du=0
\qquad\text{on }Y=1.
\]
Finally, on the lateral sides \(x\in\partial\Omega\), the physical normal is horizontal, so
\[
\partial_x\widehat u=0.
\]
Since \(\partial_x\widehat u=\mathcal Du\), this becomes
\[
\mathcal Du=0
\qquad\text{on }\partial\Omega\times(0,1).
\] 
\end{proof}

Thus, at the level of smooth solutions, the rescaled equation still contains the singular transverse operator \(\varepsilon^{-2}g^{-2}\partial_Y^2\), but the transformed Neumann condition on the upper boundary couples this operator to the tangential derivative \(\mc D\).
Consequently, the averaged equation is no longer obtained by simply discarding the transverse diffusion term.
The boundary contribution from the singular operator has to be combined with the averaged tangential operator.
This cancellation is the key point in the formal fast-slow splitting below.

We now define the cross-sectional average on the fixed strip by
\[
Mu(x):=\int_0^1 u(x,Y)\,dY .
\]
This is exactly the physical fiber average over the thin cross-section, since
\[
\frac1{\varepsilon g(x)}\int_0^{\varepsilon g(x)}
\widehat u(x,y)\,dy
=
\int_0^1 u(x,Y)\,dY .
\]
We then write
\[
v:=Mu,\qquad w:=(\Id-M)u,
\]
so that \(v\) is independent of \(Y\) and \(Mw=0\).

\begin{theorem}[Fast-slow splitting on a variable-thickness thin strip]
\label{prop:general-splitting}
Let \(u\) be a sufficiently smooth solution of \eqref{eq:rescaled-operator-general} satisfying the pulled-back boundary conditions from \Cref{prop: rescaled thin domain}. Define
\[
v:=Mu,
\qquad
w:=(\Id-M)u,
\qquad
u=v+w.
\]
Then \(v\) is independent of \(Y\), \(Mw=0\), and all traces below are understood in the classical sense.
Set
\begin{align}
     B_g v&:=\frac1g\partial_x(g\partial_xv),\\
     \intertext{and} \label{eq: boundary term Rg}
     R_gw&:=-\frac1g\partial_x\bigl(g' w(\cdot,1)\bigr).
\end{align}
Then the rescaled problem can be written as
\begin{align}
\partial_t v
&=
 B_gv+R_gw+Mf(v+w),
\label{eq:slow-general}
\\
\partial_t w
&=
\frac1{\varepsilon^2g(x)^2}\partial_Y^2w
+\mathcal D^2w
-\mathcal E\bigl(\frac{g'(x)}{g(x)}\,\partial_xv+R_gw\bigr)
+(\Id-M)f(v+w),
\label{eq:fast-general}
\end{align}
where \(\mathcal E\phi(x,Y):=\phi(x)\).
The boundary conditions for the split variables are
\begin{align*}
\partial_Yw=0
\qquad\text{on } \Omega \times \{0\},\\
\partial_Yw-\varepsilon^2gg'\mathcal Dw
=
\varepsilon^2gg'\partial_xv
\qquad\text{on } \Omega\times \{1\},
\intertext{and}
\mathcal Dw=-\partial_xv
\qquad\text{on }\partial\Omega\times(0,1).
\end{align*}
In addition, averaging the lateral boundary condition gives the compatibility condition
\[
\partial_x v-\frac{g'}g w(\cdot,1)=0
\qquad\text{on }\partial\Omega .
\]
Moreover, $B_g=\partial_x^2$ and $R_g=0$ whenever \(g\) is constant.
\end{theorem}

\begin{proof}
We start with the averaged equation. 
The main point is that, unlike in the flat coordinate-Neumann case, the average of the singular transverse term produces a boundary contribution. Indeed,
\[
M\left(\frac1{\varepsilon^2g^2}\partial_Y^2u\right)
=
\frac1{\varepsilon^2g^2}
\int_0^1\partial_Y^2u\,dY
=
\frac1{\varepsilon^2g^2}
\bigl[\partial_Yu\bigr]_{Y=0}^{Y=1}.
\]
Using the transformed Neumann conditions,
\[
\partial_Yu|_{Y=0}=0,
\qquad
\partial_Yu|_{Y=1}
=
\varepsilon^2gg'\mathcal Du|_{Y=1},
\]
we obtain
\[
M\left(\frac1{\varepsilon^2g^2}\partial_Y^2u\right)
=
\frac{g'}g\,\mathcal Du|_{Y=1}.
\]
Next we compute \(M\mathcal D^2u\). For any sufficiently smooth \(\phi\), one has
\[
M(\mathcal D\phi)
=
\int_0^1(\partial_x\phi-Y \frac{g'}g\partial_Y\phi)\,dY.
\]
Since \(\frac{g'}g\) is independent of \(Y\),
\[
M(\mathcal D\phi)
=
\partial_xM\phi
-
\frac{g'}g\int_0^1Y\partial_Y\phi\,dY.
\]
Integration by parts in \(Y\) gives
\[
\int_0^1Y\partial_Y\phi\,dY
=
\bigl[Y\phi\bigr]_{0}^{1}
-
\int_0^1\phi\,dY
=
\phi(\cdot,1)-M\phi.
\]
Therefore
\[
M(\mathcal D\phi)
=
\partial_x M\phi
-
\frac{g'}g\phi(\cdot,1)
+
\frac{g'}g M\phi.
\]
Applying this first with \(\phi=u=v+w\), and using \(Mu=v\), \(Mw=0\), gives
\[
M(\mathcal Du)
=
\partial_xv-\frac{g'}g(v+w(\cdot,1))+\frac{g'}g v
=
\partial_xv-\frac{g'}g w(\cdot,1).
\]
Set
\[
P:=M(\mathcal Du)=\partial_xv-\frac{g'}g w(\cdot,1).
\]
Applying the same identity with \(\phi=\mathcal Du\), we get
\[
M(\mathcal D^2u)
=
\partial_xP
-
\frac{g'}g(\mathcal Du)(\cdot,1)
+
\frac{g'}gP.
\]
Substituting \(P=\partial_xv-\frac{g'}gw(\cdot,1)\) yields
\[
\begin{aligned}
M(\mathcal D^2u)
&=
\partial_x\bigl(\partial_xv-\frac{g'}gw(\cdot,1)\bigr)
-
\frac{g'}g(\mathcal Du)(\cdot,1)
+
\frac{g'}g\bigl(\partial_xv-\frac{g'}g w(\cdot,1)\bigr)
\\
&=
\partial_x^2v
+
\frac{g'}g\partial_xv
-
\partial_x\Bigl(\frac{g'}gw(\cdot,1)\Bigr)
-
\Bigl(\frac{g'}g\Bigr)^2w(\cdot,1)
-
\frac{g'}g(\mathcal Du)(\cdot,1).
\end{aligned}
\]
Adding the averaged singular term yields
\[
\begin{aligned}
M\left(
\frac1{\varepsilon^2g^2}\partial_Y^2u+\mathcal D^2u
\right)
&=
\frac{g'}g(\mathcal Du)(\cdot,1)
+
M(\mathcal D^2u)
\\
&=
\partial_x^2v
+
\frac{g'}g\partial_xv
-
\partial_x\Bigl(\frac{g'}g w(\cdot,1)\Bigr)
-
\Bigl( \frac{g'}g\Bigr)^2w(\cdot,1).
\end{aligned}
\]
The boundary contribution \(\frac{g'}g(\mathcal Du)(\cdot,1)\) cancels exactly. Thus, defining
\[
 B_gv:=\frac1g\partial_x(g\partial_xv),
\]
and
\[
R_gw
:=
-\partial_x\Bigl( \frac{g'}g w(\cdot,1)\Bigr)-\Bigl(\frac{g'}g\Bigr)^2w(\cdot,1)= -\frac1g\partial_x\bigl(g'w(\cdot,1)\bigr),
\]
we have
\[
M\left(
\frac1{\varepsilon^2g^2}\partial_Y^2u+\mathcal D^2u
\right)
=
B_gv+R_gw.
\]
Applying \(M\) to the equation therefore gives
\[
\partial_tv
=
 B_gv+R_gw+Mf(v+w).
\]

It remains to derive the equation for \(w\). Since \(w=u-v\), we have
\[
\partial_tw=\partial_tu-\partial_tv.
\]
Using
\[
\mathcal Dv=\partial_xv,
\qquad
\mathcal D^2v=\partial_x^2v,
\quad \text{and}\quad 
\partial_Yv=0,
\]
we obtain
\[
\begin{aligned}
\partial_tw
&=
\frac1{\varepsilon^2g^2}\partial_Y^2w
+
\mathcal D^2w
+
\partial_x^2v
+
f(v+w)
-
\left(
 B_gv+R_gw+Mf(v+w)
\right)
\\
&=
\frac1{\varepsilon^2g^2}\partial_Y^2w
+
\mathcal D^2w
-
\frac{g'}g\partial_xv
-
R_gw
+
(\Id-M)f(v+w).
\end{aligned}
\]
Viewing the \(x\)-dependent terms as functions on \(Q\) through the embedding
\[
(\mathcal E\phi)(x,Y):=\phi(x),
\]
this becomes
\[
\partial_tw
=
\frac1{\varepsilon^2g^2}\partial_Y^2w
+
\mathcal D^2w
-
\mathcal E\bigl(\frac{g'}g\partial_xv+R_gw\bigr)
+
(\Id-M)f(v+w).
\]

Finally, the boundary conditions for \(v\) and \(w\) follow by substituting \(u=v+w\) into the
transformed Neumann conditions. Since \(v\) is independent of \(Y\),
\[
\partial_Yu=\partial_Yw,
\qquad
\mathcal Du=\partial_xv+\mathcal Dw.
\]
Thus the lower boundary condition gives
\[
\partial_Yw=0
\qquad\text{on }Y=0.
\]
The upper boundary condition gives
\[
\partial_Yw-\varepsilon^2gg'(\partial_xv+\mathcal Dw)=0,
\]
or equivalently
\[
\partial_Yw-\varepsilon^2gg'\mathcal Dw
=
\varepsilon^2gg'\partial_xv
\qquad\text{on }Y=1.
\]
The lateral condition gives
\[
\partial_xv+\mathcal Dw=0
\qquad\text{on }\partial\Omega\times(0,1),
\]
that is,
\[
\mathcal Dw=-\partial_xv
\qquad\text{on }\partial\Omega\times(0,1).
\]
Averaging this last condition in \(Y\), and using
\[
M(\mathcal Dw)
=
-\frac{g'}gw(\cdot,1),
\]
gives the induced slow boundary condition
\[
\partial_xv- \frac{g'}gw(\cdot,1)=0
\qquad\text{on }\partial\Omega.
\]

If \(g\) is constant, then \(g'=0\), \(b=0\), and \(R_g=0\). Hence
\[
 B_g=\partial_x^2,
\]
and the system reduces to the standard flat-domain fast-slow splitting. This completes the proof.
\end{proof}

\begin{remark}
System \eqref{eq:slow-general}-\eqref{eq:fast-general} identifies the structural mechanism by which variable-thickness thin domains lead to fast-slow systems.
The averaged variable \(v\) evolves on the base domain and carries the reduced dynamics.
The fluctuation \(w\) belongs to the mean-zero subspace and is acted upon by the singularly strong transverse operator \(\varepsilon^{-2}g^{-2}\partial_Y^2\).
The operator \(B_g=(1/g)\partial_x(g\partial_x)\) records the leading cross-sectional geometry in the averaged equation, while \(R_gw\) records the additional coupling between the averaged equation and the upper trace of the mean-zero fluctuation.
\end{remark}

\section{Existence of slow manifolds}\label{sec:manifolds}

We now pass from the concrete fast-slow splitting derived in \Cref{sec:general-thin} to an abstract semilinear framework in which the fast variable evolves in a stable space \(\mc X\) and the slow variable evolves in a space \(\mc Y\). 
We first state the abstract assumptions needed for the Lyapunov-Perron construction and then verify them for the thin-domain system from \Cref{sec:general-thin}. 
For the general framework we follow the Generalized Fenichel Theory introduced in a series of papers  by Hummel \& Kuehn \cite{hummel2022slow} and Kuehn \& Sulzbach \cite{kuehn2024infinite,kuehn2025fast,kuehn2026approximate}.\\

We consider the following abstract Cauchy problem on a Banach space:
\begin{align}\label{eq: fast diffusion}
    \begin{split}
        \partial_t w^\varepsilon &= \frac{1}{\varepsilon^2} A_\varepsilon w^\varepsilon + F(w^\varepsilon, v^\varepsilon),\\
        \partial_t v^\varepsilon&= B v^\varepsilon + G(w^\varepsilon, v^\varepsilon).
    \end{split}
\end{align}
Here $0< \varepsilon\ll 1$ is a small parameter encoding the separation of  (time) scales, and in the thin domain setting, the separation of spatial scales.
Moreover, \(A_\varepsilon\) is a family of sectorial operators on a fixed Hilbert space \(\mc X\), and  $B$ is a sectorial operator on the Hilbert space \(\mc Y\).
The initial data satisfies
\begin{align}
    w^\varepsilon(0)&=  w^\varepsilon_0 \in \mc X_1^\eps,\qquad  v^\varepsilon(0)=  v^\varepsilon_0\in \mc Y_1,
\end{align}
where $\mc X_1^\eps=D(A_\eps)$ denotes the domain of $A_\eps$, which can depend on $\eps$, and $\mc Y_1=D(B)$ denotes the domain of $B$.
The formal limit of the system as $\varepsilon\to 0$ is
\begin{align}\label{eq: fast diffuion limit}
   \begin{split}
       0&= A_0 w^0,\\
       \partial_t v^0&= B v^0+ G(w^0,v^0),\\
       v^0(0)&=v_0,
   \end{split}
\end{align} 
where the critical manifold $S_0$ is given as the null space of the operator $A_0$ in the space $\mc X_{1}^0$, i.e.
\begin{align}\label{eq: critical manifold}
    S_0 := \{(w,v)\in \mc X_1^0\times \mc Y_1 : A_0 w = 0\}.
\end{align}
If \(A_0\) is invertible on the fast space, then \(\ker A_0=\{0\}\) and hence
\[
 S_0 = \{(0,v): v\in \mc Y_1\}.
\]

\subsection{Assumptions}\label{Sec:assumptions}
For the operators $A$ and $B$ we make the following assumptions:
\begin{assumption}[Assumptions on the operators $A_\eps$ and $B$ ]\label{ass: operators}
\leavevmode
\begin{itemize}
    \item $A_\eps$ and $B$ are closed linear operators, where $A_\varepsilon : \mc X \supset D(A_\varepsilon) \to \mc X$ generates the exponentially stable and analytic $C_0$-semigroup $(\textnormal{e}^{tA_\eps})_{t\geq 0}\subset \mathcal{B}(\mc X)$ and $B:\mc Y\supset D(B)\to \mc Y$ generates the analytic $C_0$-semigroup $(\textnormal{e}^{tB})_{t\geq 0}\subset \mathcal{B}(\mc Y)$, respectively, where $\mathcal{B}(\cdot)$ denotes the space of bounded linear operators.
    \item The interpolation-extrapolation scales generated by $(\mc X,A_\eps)$ and $(\mc Y,B)$ are given by $(\mc X_\alpha^\eps)_{\alpha\in [-1,\infty)}$ and $(\mc Y_\alpha)_{\alpha\in [-1,\infty)}$ (for details regarding the definitions and basic properties of interpolation-extrapolation scales see \cite{amann1995linear}).
     \item There are constants $C_A,C_B, M_A, M_B>0$, $\omega_A <0 $ and $\omega_B\in \mathbb{R}$ such that
    \begin{align}\label{eq: semigroup estimate}
    \begin{split}
        \|\textnormal{e}^{tA_\eps}\|_{\mathcal{B}(\mc X)}&\leq M_A \textnormal{e}^{\omega_A t},\qquad \|\textnormal{e}^{tA_\eps}\|_{\mathcal{B}(\mc X_\gamma^\eps,\mc X_{\theta}^\eps)}\leq C_At^{\gamma-\theta}\textnormal{e}^{\omega_A t},\\
        \|\textnormal{e}^{tB}\|_{\mathcal{B}(\mc Y)}&\leq M_B \textnormal{e}^{\omega_B t},\qquad \|\textnormal{e}^{tB}\|_{\mathcal{B}(\mc Y_\delta,\mc Y_{\theta})}\leq C_B t^{\delta-\theta} \textnormal{e}^{\omega_B t}
        \end{split}
    \end{align}
    hold for all $t\in (0,\infty)$, where $0<\delta,\gamma\leq \theta\leq 1$.
\end{itemize}
\begin{remark}
    A particular realization of the assumptions for the $\eps$-dependent operator $A_\eps$ and fast space $\mc X_\alpha^\eps$ is the following.\\
    Let \(\mc X\) be a Hilbert space and let \(\mc V\hookrightarrow \mc X\) be a dense, continuously embedded Hilbert space.
    For each \(\varepsilon\in(0,1]\), let \(A_\varepsilon\) be the self-adjoint operator associated with a closed, symmetric, coercive form 
    \[ \mathfrak a_\varepsilon:\mc V\times \mc V\to\mathbb R .\]
    Assume that there are constants \(c_0,c_1,\mu>0\), independent of \(\varepsilon\), such that
    \[ \mathfrak a_\varepsilon(w,w)\ge \mu\|w\|_{\mc X}^2,\]
    and
    \[ c_0\|w\|_{\mc V}^2 \le \|w\|_{\mc X}^2+\mathfrak a_\varepsilon(w,w) \le c_1\|w\|_{\mc V}^2 \qquad\text{for all }w\in \mc V. \]
    Then \(A_\varepsilon\) generates an exponentially stable analytic semigroup on \(\mc X\), and
    \[\|e^{tA_\varepsilon}\|_{\mathcal B(\mc X)}\le M_Ae^{\omega_A t},\qquad \omega_A<0, \]
    with constants independent of \(\varepsilon\).
    Moreover, if
    \[\mc X_{\alpha}^\eps:=D((-A_\varepsilon)^\alpha),\]
    then
    \[\mc X_{\varepsilon,1/2}=\mc V\]
    as sets, with uniformly equivalent norms, while \(\mc X_{1}^\eps=D(A_\varepsilon)\) may depend on \(\varepsilon\).
    For \(0\le \beta\le \theta<1\), the analytic semigroup satisfies
    \[ \|e^{tA_\varepsilon}\|_{\mathcal B(\mc X_{\gamma}^\eps, \mc X_{\theta}^\eps)} \le C_A t^{\gamma-\theta}e^{\omega_A t}, \qquad t>0. \]
\end{remark}

\end{assumption}
For the nonlinear functions $F$ and $G$ we have the following assumptions:
\begin{assumption}[Assumptions on nonlinearities $F$ and $G$]\label{ass: nonlinearities}
\leavevmode
\begin{itemize}
    \item The nonlinearities $F:\mc X_\theta^\eps\times \mc Y_{1}\to \mc X$ and $G:\mc X_{\theta}^\eps\times \mc Y_{1}\to\mc  Y_\delta$ are Fr\'echet differentiable. Moreover, there are constants $L_F, L_G>0$ such that
    \begin{align} \label{eq: lip est non lin}
       \|\textnormal{D} F(x,y)\|_{\mathcal{B}(\mc X_\theta^\eps\times \mc Y_{1},\mc X)}&\leq L_F,\qquad  \|\textnormal{D} G(x,y)\|_{\mathcal{B}(\mc X_\theta^\eps\times \mc Y_{1},\mc Y_\delta)}\leq L_G, 
    \end{align}
    for all $(x,y)\in  \mc X_\theta^\eps\times \mc Y_{1}$.
    This implies that the following inequalities hold
        \begin{align}\label{eq: Lip F}
        \|F(x_1,y_1)-F(x_2,y_2)\|_{\mc X}&\leq L_F( \|x_1-x_2\|_{\mc X_\theta^\eps}+\|y_1-y_2\|_{\mc Y_{1}}),\\
        \intertext{and}
        \label{eq: Lip G}
        \|G(x_1,y_1)-G(x_2,y_2)\|_{\mc Y_\delta}&\leq L_G( \|x_1-x_2\|_{\mc X_\theta^\eps}+\|y_1-y_2\|_{\mc Y_{1}})
    \end{align}
    for all $x_1,x_2\in \mc X_\theta^\eps$ and $y_1,y_2\in\mc Y_{1}$.  
    \item For convenience, we assume that $F(0,0)=0$ and $G(0,0)=0$.
\end{itemize}
\end{assumption}

\begin{remark}[Strong solutions under the abstract assumptions]\label{rem: existence strong sol}
These assumptions guarantee the existence of strong solutions $(w^\varepsilon,v^\varepsilon)$ and $(w^0,v^0)$ in the space $C^1((0,T_*);\mc X\times \mc Y)\cap C([0,T_*);\mc X_\theta^\eps\times \mc Y_1)$, where $T_*$ denotes the maximal time of existence of the solutions. 
Moreover, we can write the solutions in the mild solution form using the Duhamel formula.
For more details we refer to \cite{lunardi2012analytic}.
\end{remark}

\begin{remark}
In applications, the Lipschitz assumption on the nonlinearities can be justified in several ways.
If a priori estimates or boundedness results are available, the assumption follows directly, as we will see in \Cref{sec:example}.
Alternatively, one may introduce a cut-off function to enforce the required Lipschitz bounds.
\end{remark}

We next impose additional technical assumptions on the Banach space \(Y\) that allow a decomposition into slow and fast spectral blocks. This splitting is needed because the operator \(B\) is unbounded and, in general, does not generate a backward evolution on the whole of \(\mc Y_1\); see \cite{hummel2022slow} for the underlying Lyapunov-Perron framework.\\

Let \(\zeta>0\) be a small parameter and write
\[
\mc Y = \mc Y_F^\zeta \oplus \mc Y_S^\zeta,
\]
to indicate the fast and slow components.
The splitting satisfies the following assumptions.
\begin{assumption}[Assumptions on the splitting and the slow variable space $\mc Y$]\label{ass: splitting slow comp}
\leavevmode
\begin{itemize}
    \item The spaces $\mc Y^\zeta_F$ and $\mc Y^\zeta_S$ are closed in $\mc Y$ and the projections $\pr_{\mc Y^\zeta_F}$ and $\pr_{\mc Y^\zeta_S}$ commute with $B$ on $\mc Y_1$.
    \item The spaces $\mc Y^\zeta_F\cap \mc Y_1$ and $\mc Y^\zeta_S\cap \mc Y_1$ are closed subspaces of $\mc Y_1$ and are endowed with the norm $\|\cdot\|_{\mc Y_1}$.
    \item The realization of \(B\) in \(\mc Y_F^\zeta\),
    \[
    B_F^\zeta:D(B_F^\zeta)\subset \mc Y_F^\zeta\to \mc Y_F^\zeta,
    \qquad
    B_F^\zeta v:=Bv,
    \]
    with
    \[
    D(B_F^\zeta)
    :=
    \{v\in D(B)\cap \mc Y_F^\zeta:\ Bv\in \mc Y_F^\zeta\},
    \]
    has \(0\) in its resolvent set.
    \item The realization \(B_S^\zeta\) of \(B\) in \(\mc Y_S^\zeta\) generates a
    \(C_0\)-group
    \[
    (e^{tB_S^\zeta})_{t\in\mathbb R}
    \subset \mathcal B(\mc Y_S^\zeta),
    \]
    and for \(t\ge0\) this group agrees with \(e^{tB}\) on \(\mc Y_S^\zeta\).  
    \item The fast subspace $\mc Y^\zeta_F$ contains the parts of $\mc Y_1$ that decay under the semigroup $(\textnormal{e}^{tB})_{t\geq 0}$ almost as fast as functions in the fast variable space $\mc X_\theta^\eps$ under the semi-group $(\textnormal{e}^{\zeta^{-1}t A_\eps})_{t\geq 0}$ generated by the operator $A$.
    The space $\mc Y^\zeta_S$ on the other hand contains the parts of $\mc Y_1$ that do not decay or which decay only slowly under the semigroup $(\textnormal{e}^{tB})_{t\geq 0}$ compared to $\mc X_\theta^\eps$ under $(\textnormal{e}^{\zeta^{-1}t A_\eps})_{t\geq 0}$. 
    Hence, there are constants $C_B,\,M_B>0$ such that for $\zeta>0$ small enough there are constants $N_F^\zeta,N_S^\zeta$ satisfying   
    $$0\leq N_F^\zeta< N_S^\zeta\leq |\zeta^{-1}\omega_A|$$
    such that for all $t\geq 0$ and $y_F\in \mc Y^\zeta_F,\, y_S\in \mc Y^\zeta_S$ we have 
    \begin{align*}
        \|\textnormal{e}^{tB}y_F\|_{\mc Y_1}&\leq C_B t^{\delta -1} \textnormal{e}^{(N_F^\zeta +\zeta^{-1}\omega_A)t}\|y_F\|_{\mc Y_\delta},\\
        \|\textnormal{e}^{-tB}y_S\|_{\mc Y_1}&\leq C_B t^{\delta -1} \textnormal{e}^{-(N_S^\zeta +\zeta^{-1}\omega_A)t}\|y_S\|_{\mc Y_\delta}.
    \end{align*}
   
    \item The Lyapunov-Perron contraction constant satisfies
    \begin{align}\label{spectral gap condition}
    L_{\mathrm{spec}}
    :={}&
    \frac{
    2^{1-\theta}\varepsilon^2 L_F C_A\Gamma(1-\theta)
    }{
    \Bigl(
    2(\varepsilon^2\zeta^{-1}-1)\omega_A
    +\varepsilon^2(N_S^\zeta+N_F^\zeta)
    \Bigr)^{1-\theta}
    }
    +
    \frac{
    2^{1+\delta}L_GC_B\Gamma(\delta)
    }{
    (N_S^\zeta-N_F^\zeta)^\delta
    }
    <1.
    \end{align}
\end{itemize}
\end{assumption}

\begin{remark}
    Note that, in contrast to previous results, we may choose $\varepsilon^2=\zeta$. This is a consequence of the special structure of the fast-diffusion problem \eqref{eq: fast diffusion}.
\end{remark}

\subsection{Slow Manifolds via the splitting of the slow variable}\label{sec:slow manifold}

To generalize Fenichel's Theorem for the existence of a slow manifold  to the present fast diffusion setting,
we introduce a system with a splitting of the formal slow variable, i.e., the $v$-component, satisfying the above assumptions, which gives
\begin{align}\label{eq:fast diffusion system splitting}
    \begin{split}
        \partial_t w^\varepsilon &= \frac{1}{\varepsilon^2}A_\eps w^\varepsilon + F(w^\varepsilon,v_F^\varepsilon, v_S^\varepsilon),\\
        \partial_t v_F^\varepsilon &= B v_F^\varepsilon +\pr_{\mc Y_F^\zeta} G(w^\varepsilon,v_F^\varepsilon, v_S^\varepsilon),\\
        \partial_t v_S^\varepsilon &= B v_S^\varepsilon +\pr_{\mc Y_S^\zeta} G(w^\varepsilon,v_F^\varepsilon, v_S^\varepsilon),\\
        w^\varepsilon(0)&= w^\varepsilon_0,\quad  v^\varepsilon_F(0)= \pr_{\mc Y_F^\zeta} v_0,\quad v^\varepsilon_S(0)= \pr_{\mc Y_S^\zeta} v_0.
    \end{split}
\end{align}
We first establish the invariant graph by considering, for a fixed $v_{S,0}\in \mc Y_S^\zeta\cap \mc Y_1$,
the Lyapunov-Perron operator
\[
\mathcal L_{v_{S,0},\varepsilon}:C_\eta^{\varepsilon,\theta}\to
C_\eta^{\varepsilon,\theta}.
\]
It is defined by
\[
\mathcal L_{v_{S,0},\varepsilon}
\begin{pmatrix}
w\\ v_F\\ v_S
\end{pmatrix}
(t)
:=
\begin{pmatrix}
\displaystyle
\int_{-\infty}^t e^{\varepsilon^{-2}A_\varepsilon(t-s)}
F(w(s),v_F(s),v_S(s))\,ds
\\[3mm]
\displaystyle
\int_{-\infty}^t e^{B_F^\zeta(t-s)}
\pr_{\mc Y_F^\zeta}G(w(s),v_F(s),v_S(s))\,ds
\\[3mm]
\displaystyle
e^{tB_S^\zeta}v_{S,0}
+
\int_0^t e^{B_S^\zeta(t-s)}
\pr_{\mc Y_S^\zeta}G(w(s),v_F(s),v_S(s))\,ds
\end{pmatrix},
\qquad t\le0 .
\]
Here the last integral is understood as an integral over the interval
\([t,0]\) with reversed orientation when \(t<0\), and it is meaningful because
\(B_S^\zeta\) generates a group on \(\mc Y_S^\zeta\).

We seek a fixed point in
\[
C_\eta^{\varepsilon,\theta}
:=
\left\{
(w,v_F,v_S)\in C((-\infty,0],
\mc X_\theta^\varepsilon\times(\mc Y_F^\zeta\cap\mc Y_1)
\times(\mc Y_S^\zeta\cap\mc Y_1)):
\|(w,v_F,v_S)\|_{C_\eta^{\varepsilon,\theta}}<\infty
\right\},
\]
where
\[
\|(w,v_F,v_S)\|_{C_\eta^{\varepsilon,\theta}}
:=
\sup_{t\le0}e^{-\eta t}
\left(
\|w(t)\|_{\mc X_\theta^\varepsilon}
+\|v_F(t)\|_{\mc Y_1}
+\|v_S(t)\|_{\mc Y_1}
\right),
\]
and
\[
\eta=\zeta^{-1}\omega_A+\frac{N_S^\zeta+N_F^\zeta}{2}<0.
\]

\begin{lemma}[Lyapunov-Perron fixed point]\label{lemma: lyapunov perron}
  Let the Assumptions \ref{ass: operators}, \ref{ass: nonlinearities} and \ref{ass: splitting slow comp} hold.
  Then the Lyapunov-Perron operator has a unique fixed point.
\end{lemma} 

\begin{proof}
This is the Lyapunov-Perron contraction theorem for the fast-diffusion scaling used here. 
The constants in \eqref{spectral gap condition} are precisely those obtained by estimating the three components of the Lyapunov-Perron map in the weighted space \(C_\eta^{\varepsilon,\theta}\).
We therefore invoke \cite[Lemma 3.3]{kuehn2026approximate}.
\end{proof}

We denote the fixed point of the Lyapunov-Perron operator evaluated at $t=0$ by 
\[
\bigl(h_w^{\varepsilon,\zeta}(v_{S,0}),
h_{v_F}^{\varepsilon,\zeta}(v_{S,0}),
v_{S,0}\bigr),
\]
where
\[
h_w^{\varepsilon,\zeta}:\mc Y_S^\zeta\cap \mc Y_1
\to \mc X_{\theta}^\varepsilon,
\qquad
h_{v_F}^{\varepsilon,\zeta}:\mc Y_S^\zeta\cap \mc Y_1
\to\mc  Y_F^\zeta\cap \mc Y_1.
\]
The invariant manifold $S_{\varepsilon,\zeta}$ is then given by
\begin{align}
    S_{\varepsilon,\zeta}:=\{ \big(h^{\varepsilon,\zeta}_{w}(v_S),h^{\varepsilon,\zeta}_{v_F}(v_S),v_S\big):v_S\in \mc Y_1\cap \mc Y_S^\zeta\}.
   \end{align}

The next result shows that this invariant manifold is indeed the slow manifold of the system and we obtain the following generalized Fenichel Theorem.
\begin{theorem}[Generalized Fenichel Theorem]\label{thm: Fenichel abstract}
   Let the Assumptions \ref{ass: operators}, \ref{ass: nonlinearities} and \ref{ass: splitting slow comp} hold and let the parameters $\varepsilon,\zeta>0$ satisfy $\varepsilon^2\zeta^{-1}\leq1$.
    Then the invariant manifold $S_{\varepsilon,\zeta}$ is indeed the slow manifold of the system satisfying
    \begin{enumerate}
        \item[(i)] Regularity: The manifold \(S_{\varepsilon,\zeta}\) is Lipschitz continuous.
        If, in addition, the nonlinearities satisfy the corresponding \(C^k\)-bounds
        required in the Lyapunov-Perron theorem, then \(S_{\varepsilon,\zeta}\) is
        \(C^k\)-smooth.
        \item[(ii)] Distance: The slow manifold has a (Hausdorff semi-)distance of $\mathcal{O}(\varepsilon)$ to $S_0$ as $\varepsilon,\zeta\to 0$.
        \item[(iii)] Attraction: $S_{\varepsilon,\zeta}$ is a locally exponential attracting invariant manifold.
        \item[(iv)] Convergence of semiflows: The semiflow on the slow manifold $S_{\varepsilon,\zeta}$ converges to the semiflow on the critical manifold $S_0$.
    \end{enumerate}
\end{theorem}
The dynamics on the slow manifold $S_{\varepsilon,\zeta}$ are thus given by
\begin{align}\label{eq: slow dynamics}
\begin{split}
    \partial_t v_S^\eps&= Bv_S^\eps +\pr_{\mc Y_S^\zeta}G\bigl(h^{\eps,\zeta}_{w}(v_S^\eps),h^{\eps,\zeta}_{v_F}(v_S^\eps),v_S^\eps\bigr),\\
    v_S^\eps(0)&=  \pr_{\mc Y_S^\zeta} v_0.
    \end{split}
\end{align}

\begin{remark}[Why do we need the splitting?]\label{rem:why-splitting}
    The obstruction to a direct two-component Lyapunov-Perron argument is easy to state: $B$ generates a forward analytic semigroup, but in general there is no backward evolution on the whole of $\mc Y_1$. The remedy is to split the slow variable into finitely many retained modes and a stable complement.
    Without the splitting $v=v_F+v_S$, the backward integral for the slow component is not available because $B$ only generates a forward analytic semigroup. The restriction to the finite-dimensional space $\mc Y_S^\zeta$ removes this obstruction.
\end{remark}

\begin{remark}[From the splitting to the spectral gap]
    The splitting $Y=\mc Y_F^\zeta\oplus \mc Y_S^\zeta$ resolves the first obstruction from \Cref{rem:why-splitting}: it allows us to use the backward evolution only on the finite-dimensional part $\mc Y_S^\zeta$, while the tail $\mc Y_F^\zeta$ is treated by forward decay. However, this splitting by itself is not sufficient for the Lyapunov-Perron map to be a contraction.

    The reason is that the nonlinear terms still couple the three components $w$, $v_F$, and $v_S$. Hence one must ensure that the decay on the tail $\mc Y_F^\zeta$ is strong enough, and that the backward evolution on $\mc Y_S^\zeta$ is separated enough from this tail decay, so that these couplings are dominated by the linear dynamics. In other words, the modes placed into $\mc Y_F^\zeta$ must be genuinely fast relative to the retained modes in $\mc Y_S^\zeta$; otherwise the nonlinear transfer between the two blocks is too large and the Lyapunov-Perron operator need not be contractive.

    This is precisely the role of the spectral gap condition~\eqref{spectral gap condition}. The quantity $N_S^\zeta-N_F^\zeta$ measures the separation between the retained block and the discarded tail of the slow operator $B$, while the term involving $\zeta^{-1}\omega_A$ measures the comparison with the fast decay in the $w$-equation. The requirement $L_{\mathrm{spec}}<1$ is therefore the quantitative condition that the linear separation of scales dominates the nonlinear coupling.

    Put differently: the splitting is needed to define the mixed backward-forward Lyapunov-Perron map, and the spectral gap is needed to make this map contractive. Without such a gap, the tail of $B$ may decay too slowly compared to the fast scale, so the distinction between retained slow modes and discarded fast modes is not strong enough to produce an invariant slow manifold by this method.

    The thin-domain application in the next part verifies this mechanism explicitly for a one-dimensional base domain.
\end{remark}

\subsection{Application to the thin-domain problem}
\label{subsec:thin-domain-application}

We now verify the abstract assumptions from \Cref{Sec:assumptions} for the thin-domain
fast-slow system derived in \Cref{sec:general-thin}. In this subsection we work in one
tangential dimension and assume
\[
\Omega=(0,L),\qquad
g\in C^3(\overline\Omega),\qquad
0<g_*\le g(x)\le g^* .
\]
We also impose the endpoint compatibility condition
\[
g'(0)=g'(L)=0.
\]
This condition is not merely technical. It is the hypothesis which decouples
the slow boundary condition from the fast trace. Indeed, \Cref{prop:general-splitting} gives the
induced slow boundary condition
\[
\partial_xv-\frac{g'}g\,w(\cdot,1)=0
\qquad\text{on }\partial\Omega .
\]
If \(g'=0\) on \(\partial\Omega\), this reduces to the homogeneous Neumann
condition
\[
\partial_xv=0
\qquad\text{on }\partial\Omega.
\]
Without this endpoint condition the slow domain would depend on the fast
boundary trace \(w(\cdot,1)\), and the slow operator could not be treated as a
fixed Sturm-Liouville operator with Neumann boundary conditions on \(\Omega\).

The verification of the abstract hypotheses is organized as follows.
First we define the fixed fast and slow spaces and the corresponding fast and slow operators. 
Second, we explain why the fluctuation \(w=u-v\) does not by itself belong to the homogeneous fast domain and introduce a lifting that produces a homogeneous fast variable. 
Third, we strengthen the fast topology to a trace-compatible space, because the geometry-induced term \(R_gw\) (see \eqref{eq: boundary term Rg}) depends on the boundary trace \(w(\cdot,1)\).
Fourth, we derive the lifted mild system and prove that its nonlinearities satisfy the required uniform Lipschitz estimates.
Finally, we verify the spectral splitting of the slow Sturm-Liouville operator and invoke the abstract Lyapunov-Perron theorem.

\subsubsection*{Fast and slow spaces}
\label{subsubsec:thin-domain-fast-slow-spaces}

We use the notation
\[
Q:=\Omega\times(0,1),
\qquad
q(x):=\frac{g'(x)}{g(x)},
\qquad
\mathcal D:=\partial_x-Yq(x)\partial_Y,
\]
and
\[
\mathcal E\phi(x,Y):=\phi(x).
\]
The averaging operator is
\[
Mw(x):=\int_0^1 w(x,Y)\,dY.
\]
The fast space is the mean-zero space
\[
\mathcal X:=\ker M
=
\{z\in L^2(Q):Mz=0\}.
\]
We equip \(\mathcal X\) with the weighted inner product
\[
(z,\phi)_g
:=
\int_Q z(x,Y)\phi(x,Y)g(x)\,dx\,dY.
\]
The corresponding Hilbert space will be denoted by \(\mathcal X_g\). Since
\(g_*\le g\le g^*\), this norm is equivalent to the standard \(L^2(Q)\)-norm.
The common form domain is
\[
\mathcal V:=H^1(Q)\cap \mathcal X.
\]

For \(\varepsilon\in(0,1]\), define the symmetric form
\[
\mathfrak a_\varepsilon(z,\phi)
:=
\int_Q \frac1{g(x)^2}\partial_Yz\,\partial_Y\phi\,g(x)\,dx\,dY
+
\varepsilon^2
\int_Q \mathcal Dz\,\mathcal D\phi\,g(x)\,dx\,dY,
\qquad
z,\phi\in \mathcal V.
\]
The fast operator \(A_\varepsilon\) is defined as the operator associated with
\(-\mathfrak a_\varepsilon\), namely
\[
D(A_\varepsilon)
:=
\left\{
z\in \mathcal V:
\exists h\in\mathcal X
\text{ such that }
\mathfrak a_\varepsilon(z,\phi)=-(h,\phi)_g
\text{ for all }\phi\in\mathcal V
\right\},
\]
and
\[
A_\varepsilon z:=h.
\]
This variational definition is the primary definition of the fast operator.

For the slow variable we use
\[
\mathcal Y:=L^2(\Omega,g(x)\,dx)
\]
and the Neumann Sturm-Liouville operator
\[
B_gv:=\frac1g\partial_x(g\partial_xv),
\qquad
\mathcal Y_1:=D(B_g)
=
\{v\in H^2(\Omega):\partial_xv=0\text{ on }\partial\Omega\}.
\]

\begin{lemma}[Fast operator on the mean-zero space]
\label{lem:fast-operator-thin-domain}
For every \(\varepsilon\in(0,1]\), the operator
\[
A_\varepsilon:D(A_\varepsilon)\subset\mathcal X\to\mathcal X
\]
is densely defined, self-adjoint on \(\mathcal X_g\), sectorial, and generates
an exponentially stable analytic \(C_0\)-semigroup. Moreover, if
\[
\mu:=\frac{\pi^2}{(g^*)^2},
\]
then
\[
\sigma(A_\varepsilon)\subset(-\infty,-\mu]
\qquad\text{for all }\varepsilon\in(0,1].
\]
In particular,
\[
0\in\rho(A_\varepsilon),
\qquad
\|A_\varepsilon^{-1}\|_{\mathcal B(\mathcal X_g)}
\le
\mu^{-1}.
\]
Finally,
\[
D((-A_\varepsilon)^{1/2})
=
\mathcal V
=
H^1(Q)\cap\ker M,
\]
with equivalence of norms uniform in \(\varepsilon\) when \(\mathcal V\) is
equipped with the form norm
\[
\|z\|_{\mathcal V,\varepsilon}^2
:=
\|z\|_{\mathcal X_g}^2+\mathfrak a_\varepsilon(z,z).
\]
\end{lemma}

\begin{proof}
For each fixed \(\varepsilon>0\), the form
\(\mathfrak a_\varepsilon\) is densely defined, symmetric, and closed on
\(\mathcal V\). Indeed, the form norm controls \(\partial_Yz\) and
\(\varepsilon\mathcal Dz\). Since
\[
\partial_xz=\mathcal Dz+Yq\partial_Yz,
\]
it controls the full \(H^1(Q)\)-norm for fixed \(\varepsilon\), although the
corresponding constant may degenerate as \(\varepsilon\to0\).

Since \(Mz=0\), the one-dimensional Poincar\'e inequality in the \(Y\)-variable
gives, for a.e. \(x\in\Omega\),
\[
\int_0^1 |z(x,Y)|^2\,dY
\le
\frac1{\pi^2}
\int_0^1 |\partial_Yz(x,Y)|^2\,dY .
\]
Using \(g(x)\le g^*\), we obtain
\[
\|z\|_{\mathcal X_g}^2
\le
\frac{(g^*)^2}{\pi^2}
\int_Q \frac1{g(x)^2}|\partial_Yz|^2g(x)\,dx\,dY
\le
\frac{(g^*)^2}{\pi^2}\mathfrak a_\varepsilon(z,z).
\]
Thus
\[
\mathfrak a_\varepsilon(z,z)
\ge
\mu\|z\|_{\mathcal X_g}^2,
\qquad
\mu=\frac{\pi^2}{(g^*)^2},
\]
uniformly in \(\varepsilon\). The representation theorem for closed coercive
forms therefore gives a self-adjoint operator associated with
\(-\mathfrak a_\varepsilon\), and the spectral inclusion follows from the
coercivity estimate.
For more details we refer to \cite{Kato1966,ouhabaz2009analysis}.
Hence \(A_\varepsilon\) is sectorial of angle zero and
generates an analytic exponentially stable semigroup on \(\mathcal X_g\).
Since the weighted and unweighted \(L^2\)-norms are equivalent, the same
generation and stability estimates hold on \(\mathcal X\), up to changing
constants.

The square-root domain of the nonnegative self-adjoint operator
\(-A_\varepsilon\) associated with \(\mathfrak a_\varepsilon\) is exactly the
form domain. Hence
\[
D((-A_\varepsilon)^{1/2})=\mathcal V
\]
with equivalence of the square-root norm and the form norm
\(\|\cdot\|_{\mathcal V,\varepsilon}\).
\end{proof}

\begin{remark}[No uniform full \(H^1\)-control]
Note the norm on $\mc V$ is not uniformly equivalent to the standard \(H^1(Q)\)-norm as \(\varepsilon\to0\). 
For example, in the flat case
take
\[
z_N(x,Y)=\sin(N\pi x/L)\cos(\pi Y),
\qquad Mz_N=0.
\]
Then
\[
\|z_N\|_{H^1(Q)}^2\sim 1+N^2,
\]
whereas
\[
\|z_N\|_{\mathcal X}^2+\mathfrak a_\varepsilon(z_N,z_N)
\sim
1+\varepsilon^2N^2.
\]
Choosing \(N\sim\varepsilon^{-2}\) shows that no \(\varepsilon\)-independent
lower bound by the standard \(H^1\)-norm can hold.
\end{remark}

\begin{remark}[Strong form under additional regularity]
\label{rem:strong-form-fast-operator}
If the conormal elliptic problem has the required \(H^2\)-regularity, then the
variational operator is represented by
\[
A_\varepsilon z
=
\frac1{g(x)^2}\partial_Y^2z
+
\varepsilon^2\mathcal D^2z
-
\varepsilon^2\mathcal E(R_gz),
\]
where
\[
R_gz:=-\frac1g\partial_x\bigl(g'z(\cdot,1)\bigr).
\]
In that regular case,
\[
D(A_\varepsilon)
=
\left\{
z\in H^2(Q)\cap \mathcal X:
\begin{array}{ll}
\partial_Yz=0, & Y=0,\\[1mm]
\partial_Yz-\varepsilon^2gg'\mathcal Dz=0, & Y=1,\\[1mm]
\mathcal Dz=0, & \partial\Omega\times(0,1)
\end{array}
\right\}.
\]
Equivalently, the upper conormal condition is
\[
\bigl(1+\varepsilon^2(g')^2\bigr)\partial_Yz
-
\varepsilon^2gg'\partial_xz
=
0
\qquad\text{on }Y=1.
\]
The term \(-\varepsilon^2\mathcal E(R_gz)\) is the projection which keeps
\(A_\varepsilon z\) in the mean-zero space \(\ker M\). Without additional
elliptic regularity, the displayed \(H^2\)-domain should be read only as the
formal strong realization; the variational definition above is the operator
definition.

To verify the natural boundary conditions, let \(z,\phi\in H^2(Q)\cap\mathcal X\).
Then
\[
\mathfrak a_\varepsilon(z,\phi)
=
\int_Q \frac1g\partial_Yz\,\partial_Y\phi
+
\varepsilon^2\int_Q g\mathcal Dz\,\mathcal D\phi .
\]
Since
\[
\mathcal D\phi=\partial_x\phi-Yq\partial_Y\phi,
\qquad q=\frac{g'}g,
\]
integration by parts gives the boundary contribution
\[
-\int_\Omega \frac1g\partial_Yz\,\phi\Big|_{Y=0}\,dx
+
\int_\Omega
\left(
\frac1g\partial_Yz-\varepsilon^2g'\mathcal Dz
\right)
\phi\Big|_{Y=1}\,dx
+
\varepsilon^2\int_{\partial\Omega\times(0,1)}
n_x\,g\mathcal Dz\,\phi\,dY .
\]
Hence the natural boundary conditions are
\[
\partial_Yz=0\quad\text{on }Y=0,
\]
\[
\partial_Yz-\varepsilon^2gg'\mathcal Dz=0\quad\text{on }Y=1,
\]
and
\[
\mathcal Dz=0\quad\text{on }\partial\Omega\times(0,1).
\]
Thus the variational form indeed generates the conormal boundary conditions
used in the strong realization.
\end{remark}

\begin{remark}[Limit operator and critical manifold]
The conormal strong domains depend on \(\varepsilon\), so one should not
interpret \(A_\varepsilon\to A_0\) as convergence on a common strong domain.
The natural convergence is form convergence. The limit form is
\[
\mathfrak a_0(z,\phi)
=
\int_Q \frac1{g(x)^2}\partial_Yz\,\partial_Y\phi\,g(x)\,dx\,dY,
\]
with form domain
\[
\mathcal V_0
=
\{z\in L^2(Q):\partial_Yz\in L^2(Q),\ Mz=0\}.
\]
The associated limit operator is
\[
A_0z=\frac1{g(x)^2}\partial_Y^2z,
\]
with Neumann boundary conditions in \(Y\) and the mean-zero constraint. No
lateral boundary condition appears in the limiting vertical operator. Since
\(A_0\) is invertible on \(\ker M\), the critical manifold of the lifted
fast-slow system is
\[
S_0=\{(0,v):v\in\mathcal Y_1\}.
\]
\end{remark}

\subsubsection*{Conormal liftings and homogeneous boundary conditions}
\label{subsubsec:thin-domain-conormal-liftings}

The fluctuation \(w=u-v\) from Section~3 does not satisfy homogeneous conormal
boundary conditions. Instead,
\[
\partial_Yw-\varepsilon^2gg'\mathcal Dw
=
\varepsilon^2gg'\partial_xv
\qquad\text{on }Y=1,
\]
and
\[
\mathcal Dw=-\partial_xv
\qquad\text{on }\partial\Omega\times(0,1).
\]
Since \(v\in\mathcal Y_1\) and \(g'=0\) on \(\partial\Omega\), the lateral
compatibility is consistent with homogeneous conormal data after lifting.

\begin{lemma}[Strong and weak conormal liftings]
\label{lem:strong-weak-conormal-lifting}
For every sufficiently small \(\varepsilon>0\), there are linear operators
\[
\Lambda_\varepsilon^1:\mathcal Y_1\to H^2(Q)\cap\mathcal X,
\qquad
\Lambda_\varepsilon^0:\mathcal Y\to\mathcal X,
\]
with the following properties.

For every \(v\in\mathcal Y_1\), the strong lifting
\(\Lambda_\varepsilon^1v\) satisfies
\[
M\Lambda_\varepsilon^1v=0,
\]
\[
\partial_Y(\Lambda_\varepsilon^1v)=0
\qquad\text{on }\Omega\times\{0\},
\]
\[
\partial_Y(\Lambda_\varepsilon^1v)
-
\varepsilon^2gg'\mathcal D(\Lambda_\varepsilon^1v)
=
\varepsilon^2gg'\partial_xv
\qquad\text{on }\Omega\times\{1\},
\]
and
\[
\mathcal D(\Lambda_\varepsilon^1v)=0
\qquad\text{on }\partial\Omega\times(0,1).
\]
Moreover,
\[
\|\Lambda_\varepsilon^1v\|_{H^2(Q)}
\le
C\varepsilon^2\|v\|_{\mathcal Y_1}.
\]
Consequently, for every \(\theta\in(3/4,1)\),
\[
\|\Lambda_\varepsilon^1v\|_{H^{2\theta}(Q)}
\le
C\varepsilon^2\|v\|_{\mathcal Y_1}.
\]
In particular, for every \(0<\delta<\theta-\frac34\),
\[
\left\|
\frac1{\varepsilon^2g^2}\partial_Y^2\Lambda_\varepsilon^1v
\right\|_{L^2(Q)}
+
\|\mathcal D^2\Lambda_\varepsilon^1v\|_{L^2(Q)}
+
\|R_g\Lambda_\varepsilon^1v\|_{\mathcal Y_\delta}
\le
C\|v\|_{\mathcal Y_1}.
\]
The weak lifting satisfies
\[
\|\Lambda_\varepsilon^0\eta\|_{\mathcal X}
\le
C\varepsilon^2\|\eta\|_{\mathcal Y},
\qquad \eta\in\mathcal Y.
\]
If \(\eta\in\mathcal Y_1\), then
\[
\Lambda_\varepsilon^0\eta=\Lambda_\varepsilon^1\eta .
\]
\end{lemma}

\begin{proof}
Define the upper conormal trace operator
\[
\mathcal B_\varepsilon \xi
:=
\partial_Y\xi-\varepsilon^2gg'\mathcal D\xi
=
\bigl(1+\varepsilon^2(g')^2\bigr)\partial_Y\xi
-
\varepsilon^2gg'\partial_x\xi
\qquad\text{on }Y=1.
\]
Consider the trace map
\[
\mathcal T_\varepsilon \xi
:=
\left(
M\xi,\,
\partial_Y\xi|_{Y=0},\,
\mathcal B_\varepsilon\xi|_{Y=1},\,
\mathcal D\xi|_{\partial\Omega\times(0,1)}
\right).
\]
At \(\varepsilon=0\), this becomes
\[
\mathcal T_0\xi
=
\left(
M\xi,\,
\partial_Y\xi|_{Y=0},\,
\partial_Y\xi|_{Y=1},\,
\partial_x\xi|_{\partial\Omega\times(0,1)}
\right).
\]
By the standard right-inverse theorem for compatible Neumann trace data on
rectangles, together with the prescribed fibre mean,
\(\mathcal T_0\) admits a bounded right inverse from the compatible trace space
into \(H^2(Q)\); see \cite{Grisvard1985,McLean2000}.
Moreover,
\[
\mathcal T_\varepsilon-\mathcal T_0
=
\mathcal O(\varepsilon^2)
\]
as a map from \(H^2(Q)\) into the same trace space. Hence the right inverse
persists for sufficiently small \(\varepsilon\) by a Neumann-series argument.
We denote the resulting uniformly bounded right inverse by
\(\mathcal R_\varepsilon\).

For \(v\in\mc Y_1\), set
\[
h_\varepsilon(v):=\varepsilon^2gg'\partial_xv
\]
and define
\[
\Lambda_\varepsilon^1v
:=
\mathcal R_\varepsilon(0,0,h_\varepsilon(v),0).
\]
Therefore, by construction,
\[
M\Lambda_\varepsilon^1v=0.
\]

From the definition of $h_\varepsilon(v)$ it follows that
\[
\|h_\varepsilon(v)\|_{H^{1/2}(\Omega)}
\le
C\varepsilon^2\|v\|_{\mathcal Y_1}.
\]
Applying now the uniformly bounded right inverse to the compatible trace datum $(0,0,h_\varepsilon(v),0)$
gives, by construction, \(\Lambda_\varepsilon^1v\), and hence
\[
\|\Lambda_\varepsilon^1v\|_{H^2(Q)}
\le
C\varepsilon^2\|v\|_{\mathcal Y_1}.
\]
Interpolation then gives the \(H^{2\theta}\)-estimate. 

The estimate involving
\(\varepsilon^{-2}\partial_Y^2\Lambda_\varepsilon^1v\) follows from the
\(H^2\)-estimate. The estimate for \(\mathcal D^2\Lambda_\varepsilon^1v\)
follows from boundedness of the coefficients of \(\mathcal D\). Finally, the
trace theorem gives
\[
\Lambda_\varepsilon^1v(\cdot,1)\in H^{2\theta-1/2}(\Omega).
\]
Thus
\[
R_g\Lambda_\varepsilon^1v
=
-\frac1g\partial_x\bigl(g'\Lambda_\varepsilon^1v(\cdot,1)\bigr)
\in H^{2\theta-3/2}(\Omega).
\]
Since \(\mathcal Y_\delta\simeq H^{2\delta}(\Omega)\) for
\(0<\delta<1\), the embedding
\[
H^{2\theta-3/2}(\Omega)\hookrightarrow H^{2\delta}(\Omega)
\]
holds whenever
\[
0<\delta<\theta-\frac34.
\]
It remains to construct the weak lifting. The same family of trace right
inverses admits a weak extension. More precisely, the conormal datum
\[
\varepsilon^2gg'\partial_x\eta
\]
is meaningful as an element of \(H^{-1}(\Omega)\) whenever
\(\eta\in \mc Y=L^2(\Omega,g\,dx)\), and
\[
\|\varepsilon^2gg'\partial_x\eta\|_{H^{-1}(\Omega)}
\le C\varepsilon^2\|\eta\|_{\mc Y}.
\]
Using the weak trace right inverse associated with the same operator
\(\mathcal T_\varepsilon\), we define
\[
\Lambda_\varepsilon^0\eta
:=
\mathcal R_\varepsilon^{\,0}
\bigl(0,0,\varepsilon^2gg'\partial_x\eta,0\bigr).
\]
Then
\[
\Lambda_\varepsilon^0:\mc Y\to\mc X
\]
is linear and satisfies
\[
\|\Lambda_\varepsilon^0\eta\|_{\mc X}
\le C\varepsilon^2\|\eta\|_{\mc Y}.
\]

The weak right inverse \(\mathcal R_\varepsilon^{\,0}\) is chosen as the
continuous extension of the strong right inverse
\(\mathcal R_\varepsilon\). Equivalently, if
\(\eta_n\in\mc Y_1\) and \(\eta_n\to\eta\) in \(\mc Y\), then
\[
\Lambda_\varepsilon^0\eta
=
\lim_{n\to\infty}\Lambda_\varepsilon^1\eta_n
\qquad\text{in }\mc X.
\]
Therefore, for every \(\eta\in\mc Y_1\),
\[
\Lambda_\varepsilon^0\eta=\Lambda_\varepsilon^1\eta .
\]
\end{proof}

\subsubsection*{Lifted mild formulation}
\label{subsubsec:thin-domain-lifted-mild-system}

We now introduce the homogeneous fast variable
\[
z^\varepsilon:=w^\varepsilon-\Lambda_\varepsilon^1v^\varepsilon.
\]
Formally,
\[
w^\varepsilon=z^\varepsilon+\Lambda_\varepsilon^1v^\varepsilon,
\]
and \(z^\varepsilon\) satisfies homogeneous conormal boundary conditions.
However, at the regularity level needed for the slow-manifold construction one
should not differentiate this identity pointwise in time. Instead, the lifted
equation is understood in the mild sense.

\begin{proposition}[Mild lifted fast-slow system]
\label{prop:mild-lifted-fast-slow-system}
Let
\[
z_0^\varepsilon:=w_0^\varepsilon-\Lambda_\varepsilon^1v_0^\varepsilon.
\]
The lifted variables \((z^\varepsilon,v^\varepsilon)\) satisfy the mild system
\[
z^\varepsilon(t)
=
e^{t\varepsilon^{-2}A_\varepsilon}z_0^\varepsilon
+
\int_0^t
e^{(t-s)\varepsilon^{-2}A_\varepsilon}
F_\varepsilon(z^\varepsilon(s),v^\varepsilon(s))\,ds,
\]
\[
v^\varepsilon(t)
=
e^{tB_g}v_0^\varepsilon
+
\int_0^t
e^{(t-s)B_g}
G_\varepsilon(z^\varepsilon(s),v^\varepsilon(s))\,ds,
\]
where
\[
G_\varepsilon(z,v)
:=
R_g(z+\Lambda_\varepsilon^1v)
+
Mf(v+z+\Lambda_\varepsilon^1v),
\]
and
\begin{align*}
F_\varepsilon(z,v)
:={}&
(I-M)f(v+z+\Lambda_\varepsilon^1v)
+
\frac1{\varepsilon^2g^2}\partial_Y^2\Lambda_\varepsilon^1v
+
\mathcal D^2\Lambda_\varepsilon^1v
\\
&-
\mathcal E(R_g\Lambda_\varepsilon^1v)
-
\mathcal E\!\left(\frac{g'}g\partial_xv\right)
-
\Lambda_\varepsilon^0
\bigl(B_gv+G_\varepsilon(z,v)\bigr).
\end{align*}
\end{proposition}

\begin{proof}
For smooth data and smooth solutions  with \(\partial_t v^\varepsilon\in\mc Y_1\), one may differentiate
\[
z^\varepsilon=w^\varepsilon-\Lambda_\varepsilon^1v^\varepsilon
\]
and obtain formally
\[
\partial_tz^\varepsilon
=
\partial_tw^\varepsilon
-
\Lambda_\varepsilon^1(\partial_tv^\varepsilon).
\]
Since the weak lifting is the continuous extension of the strong lifting and
\(\Lambda_\varepsilon^0=\Lambda_\varepsilon^1\) on \(\mc Y_1\), this can be
written as
\[
\partial_tz^\varepsilon
=
\partial_tw^\varepsilon
-
\Lambda_\varepsilon^0(\partial_tv^\varepsilon).
\]
The slow equation gives
\[
\partial_tv^\varepsilon
=
B_gv^\varepsilon+G_\varepsilon(z^\varepsilon,v^\varepsilon).
\]
Substituting
\[
w^\varepsilon=z^\varepsilon+\Lambda_\varepsilon^1v^\varepsilon
\]
into the split system from \Cref{sec:general-thin}, and using the homogeneous
conormal boundary conditions satisfied by \(z^\varepsilon\), the homogeneous
fast part is represented by \(A_\varepsilon z^\varepsilon\). The remaining
terms are precisely $F_\varepsilon(z^\varepsilon,v^\varepsilon)$.
Thus, for smooth solutions,
\[
\partial_t z^\varepsilon
=
\varepsilon^{-2}A_\varepsilon z^\varepsilon
+
F_\varepsilon(z^\varepsilon,v^\varepsilon),
\qquad
\partial_t v^\varepsilon
=
B_gv^\varepsilon
+
G_\varepsilon(z^\varepsilon,v^\varepsilon).
\]
Applying the variation-of-constants formula gives the two displayed mild
equations.

We now pass to general data. Let
\[
z_0^\varepsilon\in \mathcal X,
\qquad
v_0^\varepsilon\in \mathcal Y_1,
\]
and approximate these data by smooth data
\[
z_{0,n}^\varepsilon\in D(A_\varepsilon),
\qquad
v_{0,n}^\varepsilon\in D(B_g)
\]
such that
\[
z_{0,n}^\varepsilon\to z_0^\varepsilon
\quad\text{in }\mathcal X,
\qquad
v_{0,n}^\varepsilon\to v_0^\varepsilon
\quad\text{in }\mathcal Y_1 .
\]
For the corresponding smooth solutions
\((z_n^\varepsilon,v_n^\varepsilon)\), the mild identities have already been
proved. The continuity of the semigroups \(e^{t\varepsilon^{-2}A_\varepsilon}\) and \(e^{tB_g}\),
the boundedness of
\[
\Lambda_\varepsilon^0:\mathcal Y\to\mathcal X,
\qquad
\Lambda_\varepsilon^1:\mathcal Y_1\to H^2(Q)\cap\mathcal X,
\]
and the continuity of the nonlinear terms imply
\[
F_\varepsilon(z_n^\varepsilon,v_n^\varepsilon)
\to
F_\varepsilon(z^\varepsilon,v^\varepsilon)
\quad\text{in }L^1_{\mathrm{loc}}(0,T;\mathcal X),
\]
and
\[
G_\varepsilon(z_n^\varepsilon,v_n^\varepsilon)
\to
G_\varepsilon(z^\varepsilon,v^\varepsilon)
\quad\text{in }L^1_{\mathrm{loc}}(0,T;\mathcal Y).
\]
Passing to the limit in the two Duhamel formulas therefore yields the displayed
mild system for general data.
\end{proof}

\begin{remark}[The weak lifting]
The term
\[
\Lambda_\varepsilon^0\bigl(B_gv^\varepsilon
+G_\varepsilon(z^\varepsilon,v^\varepsilon)\bigr)
\]
is the mild replacement of the formal expression
\[
\Lambda_\varepsilon^1(\partial_t v^\varepsilon).
\]
For smooth solutions these two expressions agree, because
\[
\partial_t v^\varepsilon
=
B_gv^\varepsilon+G_\varepsilon(z^\varepsilon,v^\varepsilon)
\]
and \(\Lambda_\varepsilon^0=\Lambda_\varepsilon^1\) on \(\mathcal Y_1\).
For general mild solutions, however, one should not differentiate
\(v^\varepsilon\) pointwise in time. The extension
\[
\Lambda_\varepsilon^0:\mathcal Y\to\mathcal X
\]
allows the same term to be interpreted directly in the Duhamel formula.
Thus the lifted system is obtained by closing the smooth identity in the
natural semigroup topology.
\end{remark}

\subsubsection*{Trace-compatible topology and nonlinear estimates}
\label{subsubsec:thin-domain-trace-compatible-topology}

For nonconstant \(g\), the nonlinear estimates require a trace-compatible fast
topology. We therefore define, for \(\theta\in(3/4,1)\),
\[
\mc Z_\theta^\varepsilon
:=
\mathcal X_\theta^\varepsilon
\cap H^{2\theta}(Q)\cap\mathcal X,
\]
with norm
\[
\|z\|_{\mc Z_\theta^\varepsilon}
:=
\|z\|_{\mathcal X_\theta^\varepsilon}
+
\|z\|_{H^{2\theta}(Q)}.
\]
The additional isotropic Sobolev component is needed because
\[
R_gz=-\frac1g\partial_x\bigl(g'z(\cdot,1)\bigr)
\]
contains a tangential derivative of a boundary trace. If
\[
0<\delta<\theta-\frac34,
\]
then the trace theorem gives
\[
z(\cdot,1)\in H^{2\theta-1/2}(\Omega),
\]
and hence
\[
R_gz\in H^{2\theta-3/2}(\Omega)
\hookrightarrow H^{2\delta}(\Omega)
\simeq\mathcal Y_\delta.
\]
Consequently,
\[
\|R_gz\|_{\mathcal Y_\delta}
\le
C\|z\|_{\mc Z_\theta^\varepsilon}.
\]

\begin{lemma}[Mean-zero consistency of the lifted fast vector field]
For every \(v\in \mc Y_1\), the geometric lifting terms satisfy
\[
M\left[
 \frac{1}{\varepsilon^2g^2}\partial_Y^2\Lambda_\varepsilon^1v
 +\mc D^2\Lambda_\varepsilon^1v
 -\mc E(R_g\Lambda_\varepsilon^1v)
 -\mc E\left(\frac{g'}{g}\partial_xv\right)
\right]=0.
\]
Consequently,
\[
F_\varepsilon(z,v)\in\mc X
\]
whenever the terms in the definition of \(F_\varepsilon\) are well-defined.

\end{lemma}

\begin{proof}
Apply the averaging identity from the fast-slow splitting to
\(\xi=\Lambda_\varepsilon^1v\).  Since \(M\Lambda_\varepsilon^1v=0\) and
\[
\partial_Y\Lambda_\varepsilon^1v-\varepsilon^2gg'\mc D\Lambda_\varepsilon^1v
 =\varepsilon^2gg'\partial_xv
 \quad\text{on }Y=1,
\]
the same cancellation as in the derivation of Theorem~3.3 gives
\[
M\left[
 \frac{1}{\varepsilon^2g^2}\partial_Y^2\Lambda_\varepsilon^1v
 +\mc D^2\Lambda_\varepsilon^1v
\right]
=
R_g\Lambda_\varepsilon^1v+\frac{g'}{g}\partial_xv .
\]
This is precisely the displayed identity.  
Next, by construction,
\[
\Lambda_\varepsilon^0:\mc Y\to \mc X=\ker M.
\]
Therefore
\[
M\Lambda_\varepsilon^0\bigl(B_gv+G_\varepsilon(z,v)\bigr)=0.
\]
Finally, \((I-M)f(\cdot)\in \mc X\) by definition. 
Combining these observations shows that all terms in \(F_\varepsilon(z,v)\) have zero average after the geometric lifting terms are grouped as above. 
Hence \(F_\varepsilon(z,v)\in \mc X\).
\end{proof}

\begin{lemma}[Trace-compatible fast smoothing estimate]
\label{lem:trace-compatible-fast-smoothing}
Let
\[
L_\varepsilon:=\varepsilon^{-2}A_\varepsilon .
\]
By \Cref{prop:uniform-H2-projected-fast-operator}, the conormal realization of
\(L_\varepsilon=\varepsilon^{-2}A_\varepsilon\) satisfies
\[
D(L_\varepsilon)\subset H^2(Q)\cap\mathcal X
\]
and
\begin{align} \label{eq:uniform-parameter-elliptic-estimate}
\|u\|_{H^2(Q)}
\le
C\bigl(\|L_\varepsilon u\|_{\mathcal X}+\|u\|_{\mathcal X}\bigr),
\quad \textnormal{for }\quad u\in D(L_\varepsilon),
\end{align}
with \(C\) independent of sufficiently small \(\varepsilon>0\).
Then, for every \(\theta\in(0,1)\), there are constants
\(C_\theta>0\) and \(c>0\), independent of sufficiently small \(\varepsilon>0\), such that
\[
\|e^{tL_\varepsilon}\|_{\mathcal B(\mc X,\mc Z_\theta^\varepsilon)}
 \le C_\theta t^{-\theta}e^{-ct/\varepsilon^2},
 \qquad t>0 .
\]
\end{lemma}

\begin{proof}
The \(\mathcal X_\theta^\varepsilon\)-part follows from analyticity of
\(A_\varepsilon\). Since \(L_\varepsilon=\varepsilon^{-2}A_\varepsilon\),
\[
(-L_\varepsilon)^\theta
=
\varepsilon^{-2\theta}(-A_\varepsilon)^\theta.
\]
Therefore
\[
\|e^{tL_\varepsilon}\|_{\mathcal B(\mathcal X,\mathcal X_\theta^\varepsilon)}
\le
C \varepsilon^{2\theta}t^{-\theta}e^{-ct/\varepsilon^2}.
\]

It remains to control the isotropic Sobolev component. The uniform elliptic
estimate \eqref{eq:uniform-parameter-elliptic-estimate} gives a continuous
embedding
\[
D(L_\varepsilon)\hookrightarrow H^2(Q)
\]
with constant independent of \(\varepsilon\). Interpolating between
\(\mathcal X=L^2(Q)\cap\ker M\) and \(D(L_\varepsilon)\) yields
\[
D((-L_\varepsilon)^\theta)\hookrightarrow H^{2\theta}(Q),
\]
again with a uniform embedding constant. Hence
\[
\|e^{tL_\varepsilon}\|_{\mathcal B(\mathcal X,H^{2\theta}(Q))}
\le
C
\|(-L_\varepsilon)^\theta e^{tL_\varepsilon}\|_{\mathcal B(\mathcal X)}
\le
Ct^{-\theta}e^{-ct/\varepsilon^2}.
\]
Finally, equip
\[
\mc Z_\theta^\varepsilon
 :=
 \mathcal X_\theta^\varepsilon\cap H^{2\theta}(Q)\cap\mathcal X
\]
with the intersection norm
\[
\|z\|_{\mc Z_\theta^\varepsilon}
 :=
 \|z\|_{\mathcal X_\theta^\varepsilon}
 +
 \|z\|_{H^{2\theta}(Q)} .
\]
 Hence, for \(u\in\mathcal X\),
the previous two estimates imply
\[
\begin{aligned}
\|e^{tL_\varepsilon}u\|_{\mc Z_\theta^\varepsilon}
&\le
C_\theta\varepsilon^{2\theta}t^{-\theta}e^{-ct/\varepsilon^2}
\|u\|_{\mathcal X}
+
C_\theta t^{-\theta}e^{-ct/\varepsilon^2}
\|u\|_{\mathcal X}  \\
&\le
C_\theta(1+\varepsilon^{2\theta})
t^{-\theta}e^{-ct/\varepsilon^2}
\|u\|_{\mathcal X}.
\end{aligned}
\]
Since \(0<\varepsilon\le1\), we have \(1+\varepsilon^{2\theta}\le2\). Absorbing
this factor into \(C_\theta\) yields
\[
\|e^{tL_\varepsilon}\|_{\mathcal B(\mathcal X,\mc Z_\theta^\varepsilon)}
\le
C_\theta t^{-\theta}e^{-ct/\varepsilon^2},
\qquad t>0.
\]
\end{proof}

\begin{remark}[Role of the parameter-elliptic estimate]
The estimate \eqref{eq:uniform-parameter-elliptic-estimate} is the analytic
input which allows the Lyapunov-Perron argument to be run in the
trace-compatible space \(\mc Z_\theta^\varepsilon\). In the flat case it follows
directly by expanding in the Neumann eigenbasis. For nonconstant \(g\), it is a
standard uniform parameter-elliptic estimate for the smooth conormal problem on
the fixed strip, with the nonlocal projection term
\(-\mathcal E(R_gz)\) treated as part of the projected realization on
\(\ker M\).
\end{remark}

\begin{lemma}[Nonlinear estimates in the trace-compatible topology]
\label{lem:thin-domain-nonlinear-estimates}
Assume that \(f\in C_b^2(\mathbb R)\). More generally, for polynomial or only
locally smooth nonlinearities, assume that \(f\) has been replaced by a smooth
cutoff on a bounded absorbing set. Then, for
\[
\theta\in(3/4,1),
\qquad
0<\delta<\theta-\frac34,
\]
the maps
\[
F_\varepsilon:\mc Z_\theta^\varepsilon\times\mathcal Y_1\to\mathcal X,
\qquad
G_\varepsilon:\mc Z_\theta^\varepsilon\times\mathcal Y_1\to\mathcal Y_\delta
\]
are Lipschitz on bounded sets, with Lipschitz constants independent of
sufficiently small \(\varepsilon>0\). If, in addition, the system is translated
so that \(f(0)=0\), then
\[
F_\varepsilon(0,0)=0,
\qquad
G_\varepsilon(0,0)=0.
\]
\end{lemma}

\begin{proof}
For \(G_\varepsilon\), write
\[
G_\varepsilon(z,v)
=
R_gz+R_g\Lambda_\varepsilon^1v
+
Mf(v+z+\Lambda_\varepsilon^1v).
\]
The trace estimate gives
\[
\|R_g(z_1-z_2)\|_{\mathcal Y_\delta}
\le
C\|z_1-z_2\|_{\mc Z_\theta^\varepsilon}.
\]
The lifting estimate gives
\[
\|R_g\Lambda_\varepsilon^1(v_1-v_2)\|_{\mathcal Y_\delta}
\le
C\varepsilon^2\|v_1-v_2\|_{\mathcal Y_1}.
\]
Since \(2\theta>1\), \(H^{2\theta}(Q)\) is an algebra in two dimensions. Since
\(\mathcal Y_1=H^2_N(\Omega)\hookrightarrow C^1(\overline\Omega)\), the
extension \(v\mapsto\mathcal Ev\) has the required Sobolev regularity on \(Q\).
The assumption \(f\in C_b^2\), or the smooth cutoff, gives the standard
Nemytskii estimate, see for example \cite{runst2011sobolev} for details,
\[
\|f(U_1)-f(U_2)\|_{H^{2\delta}(Q)}
\le
C_R\|U_1-U_2\|_{H^{2\theta}(Q)}
\]
on bounded subsets, where
\[
U_i=v_i+z_i+\Lambda_\varepsilon^1v_i.
\]
Averaging in \(Y\) is bounded from \(H^{2\delta}(Q)\) to
\(H^{2\delta}(\Omega)\). Therefore
\[
\|G_\varepsilon(z_1,v_1)-G_\varepsilon(z_2,v_2)\|_{\mathcal Y_\delta}
\le
C_R
\bigl(
\|z_1-z_2\|_{\mc Z_\theta^\varepsilon}
+
\|v_1-v_2\|_{\mathcal Y_1}
\bigr).
\]

For \(F_\varepsilon\), the Nemytskii term satisfies
\[
\|(I-M)f(U_1)-(I-M)f(U_2)\|_{\mathcal X}
\le
C_R
\bigl(
\|z_1-z_2\|_{\mc Z_\theta^\varepsilon}
+
\|v_1-v_2\|_{\mathcal Y_1}
\bigr).
\]
The strong lifting estimates imply
\[
\left\|
\frac1{\varepsilon^2g^2}\partial_Y^2\Lambda_\varepsilon^1(v_1-v_2)
+
\mathcal D^2\Lambda_\varepsilon^1(v_1-v_2)
-
\mathcal E(R_g\Lambda_\varepsilon^1(v_1-v_2))
\right\|_{L^2(Q)}
\le
C\|v_1-v_2\|_{\mathcal Y_1}.
\]
Also,
\[
\left\|
\mathcal E\!\left(\frac{g'}g\partial_x(v_1-v_2)\right)
\right\|_{L^2(Q)}
\le
C\|v_1-v_2\|_{\mathcal Y_1}.
\]
Finally, the weak lifting estimate gives
\begin{align*}
&\left\|
\Lambda_\varepsilon^0
\bigl(
B_g(v_1-v_2)
+
G_\varepsilon(z_1,v_1)-G_\varepsilon(z_2,v_2)
\bigr)
\right\|_{L^2}
\\
&\qquad\le
C\varepsilon^2
\left(
\|B_g(v_1-v_2)\|_{\mathcal Y}
+
\|G_\varepsilon(z_1,v_1)-G_\varepsilon(z_2,v_2)\|_{\mathcal Y}
\right)
\\
&\qquad\le
C_R
\bigl(
\|z_1-z_2\|_{\mc Z_\theta^\varepsilon}
+
\|v_1-v_2\|_{\mathcal Y_1}
\bigr),
\end{align*}
because \(\mathcal Y_\delta\hookrightarrow\mathcal Y\). Combining the estimates
proves the Lipschitz bound for \(F_\varepsilon\).

If \(f(0)=0\), then
\[
\Lambda_\varepsilon^1 0=0,
\qquad
B_g0=0,
\qquad
G_\varepsilon(0,0)=0,
\]
and the formula for \(F_\varepsilon\) gives \(F_\varepsilon(0,0)=0\).
\end{proof}

\subsubsection*{Spectral splitting of the slow variable}
\label{subsubsec:thin-domain-slow-spectral-splitting}

It remains to verify the splitting assumption for the slow space. Since \(B_g\)
is a regular self-adjoint Sturm-Liouville operator on
\(\mathcal Y=L^2(\Omega,g\,dx)\), the eigenvalues of \(-B_g\) satisfy
\[
0=\lambda_0<\lambda_1<\lambda_2<\cdots,
\qquad
\lambda_k\sim\left(\frac{k\pi}{L_g}\right)^2
\]
for an effective length \(L_g>0\). In particular,
\[
\lambda_k-\lambda_{k-1}\sim Ck.
\]
Let \(\{\varphi_k\}_{k\ge0}\) be the corresponding orthonormal eigenbasis of
\(\mathcal Y\). For \(\rho\ge0\), the interpolation spaces of \(B_g\) may be
characterized spectrally by
\[
\|y\|_{\mathcal Y_\rho}^2
\sim
\sum_{k=0}^\infty (1+\lambda_k)^{2\rho}|y_k|^2,
\qquad
y=\sum_{k=0}^\infty y_k\varphi_k.
\]

Set
\[
\zeta=\varepsilon^2,
\qquad
\alpha_\zeta:=|\zeta^{-1}\omega_A|,
\]
where \(\omega_A<0\) is the uniform exponential rate from the fast semigroup
estimate. Choose \(k_0=k_0(\zeta)\) such that
\[
\lambda_{k_0}<\alpha_\zeta\le\lambda_{k_0+1}.
\]
For sufficiently small \(\varepsilon\), \(k_0\ge1\). Define
\[
\mathcal Y_S^\zeta
:=
\operatorname{span}\{\varphi_0,\ldots,\varphi_{k_0-1}\},
\]
and
\[
\mathcal Y_F^\zeta
:=
\overline{\operatorname{span}\{\varphi_k:k\ge k_0\}}^{\mathcal Y}.
\]
The projections onto these spaces commute with \(B_g\). The space
\(\mathcal Y_S^\zeta\) is finite-dimensional and therefore supports a backward
group, while
\[
0\in\rho(B_g|_{\mathcal Y_F^\zeta})
\]
because all eigenvalues of \(B_g\) on \(\mathcal Y_F^\zeta\) are bounded away
from zero.

\begin{lemma}[Slow spectral splitting estimates]
\label{lem:slow-spectral-splitting-estimates}
With the above definitions, Assumption~4.6 holds for the slow operator \(B_g\)
with
\[
N_F^\zeta:=\alpha_\zeta-\lambda_{k_0},
\qquad
N_S^\zeta:=\alpha_\zeta-\lambda_{k_0-1}.
\]
In particular,
\[
0\le N_F^\zeta<N_S^\zeta\le\alpha_\zeta,
\]
and
\[
N_S^\zeta-N_F^\zeta
=
\lambda_{k_0}-\lambda_{k_0-1}.
\]
\end{lemma}

\begin{proof}
For \(y_F=\sum_{k\ge k_0}y_k\varphi_k\), the forward semigroup satisfies
\[
e^{tB_g}y_F
=
\sum_{k\ge k_0}e^{-\lambda_kt}y_k\varphi_k.
\]
Using the spectral characterization of \(\mathcal Y_1\) and
\(\mathcal Y_\delta\), and the elementary bound
\[
(1+\lambda)^{1-\delta}e^{-\lambda t}
\le
Ct^{\delta-1}e^{-\lambda_{k_0}t}
\qquad
\text{for }\lambda\ge\lambda_{k_0},\ t>0,
\]
we obtain
\[
\|e^{tB_g}y_F\|_{\mathcal Y_1}
\le
Ct^{\delta-1}e^{-\lambda_{k_0}t}\|y_F\|_{\mathcal Y_\delta}.
\]
Since
\[
-\lambda_{k_0}
=
N_F^\zeta+\zeta^{-1}\omega_A,
\]
this is the forward estimate required for the fast slow-variable block.

For \(y_S=\sum_{k=0}^{k_0-1}y_k\varphi_k\), the backward group satisfies
\[
e^{-tB_g}y_S
=
\sum_{k=0}^{k_0-1}e^{\lambda_kt}y_k\varphi_k.
\]
Similarly,
\[
(1+\lambda)^{1-\delta}e^{\lambda t}
\le
Ct^{\delta-1}e^{\lambda_{k_0-1}t}
\qquad
\text{for }0\le\lambda\le\lambda_{k_0-1},\ t>0.
\]
Hence
\[
\|e^{-tB_g}y_S\|_{\mathcal Y_1}
\le
Ct^{\delta-1}e^{\lambda_{k_0-1}t}\|y_S\|_{\mathcal Y_\delta}.
\]
Since
\[
\lambda_{k_0-1}
=
-\bigl(N_S^\zeta+\zeta^{-1}\omega_A\bigr),
\]
this is the backward estimate required for the retained slow block.

The inequalities
\[
0\le N_F^\zeta<N_S^\zeta\le\alpha_\zeta
\]
follow from
\[
\lambda_{k_0-1}<\lambda_{k_0}<\alpha_\zeta.
\]
Finally,
\[
N_S^\zeta-N_F^\zeta
=
(\alpha_\zeta-\lambda_{k_0-1})
-
(\alpha_\zeta-\lambda_{k_0})
=
\lambda_{k_0}-\lambda_{k_0-1}.
\]
\end{proof}
Since
\[
\lambda_k\sim Ck^2,
\qquad
\lambda_k-\lambda_{k-1}\sim Ck,
\]
and
\[
\alpha_\zeta=O(\zeta^{-1})=O(\varepsilon^{-2}),
\]
we have
\[
k_0\sim C\zeta^{-1/2}=C\varepsilon^{-1}.
\]
Moreover,
\[
0\le N_F^\zeta
=
\alpha_\zeta-\lambda_{k_0}
\le
\lambda_{k_0+1}-\lambda_{k_0}
=
O(k_0)
=
O(\varepsilon^{-1}),
\]
and
\[
N_S^\zeta
=
N_F^\zeta+\lambda_{k_0}-\lambda_{k_0-1}
=
O(\varepsilon^{-1}).
\]
Therefore
\[
N_S^\zeta-N_F^\zeta
=
\lambda_{k_0}-\lambda_{k_0-1}
\sim C\varepsilon^{-1},
\qquad
N_S^\zeta+N_F^\zeta=O(\varepsilon^{-1}).
\]
For \(\zeta=\varepsilon^2\), the first denominator in the spectral-gap
condition becomes
\[
\bigl(\varepsilon^2(N_S^\zeta+N_F^\zeta)\bigr)^{1-\theta}.
\]
Consequently the first term in \(L_{\mathrm{spec}}\) satisfies
\[
O\!\left(
\frac{\varepsilon^2}
     {(\varepsilon^2\varepsilon^{-1})^{1-\theta}}
\right)
=
O(\varepsilon^{1+\theta}),
\]
while the second term satisfies
\[
(N_S^\zeta-N_F^\zeta)^{-\delta}=O(\varepsilon^\delta).
\]
Thus
\[
L_{\mathrm{spec}}
=
O(\varepsilon^{1+\theta})+O(\varepsilon^\delta)
\longrightarrow0
\qquad\text{as }\varepsilon\to0.
\]

\subsubsection*{Lyapunov-Perron conclusion}
\label{subsubsec:thin-domain-LP-conclusion}

\begin{proposition}[Trace-compatible Lyapunov-Perron variant]
\label{prop:trace-compatible-LP-replacement}
Assume that the abstract hypotheses of \Cref{Sec:assumptions} hold with
\(\mc X_\theta^\varepsilon\) replaced everywhere by a Banach space
\(\mc Z_\theta^\varepsilon\subset\mc X\). In particular, assume that
\[
\|e^{t\varepsilon^{-2}A_\varepsilon}\|_{\mathcal B(\mc X,\mc Z_\theta^\varepsilon)}
\le
C t^{-\theta}e^{-ct/\varepsilon^2},
\qquad t>0,
\]
and that \(F_\varepsilon\) and \(G_\varepsilon\) are Lipschitz with respect to
the \(\mc Z_\theta^\varepsilon\times\mc Y_1\)-topology with constants compatible
with the spectral-gap condition. Then the conclusion of
\Cref{lemma: lyapunov perron} remains valid with
\(\mc Z_\theta^\varepsilon\) replacing \(\mc X_\theta^\varepsilon\).
\end{proposition}

\begin{proof}
The proof is the same Lyapunov-Perron contraction argument as in
\Cref{lemma: lyapunov perron}. The only estimates in the fast component used
there are the smoothing bound from \(\mc X\) into the chosen fast topology and
the Lipschitz estimates for the nonlinearities in that topology. Replacing
\(\mc X_\theta^\varepsilon\) by \(\mc Z_\theta^\varepsilon\) therefore leaves the
contraction estimates unchanged.
\end{proof}

We can now state the precise conclusion for the thin-domain problem.

\begin{theorem}[Applicability of the splitting-based slow-manifold theorem]
\label{thm:thin-domain-slow-manifold-application}
Assume
\[
\Omega=(0,L),
\qquad
g\in C^3(\overline\Omega),
\qquad
0<g_*\le g\le g^*,
\qquad
g'(0)=g'(L)=0.
\]
Let
\[
\theta\in(3/4,1),
\qquad
0<\delta<\theta-\frac34,
\]
and let \(f\in C_b^2(\mathbb R)\), or let \(f\) be smoothly cut off on a fixed
bounded absorbing set. 
By \Cref{prop:uniform-H2-projected-fast-operator}, the
uniform parameter-elliptic estimate
\[
\|u\|_{H^2(Q)}
\le
C\bigl(\|L_\varepsilon u\|_{\mc X}+\|u\|_{\mc X}\bigr),
\qquad
u\in D(L_\varepsilon),
\]
holds for sufficiently small \(\varepsilon>0\).
Then the lifted mild system
\[
z^\varepsilon(t)
=
e^{t\varepsilon^{-2}A_\varepsilon}z_0^\varepsilon
+
\int_0^t
e^{(t-s)\varepsilon^{-2}A_\varepsilon}
F_\varepsilon(z^\varepsilon(s),v^\varepsilon(s))\,ds,
\]
\[
v^\varepsilon(t)
=
e^{tB_g}v_0^\varepsilon
+
\int_0^t
e^{(t-s)B_g}
G_\varepsilon(z^\varepsilon(s),v^\varepsilon(s))\,ds,
\]
satisfies the trace-compatible version of the abstract assumptions from
Section~4.1, with fast topology \(\mc Z_\theta^\varepsilon\) in place of
\(\mathcal X_\theta^\varepsilon\). Therefore, for all sufficiently small
\(\varepsilon>0\), the splitting-based Lyapunov-Perron construction applies to
the variables
\[
(z^\varepsilon,v^\varepsilon)
=
(w^\varepsilon-\Lambda_\varepsilon^1v^\varepsilon,\ v^\varepsilon).
\]
Equivalently, after the spectral splitting
\[
\mathcal Y
=
\mathcal Y_F^\zeta\oplus\mathcal Y_S^\zeta,
\qquad
\zeta=\varepsilon^2,
\]
there exists a locally exponentially attracting invariant slow manifold of the
form
\[
\mathcal S^{\varepsilon,\zeta}
=
\left\{
\bigl(
h_z^{\varepsilon,\zeta}(v_S),
h_{v_F}^{\varepsilon,\zeta}(v_S),
v_S
\bigr):
v_S\in \mathcal Y_S^\zeta\cap\mathcal Y_1
\right\}.
\]
The reduced dynamics on this manifold are
\[
\partial_tv_S^\varepsilon
=
B_gv_S^\varepsilon
+
\operatorname{pr}_{\mathcal Y_S^\zeta}
G_\varepsilon
\bigl(
h_z^{\varepsilon,\zeta}(v_S^\varepsilon),
h_{v_F}^{\varepsilon,\zeta}(v_S^\varepsilon)+v_S^\varepsilon
\bigr),
\]
understood in the finite-dimensional classical sense on
\(\mathcal Y_S^\zeta\).

The critical manifold is
\[
\mathcal S^0
=
\{(0,v):v\in\mathcal Y_1\}.
\]
The slow manifold \(\mathcal S^{\varepsilon,\zeta}\) has the same qualitative
properties as in Theorem~4.9: Lipschitz regularity, local exponential
attraction, and convergence of the reduced semiflow to the limiting semiflow on
\(\mathcal S^0\), with constants understood in the trace-compatible topology.
\end{theorem}

\begin{proof}
Lemma~\ref{lem:fast-operator-thin-domain} gives the sectoriality, analytic
semigroup generation, and uniform exponential stability of \(A_\varepsilon\) on
the mean-zero fast space. Lemma~\ref{lem:trace-compatible-fast-smoothing}
upgrades the fast semigroup estimates to the trace-compatible space
\(\mc Z_\theta^\varepsilon\). The operator \(B_g\) is a regular self-adjoint
Sturm-Liouville operator on \(\mathcal Y=L^2(\Omega,g\,dx)\), hence it is
sectorial and generates an analytic \(C_0\)-semigroup.

Lemma~\ref{lem:strong-weak-conormal-lifting} supplies the strong and weak
liftings needed to define the homogeneous fast variable and the mild lifted
system. Proposition~\ref{prop:mild-lifted-fast-slow-system} shows that the
lifted equations are well-defined in mild form, so no pointwise-in-time
interpretation of
\[
\Lambda_\varepsilon^0(\partial_tv^\varepsilon)
\]
is required. Lemma~\ref{lem:thin-domain-nonlinear-estimates} verifies the
uniform Lipschitz estimates for \(F_\varepsilon\) and \(G_\varepsilon\) in the
trace-compatible topology. Finally,
Lemma~\ref{lem:slow-spectral-splitting-estimates} verifies the slow spectral
splitting, and the calculation above shows that
\[
L_{\mathrm{spec}}\to0
\qquad\text{as }\varepsilon\to0.
\]
Thus the Lyapunov-Perron map is a contraction for all sufficiently small
\(\varepsilon>0\), and the conclusion follows from the trace-compatible version
of the abstract slow-manifold theorem.
\end{proof}

\begin{remark}[Higher-dimensional base domains]
The splitting-based Lyapunov-Perron construction above uses a growing spectral
gap in the tail of the slow operator. In one spatial base dimension, the
Sturm-Liouville eigenvalue gaps grow linearly with the mode number, which is
what was used above. For bounded base domains
\(\Omega\subset\mathbb R^2\), such a growing ordered gap is not available in
general. In special arithmetic geometries, for example rectangles with rational
squared aspect ratio, one may obtain growing gaps between distinct eigenvalue
clusters. For generic planar domains, however, one should not expect the
ordered spectral gaps of the Laplacian to grow in the way required by the
spectral gap condition. Thus the exact splitting-based construction above is
primarily a one-dimensional-base result unless additional spectral structure is
available.
\end{remark}

\section{Reduced dynamics and asymptotic expansion}\label{sec:reduced}

Having constructed the slow manifold, we next describe it asymptotically and derive the induced reduced dynamics. 
This is the point where the geometric information from \Cref{sec:general-thin} is combined with the dynamical information from \Cref{sec:manifolds}.

Throughout this section we use the split form of the system
\[
\begin{aligned}
\partial_t w^\varepsilon
&=
\varepsilon^{-2}A_\varepsilon w^\varepsilon
+
F(w^\varepsilon,v_F^\varepsilon,v_S^\varepsilon),\\
\partial_t v_F^\varepsilon
&=
Bv_F^\varepsilon
+
\pr_{\mc Y_F^\zeta}G(w^\varepsilon,v_F^\varepsilon,v_S^\varepsilon),\\
\partial_t v_S^\varepsilon
&=
Bv_S^\varepsilon
+
\pr_{\mc Y_S^\zeta}G(w^\varepsilon,v_F^\varepsilon,v_S^\varepsilon).
\end{aligned}
\]
Here, by a slight abuse of notation, the split nonlinearities are defined by
\[
F(w,v_F,v_S):=F(w,v_F+v_S),
\qquad
G(w,v_F,v_S):=G(w,v_F+v_S).
\]
We set $\zeta=\varepsilon^2$ from now on. 
Accordingly, the splitting spaces and projections depend on
\(\varepsilon\). To keep the notation readable we write
\[
\mc Y_F^\varepsilon:=\mc Y_F^{\varepsilon^2},
\qquad
\mc Y_S^\varepsilon:=\mc Y_S^{\varepsilon^2},
\qquad
\pr_{\mc Y_F}:=\pr_{\mc Y_F^\varepsilon},
\qquad
\pr_{\mc Y_S}:=\pr_{\mc Y_S^\varepsilon}.
\]
The dependence on \(\varepsilon\) should be kept in mind throughout this
section.

On the fast tail \(\mc Y_F^\varepsilon\) the operator \(B\) has spectrum of
size \(O(\varepsilon^{-2})\). We therefore write
\[
B|_{\mc Y_F^\varepsilon}
=
\varepsilon^{-2}\widetilde B_\varepsilon,
\qquad
\widetilde B_\varepsilon
:=
\varepsilon^2 B|_{\mc Y_F^\varepsilon}.
\]
Thus the \(v_F\)-equation takes the rescaled form
\begin{equation}\label{eq:rescaled-fast-tail}
\partial_t v_F^\varepsilon
=
\varepsilon^{-2}\widetilde B_\varepsilon v_F^\varepsilon
+
\pr_{\mc Y_F}G(w^\varepsilon,v_F^\varepsilon,v_S^\varepsilon).
\end{equation}
In the spectral splitting used in \Cref{subsec:thin-domain-application},
\(\widetilde B_\varepsilon\) has eigenvalues of order one on
\(\mc Y_F^\varepsilon\). We assume below that
\[
0\in \rho(\widetilde B_\varepsilon)
\quad\text{and}\quad
\sup_{0<\varepsilon\le \varepsilon_0}
\|\widetilde B_\varepsilon^{-1}\|_
{\mathcal B(\mc Y_F^\varepsilon\cap\mc Y_\delta,
             \mc Y_F^\varepsilon\cap\mc Y_1)}
<\infty .
\]
This is the natural rescaled version of invertibility of \(B\) on the fast
tail.

\begin{remark}[Scaling example]
    This scaling becomes particularly transparent in the thin-domain example for a flat boundary $g=g_0$.
Here, \(\mc Y_F^\varepsilon := \overline{\operatorname{span}\{\phi_k:k\ge k_0\}}^{\,L^2(\Omega)}\), \(\phi_k(x)=\cos\!\left(\frac{k\pi x}{L}\right) \), where \(k_0\) is chosen so that \(\left(\frac{k_0\pi}{L}\right)^2\sim \varepsilon^{-2}\). Thus \(k_0\sim C\varepsilon^{-1}\), and \(\partial_x^2\phi_{k_0}\sim -\varepsilon^{-2}\phi_{k_0}\).
\end{remark}

\begin{remark}[Meaning of the operator expansion]
\label{rem:formal-operator-expansion}
In the asymptotic calculations in this section the expression
\[
A_\varepsilon=A_0+\varepsilon^2A_1
\]
should not be understood as an identity between closed operators on a common \(\varepsilon\)-independent domain. 
Rather, it is a compact notation for the scaling decomposition of the fast generator
\[
\varepsilon^{-2}A_\varepsilon
=
\varepsilon^{-2}A_0+ A_{1,\varepsilon}.
\]
Here \(A_0\) denotes the principal transverse operator, while \(A_{1,\varepsilon}\) collects the order-one tangential, geometric, and projection terms. 

\end{remark}

\subsection{Asymptotics}

Using the definition of the slow manifold
\[
S^\varepsilon
=
\{
(h_w^\varepsilon(v_S),h_{v_F}^\varepsilon(v_S),v_S):
v_S\in \mc Y_1\cap \mc Y_S^\varepsilon
\},
\]
together with \eqref{eq: slow dynamics} and
\eqref{eq:rescaled-fast-tail}, we obtain the invariance equations
\begin{align}
\label{eq:invar-w-corrected}
Dh_w^\varepsilon(v_S)
\Big[
Bv_S
+
\pr_{\mc Y_S}G\bigl(
h_w^\varepsilon(v_S),
h_{v_F}^\varepsilon(v_S),
v_S
\bigr)
\Big]
&=
\varepsilon^{-2}A_\varepsilon h_w^\varepsilon(v_S)
+
F\bigl(
h_w^\varepsilon(v_S),
h_{v_F}^\varepsilon(v_S),
v_S
\bigr),
\\
\label{eq:invar-vF-corrected}
Dh_{v_F}^\varepsilon(v_S)
\Big[
Bv_S
+
\pr_{\mc Y_S}G\bigl(
h_w^\varepsilon(v_S),
h_{v_F}^\varepsilon(v_S),
v_S
\bigr)
\Big]
&=
\varepsilon^{-2}\widetilde B_\varepsilon h_{v_F}^\varepsilon(v_S)
+
\pr_{\mc Y_F}G\bigl(
h_w^\varepsilon(v_S),
h_{v_F}^\varepsilon(v_S),
v_S
\bigr).
\end{align}
These identities are the starting point for the asymptotic expansion.

\begin{proposition}[Formal asymptotic expansion of the slow manifold]
\label{prop: asymptotic exp}
Assume \Cref{ass: operators}, \Cref{ass: nonlinearities} and \Cref{ass: splitting slow comp}.
Assume further that the fast generator admits the formal scaling
decomposition
\[
\varepsilon^{-2}A_\varepsilon
=
\varepsilon^{-2}A_0+ A_{1,\varepsilon},
\]
in the sense described in \Cref{rem:formal-operator-expansion}, where
\(A_0\) is invertible on the fast space. Assume also that the rescaled
fast-tail operator
\[
\widetilde B_\varepsilon
:=
\varepsilon^2B|_{\mathcal Y_F^\varepsilon}
\]
is invertible on \(\mathcal Y_F^\varepsilon\). Finally, assume that \(F\)
and \(G\) are of class \(C^2\) near \((0,0,v_S)\).

We seek formal asymptotic expansions of the graph components in the form
\begin{align}
h_w^\varepsilon(v_S)
&\sim
\varepsilon^2h_{w,1}^\varepsilon(v_S)
+
\varepsilon^4h_{w,2}^\varepsilon(v_S)
+\cdots,
\\
h_{v_F}^\varepsilon(v_S)
&\sim
\varepsilon^2h_{v_F,1}^\varepsilon(v_S)
+
\varepsilon^4h_{v_F,2}^\varepsilon(v_S)
+\cdots .
\end{align}
The coefficients are allowed to depend on \(\varepsilon\) through the
\(\varepsilon\)-dependent splitting
\(\mc Y=\mc Y_F^\varepsilon\oplus\mc Y_S^\varepsilon\).

Set
\[
F_0(v_S):=F(0,0,v_S),
\qquad
G_0(v_S):=G(0,0,v_S).
\]
Then the coefficients satisfy the following recursive relations.

\begin{enumerate}[label=\textup{(\roman*)},leftmargin=2.5em]

\item The first nontrivial terms are determined by
\begin{align}
\label{eq:first-order-hw-corrected}
A_0h_{w,1}^\varepsilon(v_S)+F_0(v_S)&=0,
\\
\label{eq:first-order-hF-corrected}
\widetilde B_\varepsilon h_{v_F,1}^\varepsilon(v_S)
+
\pr_{\mc Y_F}G_0(v_S)&=0.
\end{align}

\item Let \(D_1\) and \(D_2\) denote derivatives with respect to the first
and second graph variables, evaluated at \((0,0,v_S)\). Then the second-order
terms are determined by
\begin{align}
\label{eq:second-order-hw-corrected}
\begin{split}
0={}&
A_0h_{w,2}^\varepsilon(v_S)
+
A_1h_{w,1}^\varepsilon(v_S)
+
D_1F_0(v_S)h_{w,1}^\varepsilon(v_S)
+
D_2F_0(v_S)h_{v_F,1}^\varepsilon(v_S)
\\
&\quad
-
Dh_{w,1}^\varepsilon(v_S)
\bigl[
Bv_S+\pr_{\mc Y_S}G_0(v_S)
\bigr],
\end{split}
\\
\label{eq:second-order-hF-corrected}
\begin{split}
0={}&
\widetilde B_\varepsilon h_{v_F,2}^\varepsilon(v_S)
+
\pr_{\mc Y_F}
\Big[
D_1G_0(v_S)h_{w,1}^\varepsilon(v_S)
+
D_2G_0(v_S)h_{v_F,1}^\varepsilon(v_S)
\Big]
\\
&\quad
-
Dh_{v_F,1}^\varepsilon(v_S)
\bigl[
Bv_S+\pr_{\mc Y_S}G_0(v_S)
\bigr].
\end{split}
\end{align}

\end{enumerate}
\end{proposition}

\begin{proof}
The proof is a formal coefficient comparison.
Insert the expansions of \(h_w^\varepsilon\) and
\(h_{v_F}^\varepsilon\) into the invariance equations
\eqref{eq:invar-w-corrected}-\eqref{eq:invar-vF-corrected}. By the scaling convention in \Cref{rem:formal-operator-expansion},
we have
\[
\varepsilon^{-2}A_\varepsilon h_w^\varepsilon
=
A_0h_{w,1}^\varepsilon
+
\varepsilon^2
\bigl(
A_0h_{w,2}^\varepsilon+A_1h_{w,1}^\varepsilon
\bigr)
+
O(\varepsilon^4).
\]
Similarly,
\[
\varepsilon^{-2}\widetilde B_\varepsilon h_{v_F}^\varepsilon
=
\widetilde B_\varepsilon h_{v_F,1}^\varepsilon
+
\varepsilon^2\widetilde B_\varepsilon h_{v_F,2}^\varepsilon
+
O(\varepsilon^4).
\]
Taylor expansion of \(F\) and \(G\) around \((0,0,v_S)\) gives
\[
F(h_w^\varepsilon,h_{v_F}^\varepsilon,v_S)
=
F_0(v_S)
+
\varepsilon^2
\Big[
D_1F_0(v_S)h_{w,1}^\varepsilon(v_S)
+
D_2F_0(v_S)h_{v_F,1}^\varepsilon(v_S)
\Big]
+
O(\varepsilon^4),
\]
and the analogous formula for \(G\). The left-hand sides of the invariance
equations start at order \(\varepsilon^2\), because
\(h_w^\varepsilon\) and \(h_{v_F}^\varepsilon\) start at order
\(\varepsilon^2\).

Equating the terms of order \(O(1)\) gives
\eqref{eq:first-order-hw-corrected} and
\eqref{eq:first-order-hF-corrected}. Equating the terms of order
\(O(\varepsilon^2)\) gives
\eqref{eq:second-order-hw-corrected} and
\eqref{eq:second-order-hF-corrected}.
\end{proof}

\begin{corollary}[Asymptotic reduced vector field on the split slow manifold]\label{cor: formal reduced field}
Under the assumptions of \Cref{prop: asymptotic exp},  the reduced dynamics on
\(S^\varepsilon\) formally satisfy
\begin{equation}
\label{eq:asymp-slow-dynamics-corrected}
\partial_t v_S
=
Bv_S
+
\pr_{\mc Y_S}G_0(v_S)
+
\varepsilon^2
\pr_{\mc Y_S}
\Big[
D_1G_0(v_S)h_{w,1}^\varepsilon(v_S)
+
D_2G_0(v_S)h_{v_F,1}^\varepsilon(v_S)
\Big]
+
O(\varepsilon^4).
\end{equation}
\end{corollary}

\begin{proof}
On the slow manifold,
\[
\partial_t v_S
=
Bv_S
+
\pr_{\mc Y_S}
G\bigl(
h_w^\varepsilon(v_S),
h_{v_F}^\varepsilon(v_S),
v_S
\bigr).
\]
Inserting the expansions from \Cref{prop: asymptotic exp} and Taylor expanding
\(G\) around \((0,0,v_S)\) gives the displayed formula.
\end{proof}

\begin{remark}
The zeroth-order reduced dynamics are obtained by restricting the full system
to the critical manifold. The first nontrivial correction records how the fast
variables are slaved to the retained slow modes. In examples with transverse
heterogeneity, this correction is essential for reconstructing the leading
mean-zero transverse profile.
\end{remark}

\subsection{Approximation result}

The preceding expansion is formal. 
It becomes a rigorous approximate slow-manifold construction under a fixed-domain and a uniform-invertibility assumption outlined below.

\begin{theorem}[Construction of an approximate slow manifold]
\label{thm:approximate-slow-manifold}
Assume \Cref{ass: operators}, \Cref{ass: nonlinearities} and \Cref{ass: splitting slow comp}.
Assume, in addition, that \(0<\theta<1\), that
\[
A_\varepsilon=A_0+\varepsilon^2A_1
\quad\text{as operators } \mc X_1\to\mc X,
\]
with
\[
A_0^{-1}\in\mathcal B(\mc X,\mc X_1),
\]
and that
\[
\widetilde B_\varepsilon:=\varepsilon^2B|_{\mc Y_F^\varepsilon}
\]
satisfies
\[
0\in\rho(\widetilde B_\varepsilon),
\qquad
\sup_{0<\varepsilon\le\varepsilon_0}
\|\widetilde B_\varepsilon^{-1}\|_
{\mathcal B(\mc Y_F^\varepsilon\cap\mc Y_\delta,
             \mc Y_F^\varepsilon\cap\mc Y_1)}
<\infty .
\]
Suppose further that \(F\) and \(G\) are of class \(C^{N+1}\), uniformly in a
neighborhood of the origin in
\[
\mc X_\theta^\varepsilon
\times
(\mc Y_F^\varepsilon\cap\mc Y_1)
\times
(\mc Y_S^\varepsilon\cap\mc Y_1),
\]
for some \(N\in\mathbb N\).

Then there are \(r>0\), \(\varepsilon_0>0\), and maps
\[
h_{j,w}^\varepsilon
\in
C^1\bigl(
B_{\mc Y_S^\varepsilon\cap\mc Y_1}(0,r),\mc X_1
\bigr),
\qquad
h_{j,w}^\varepsilon(0)=0,
\qquad
j=1,\dots,N,
\]
and
\[
h_{j,v_F}^\varepsilon
\in
C^1\bigl(
B_{\mc Y_S^\varepsilon\cap\mc Y_1}(0,r),
\mc Y_F^\varepsilon\cap\mc Y_1
\bigr),
\qquad
h_{j,v_F}^\varepsilon(0)=0,
\qquad
j=1,\dots,N,
\]
such that, with
\[
h_{\varepsilon,N,w}^{\rm app}(v)
:=
\sum_{j=1}^N\varepsilon^{2j}h_{j,w}^\varepsilon(v),
\qquad
h_{\varepsilon,N,v_F}^{\rm app}(v)
:=
\sum_{j=1}^N\varepsilon^{2j}h_{j,v_F}^\varepsilon(v),
\]
the graph
\[
\mathcal M_{\varepsilon,N}^{\rm app}
:=
\left\{
\bigl(
h_{\varepsilon,N,w}^{\rm app}(v),
h_{\varepsilon,N,v_F}^{\rm app}(v),
v
\bigr):
v\in B_{\mc Y_S^\varepsilon\cap\mc Y_1}(0,r)
\right\}
\]
has invariance defect of order \(O(\varepsilon^{2N})\).

More precisely, define the two-component defect
\[
\mathcal I_\varepsilon(h_w,h_{v_F})(v)
=
\begin{pmatrix}
\mathcal I_{\varepsilon,w}(h_w,h_{v_F})(v)\\[1mm]
\mathcal I_{\varepsilon,F}(h_w,h_{v_F})(v)
\end{pmatrix},
\]
where
\[
\begin{aligned}
\mathcal I_{\varepsilon,w}(h_w,h_{v_F})(v)
:={}&
Dh_w(v)
\big[
Bv+\pr_{\mc Y_S}G(h_w(v),h_{v_F}(v),v)
\big]
\\
&\quad
-\varepsilon^{-2}A_\varepsilon h_w(v)
-
F(h_w(v),h_{v_F}(v),v),
\end{aligned}
\]
and
\[
\begin{aligned}
\mathcal I_{\varepsilon,F}(h_w,h_{v_F})(v)
:={}&
Dh_{v_F}(v)
\big[
Bv+\pr_{\mc Y_S}G(h_w(v),h_{v_F}(v),v)
\big]
\\
&\quad
-\varepsilon^{-2}\widetilde B_\varepsilon h_{v_F}(v)
-
\pr_{\mc Y_F}G(h_w(v),h_{v_F}(v),v).
\end{aligned}
\]
Then there is a constant \(C>0\), independent of
\(\varepsilon\in(0,\varepsilon_0]\), such that
\[
\sup_{\|v\|_{\mc Y_1}\le r}
\left(
\|\mathcal I_{\varepsilon,w}
(h_{\varepsilon,N,w}^{\rm app},h_{\varepsilon,N,v_F}^{\rm app})(v)\|_{\mc X}
+
\|\mathcal I_{\varepsilon,F}
(h_{\varepsilon,N,w}^{\rm app},h_{\varepsilon,N,v_F}^{\rm app})(v)\|_{\mc Y_\delta}
\right)
\le
C\varepsilon^{2N}.
\]

Let \(\bar v^\varepsilon\) solve the approximate reduced equation
\[
\partial_t\bar v^\varepsilon
=
B\bar v^\varepsilon
+
\pr_{\mc Y_S}
G\bigl(
h_{\varepsilon,N,w}^{\rm app}(\bar v^\varepsilon),
h_{\varepsilon,N,v_F}^{\rm app}(\bar v^\varepsilon),
\bar v^\varepsilon
\bigr),
\qquad
\bar v^\varepsilon(0)=v_{S,0}.
\]
Let
\[
(w^\varepsilon,v_F^\varepsilon,v_S^\varepsilon)
\]
be the solution of the split full system with initial data
\[
w^\varepsilon(0)=w_0,
\qquad
v_F^\varepsilon(0)=v_{F,0},
\qquad
v_S^\varepsilon(0)=v_{S,0}.
\]
Assume that, on \(0\le t\le T\), both the full solution and the reduced
solution remain in the neighborhood on which the above estimates hold. Then
there exists \(C_T>0\), independent of \(\varepsilon\), such that
\[
\begin{aligned}
\sup_{0\le t\le T}
\Big(
&
\|w^\varepsilon(t)
-
h_{\varepsilon,N,w}^{\rm app}(\bar v^\varepsilon(t))\|_{\mc X_\theta^\varepsilon}
\\
&+
\|v_F^\varepsilon(t)
-
h_{\varepsilon,N,v_F}^{\rm app}(\bar v^\varepsilon(t))\|_{\mc Y_1}
+
\|v_S^\varepsilon(t)-\bar v^\varepsilon(t)\|_{\mc Y_1}
\Big)
\\
&\le
C_T
\Big(
\|w_0-h_{\varepsilon,N,w}^{\rm app}(v_{S,0})\|_{\mc X_\theta^\varepsilon}
+
\|v_{F,0}-h_{\varepsilon,N,v_F}^{\rm app}(v_{S,0})\|_{\mc Y_1}
+
\varepsilon^{2N}
\Big).
\end{aligned}
\]
In particular, if the initial data lie on \(\mathcal M_{\varepsilon,N}^{\rm app}\),
then
\[
\operatorname{dist}_{\mc X_\theta^\varepsilon\times\mc Y_1\times\mc Y_1}
\Big(
(w^\varepsilon(t),v_F^\varepsilon(t),v_S^\varepsilon(t)),
\mathcal M_{\varepsilon,N}^{\rm app}
\Big)
\le
C_T\varepsilon^{2N},
\qquad
0\le t\le T.
\]
Thus \(\mathcal M_{\varepsilon,N}^{\rm app}\) is approximately positively
invariant on every fixed finite time interval.
\end{theorem}

\begin{proof}
We construct the graph coefficients recursively by inserting the ansatz
\[
h_{\varepsilon,N,i}^{\rm app}(v)
=
\sum_{j=1}^N\varepsilon^{2j}h_{j,i}^\varepsilon(v),
\qquad
i=w,v_F,
\]
into the defect equation
\[
\mathcal I_\varepsilon(h_w,h_{v_F})=0.
\]
Collecting equal powers of \(\varepsilon^2\) gives, at the first nontrivial
order,
\[
A_0h_{1,w}^\varepsilon(v)+F(0,0,v)=0,
\]
and
\[
\widetilde B_\varepsilon h_{1,v_F}^\varepsilon(v)
+
\pr_{\mc Y_F}G(0,0,v)=0.
\]
These equations have unique solutions because \(A_0\) and
\(\widetilde B_\varepsilon\) are invertible with the stated bounds.

At the next orders one obtains recursive equations of the form
\[
A_0h_{j+1,w}^\varepsilon(v)=R_{j,w}^\varepsilon(v),
\qquad
\widetilde B_\varepsilon h_{j+1,v_F}^\varepsilon(v)
=
R_{j,F}^\varepsilon(v),
\]
where the right-hand sides depend only on the previously constructed
coefficients
\[
h_{1,i}^\varepsilon,\dots,h_{j,i}^\varepsilon,
\qquad i=w,v_F,
\]
and on derivatives of \(F\) and \(G\) of order at most \(j+1\). The uniform
\(C^{N+1}\)-bounds on \(F\) and \(G\), together with the uniform inverse bounds
for \(A_0\) and \(\widetilde B_\varepsilon\), give the stated \(C^1\)-regularity
of the coefficients after shrinking \(r\), if necessary.

By construction, all terms in
\[
\mathcal I_\varepsilon
(h_{\varepsilon,N,w}^{\rm app},
 h_{\varepsilon,N,v_F}^{\rm app})
\]
up to order \(\varepsilon^{2N-2}\) cancel. The Taylor remainder is of order
\(\varepsilon^{2N}\), uniformly for \(\|v\|_{\mc Y_1}\le r\). Hence the
defect estimate follows.

It remains to prove the tracking estimate. Set
\[
z_X^\varepsilon(t)
:=
w^\varepsilon(t)
-
h_{\varepsilon,N,w}^{\rm app}(\bar v^\varepsilon(t)),
\]
\[
z_Y^\varepsilon(t)
:=
v_F^\varepsilon(t)
-
h_{\varepsilon,N,v_F}^{\rm app}(\bar v^\varepsilon(t)),
\]
and
\[
y^\varepsilon(t)
:=
v_S^\varepsilon(t)-\bar v^\varepsilon(t).
\]
Differentiating and using the full split system, the reduced equation, and the
definition of the defect gives
\[
\begin{aligned}
\partial_t z_X^\varepsilon
&=
\varepsilon^{-2}A_\varepsilon z_X^\varepsilon
+
\mathcal R_1^\varepsilon(z_X^\varepsilon,z_Y^\varepsilon,y^\varepsilon,t)
+
\mathcal I_{\varepsilon,w}
(h_{\varepsilon,N,w}^{\rm app},
 h_{\varepsilon,N,v_F}^{\rm app})
(\bar v^\varepsilon(t)),
\\
\partial_t z_Y^\varepsilon
&=
\varepsilon^{-2}\widetilde B_\varepsilon z_Y^\varepsilon
+
\mathcal R_2^\varepsilon(z_X^\varepsilon,z_Y^\varepsilon,y^\varepsilon,t)
+
\mathcal I_{\varepsilon,F}
(h_{\varepsilon,N,w}^{\rm app},
 h_{\varepsilon,N,v_F}^{\rm app})
(\bar v^\varepsilon(t)),
\\
\partial_t y^\varepsilon
&=
By^\varepsilon
+
\mathcal R_3^\varepsilon(z_X^\varepsilon,z_Y^\varepsilon,y^\varepsilon,t).
\end{aligned}
\]
The remainder terms are locally Lipschitz and satisfy
\[
\begin{aligned}
&
\|\mathcal R_1^\varepsilon(z_X,z_Y,y,t)\|_{\mc X}
+
\|\mathcal R_2^\varepsilon(z_X,z_Y,y,t)\|_{\mc Y_\delta}
+
\|\mathcal R_3^\varepsilon(z_X,z_Y,y,t)\|_{\mc Y_\delta}
\\
&\qquad\le
C\Big(
\|z_X\|_{\mc X_\theta^\varepsilon}
+
\|z_Y\|_{\mc Y_1}
+
\|y\|_{\mc Y_1}
\Big)
\end{aligned}
\]
as long as the arguments remain in the chosen local neighborhood.

Writing the mild form and using the analytic semigroup estimates gives
\[
\begin{aligned}
E(t)
\le{}&
C_T
\Big(
E(0)+\varepsilon^{2N}
\Big)
+
C_T
\int_0^t
\Big(
(t-s)^{-\theta}+(t-s)^{\delta-1}+1
\Big)
E(s)\,ds,
\end{aligned}
\]
where
\[
E(t)
:=
\|z_X^\varepsilon(t)\|_{\mc X_\theta^\varepsilon}
+
\|z_Y^\varepsilon(t)\|_{\mc Y_1}
+
\|y^\varepsilon(t)\|_{\mc Y_1}.
\]
Since \(0<\theta<1\) and \(\delta>0\), the kernel is integrable on every finite
interval \([0,T]\). The singular Gronwall inequality therefore yields
\[
\sup_{0\le t\le T}E(t)
\le
C_T
\Big(
E(0)+\varepsilon^{2N}
\Big).
\]
Finally,
\[
E(0)
=
\|w_0-h_{\varepsilon,N,w}^{\rm app}(v_{S,0})\|_{\mc X_\theta^\varepsilon}
+
\|v_{F,0}-h_{\varepsilon,N,v_F}^{\rm app}(v_{S,0})\|_{\mc Y_1},
\]
which gives the claimed estimate. If the initial data lie on the approximate
graph, then \(E(0)=0\), and the approximate positive invariance follows.
\end{proof}

\begin{remark}
It is instructive to compare this construction with the Galerkin approximation in \cite{engel2021connecting}. 
In this work here the approximate slow manifold is obtained directly from the invariance equation by expanding the graph in powers of \(\varepsilon^2\). 
This leads to a recursive family of elliptic or parabolic linear problems for the graph coefficients. 
In the Galerkin approach one first projects the PDE onto finitely many spatial eigenfunctions, constructs a slow manifold for the resulting ODE system, and then lifts the approximation back to the PDE setting.
\end{remark}

\begin{remark}
An approximate slow-manifold construction can also be formulated without
splitting the slow variable. In that case, however, the estimate generally does
not improve beyond order \(O(\varepsilon^2)\), because the unbounded slow
operator \(B\) cannot be inverted backward on the full space \(\mc Y_1\). Thus
the splitting is not merely a technical device for the exact
Lyapunov-Perron construction; it is also what allows higher-order correction
terms to be organized in a stable recursive hierarchy.
\end{remark}

\subsection{Application to the thin-domain system}
\label{subsec:thin-domain-asymptotic-expansion}

We now specialize the asymptotic construction from \Cref{sec:reduced} to the
thin-domain system from \Cref{sec:general-thin} and \Cref{subsec:thin-domain-application}. The purpose of
this subsection is to identify the first formal correction in the variables
that are natural for the thin-domain problem. The rigorous finite-order
tracking theorem is then proved in \Cref{subsec:thin-domain-approximation-theorem}.

The main point is that the expansion must be formulated in the lifted
homogeneous fast variable
\[
z^\varepsilon
=
w^\varepsilon-\Lambda_\varepsilon^1v^\varepsilon,
\]
rather than directly in the fluctuation \(w^\varepsilon=u^\varepsilon-v^\varepsilon\).
Indeed, \(z^\varepsilon\) satisfies the homogeneous conormal boundary
conditions and therefore belongs to the domain of the fast operator
\(A_\varepsilon\). The original fluctuation is recovered afterwards by setting
\[
w^\varepsilon
=
z^\varepsilon+\Lambda_\varepsilon^1v^\varepsilon.
\]

\subsubsection*{Lifted split variables}

We use the notation of \Cref{prop:mild-lifted-fast-slow-system}. Thus
\[
G_\varepsilon(z,v)
=
R_g(z+\Lambda_\varepsilon^1v)
+
Mf(v+z+\Lambda_\varepsilon^1v),
\]
and
\begin{align*}
F_\varepsilon(z,v)
={}&
(I-M)f(v+z+\Lambda_\varepsilon^1v)
+
\frac1{\varepsilon^2g^2}\partial_Y^2\Lambda_\varepsilon^1v
+
\mathcal D^2\Lambda_\varepsilon^1v
\\
&-
\mathcal E(R_g\Lambda_\varepsilon^1v)
-
\mathcal E\!\left(\frac{g'}g\partial_xv\right)
-
\Lambda_\varepsilon^0
\bigl(B_gv+G_\varepsilon(z,v)\bigr).
\end{align*}
After the slow spectral splitting
\[
\mathcal Y
=
\mathcal Y_F^\varepsilon\oplus\mathcal Y_S^\varepsilon,
\qquad
\zeta=\varepsilon^2,
\]
we write
\[
v=v_F+v_S.
\]
Let
\[
\Pi_F^\varepsilon:=\pr_{\mathcal Y_F^\varepsilon},
\qquad
\Pi_S^\varepsilon:=\pr_{\mathcal Y_S^\varepsilon},
\]
and define the rescaled fast operator on the slow-variable tail by
\[
\widetilde B_{g,\varepsilon}
:=
\varepsilon^2B_g|_{\mathcal Y_F^\varepsilon}.
\]
Then the lifted split system is
\begin{align}
\label{eq:thin-lifted-split-system}
\partial_t z^\varepsilon
&=
\varepsilon^{-2}A_\varepsilon z^\varepsilon
+
F_\varepsilon(z^\varepsilon,v_F^\varepsilon+v_S^\varepsilon),
\\
\partial_t v_F^\varepsilon
&=
\varepsilon^{-2}\widetilde B_{g,\varepsilon}v_F^\varepsilon
+
\Pi_F^\varepsilon
G_\varepsilon(z^\varepsilon,v_F^\varepsilon+v_S^\varepsilon),
\\
\partial_t v_S^\varepsilon
&=
B_gv_S^\varepsilon
+
\Pi_S^\varepsilon
G_\varepsilon(z^\varepsilon,v_F^\varepsilon+v_S^\varepsilon).
\end{align}
This is the thin-domain analogue of the abstract split system used in
\Cref{prop: asymptotic exp}. The only difference is that the
fast operator and the nonlinearities still depend on \(\varepsilon\) through
the conormal boundary condition and the liftings.

\subsubsection*{First formal coefficients}

For \(v_S\in\mathcal Y_S^\varepsilon\cap\mathcal Y_1\), set
\[
F_{\varepsilon,0}(v_S)
:=
F_\varepsilon(0,v_S),
\qquad
G_{\varepsilon,0}(v_S)
:=
G_\varepsilon(0,v_S).
\]
Explicitly,
\[
G_{\varepsilon,0}(v_S)
=
R_g\Lambda_\varepsilon^1v_S
+
Mf(v_S+\Lambda_\varepsilon^1v_S),
\]
and
\begin{align}
\label{eq:thin-F-eps-zero}
F_{\varepsilon,0}(v_S)
={}&
(I-M)f(v_S+\Lambda_\varepsilon^1v_S)
+
\frac1{\varepsilon^2g^2}\partial_Y^2\Lambda_\varepsilon^1v_S
+
\mathcal D^2\Lambda_\varepsilon^1v_S
\nonumber\\
&-
\mathcal E(R_g\Lambda_\varepsilon^1v_S)
-
\mathcal E\!\left(\frac{g'}g\partial_xv_S\right)
-
\Lambda_\varepsilon^0
\bigl(B_gv_S+G_{\varepsilon,0}(v_S)\bigr).
\end{align}
By the mean-zero consistency result in
\Cref{subsec:thin-domain-application},
\[
F_{\varepsilon,0}(v_S)\in\mathcal X.
\]
Thus the first transverse coefficient can be obtained by inverting the
vertical limit operator \(A_0\) on the mean-zero space.

\begin{proposition}[First formal thin-domain correction]
\label{prop:first-thin-domain-correction}
Assume the hypotheses of
\Cref{thm:thin-domain-slow-manifold-application}. Suppose, in addition, that
the first-order coefficient expansion of \Cref{prop: asymptotic exp} is
applicable to the lifted system \eqref{eq:thin-lifted-split-system} in the
trace-compatible topology \(\mathcal Z_\theta^\varepsilon\). Then the graph
components have the formal expansions
\[
z
\sim
\varepsilon^2h_{z,1}^\varepsilon(v_S)
+
\varepsilon^4h_{z,2}^\varepsilon(v_S)
+
\cdots,
\]
and
\[
v_F
\sim
\varepsilon^2h_{F,1}^\varepsilon(v_S)
+
\varepsilon^4h_{F,2}^\varepsilon(v_S)
+
\cdots .
\]
The first coefficients are determined by
\begin{equation}
\label{eq:first-thin-z-correction}
A_0h_{z,1}^\varepsilon(v_S)
+
F_{\varepsilon,0}(v_S)
=
0,
\end{equation}
and
\begin{equation}
\label{eq:first-thin-vF-correction}
\widetilde B_{g,\varepsilon}h_{F,1}^\varepsilon(v_S)
+
\Pi_F^\varepsilon G_{\varepsilon,0}(v_S)
=
0.
\end{equation}
Equivalently,
\[
h_{z,1}^\varepsilon(v_S)
=
-A_0^{-1}F_{\varepsilon,0}(v_S),
\qquad
h_{F,1}^\varepsilon(v_S)
=
-
\widetilde B_{g,\varepsilon}^{-1}
\Pi_F^\varepsilon G_{\varepsilon,0}(v_S).
\]
Consequently, the first-order approximate graph in lifted variables is
\[
\mathcal M_{\varepsilon,1}^{\rm td,app}
=
\left\{
\bigl(
\varepsilon^2h_{z,1}^\varepsilon(v_S),
\varepsilon^2h_{F,1}^\varepsilon(v_S),
v_S
\bigr):
v_S\in \mathcal Y_S^\varepsilon\cap\mathcal Y_1
\right\}.
\]
\end{proposition}

\begin{proof}
The lifted system \eqref{eq:thin-lifted-split-system} has the same split
structure as the abstract system in \Cref{prop: asymptotic exp},
with \(w\) replaced by \(z\), \(B\) replaced by \(B_g\), and
\(F,G\) replaced by \(F_\varepsilon,G_\varepsilon\). The rescaled operator on
the fast slow-variable tail is \(\widetilde B_{g,\varepsilon}\). Collecting the
terms of order \(O(1)\) in the two fast invariance equations therefore gives
\eqref{eq:first-thin-z-correction} and
\eqref{eq:first-thin-vF-correction}. The displayed formulas follow from
invertibility of \(A_0\) on \(\mathcal X=\ker M\) and of
\(\widetilde B_{g,\varepsilon}\) on \(\mathcal Y_F^\varepsilon\).
\end{proof}

\subsubsection*{Corrected reduced dynamics and reconstruction}

The corresponding first-order reduced equation on the retained slow modes is
obtained by inserting the first-order approximate graph into the \(v_S\)-equation. For
\(\xi\in\mathcal Z_\theta^\varepsilon\) and \(\eta\in\mathcal Y_1\), the
linearizations of \(G_\varepsilon\) at \((0,v_S)\) are
\[
D_zG_\varepsilon(0,v_S)\xi
=
R_g\xi
+
M\bigl[
f'(v_S+\Lambda_\varepsilon^1v_S)\xi
\bigr],
\]
and
\[
D_vG_\varepsilon(0,v_S)\eta
=
R_g\Lambda_\varepsilon^1\eta
+
M\bigl[
f'(v_S+\Lambda_\varepsilon^1v_S)
(\eta+\Lambda_\varepsilon^1\eta)
\bigr].
\]
Hence the first-order corrected reduced equation is
\begin{align}
\label{eq:thin-corrected-reduced-dynamics}
\partial_tv_S
={}&
B_gv_S
+
\Pi_S^\varepsilon G_{\varepsilon,0}(v_S)
\nonumber\\
&+
\varepsilon^2
\Pi_S^\varepsilon
\Bigl[
D_zG_\varepsilon(0,v_S)h_{z,1}^\varepsilon(v_S)
+
D_vG_\varepsilon(0,v_S)h_{F,1}^\varepsilon(v_S)
\Bigr]
+
O(\varepsilon^4).
\end{align}
Here the remainder is understood in the finite-dimensional space
\(\mathcal Y_S^\varepsilon\), uniformly for \(v_S\) in bounded subsets of
\(\mathcal Y_S^\varepsilon\cap\mathcal Y_1\).

To translate the graph back to the original thin-domain variable, use
\[
w=z+\Lambda_\varepsilon^1(v_S+v_F).
\]
Since
\[
v_F=\varepsilon^2h_{F,1}^\varepsilon(v_S)+O(\varepsilon^4)
\]
and
\[
\|\Lambda_\varepsilon^1\eta\|_{H^2(Q)}
\le
C\varepsilon^2\|\eta\|_{\mathcal Y_1},
\]
the term \(\Lambda_\varepsilon^1v_F\) contributes only at order
\(O(\varepsilon^4)\). Thus the first-order physical reconstruction is
\begin{equation}
\label{eq:thin-domain-first-reconstruction}
u_{\varepsilon}^{\rm app}(x,Y)
=
v_S(x)
+
\varepsilon^2h_{F,1}^\varepsilon(v_S)(x)
+
\Lambda_\varepsilon^1v_S(x,Y)
+
\varepsilon^2h_{z,1}^\varepsilon(v_S)(x,Y)
+
O(\varepsilon^4).
\end{equation}
The two leading transverse contributions are
\[
\Lambda_\varepsilon^1v_S
\qquad\text{and}\qquad
\varepsilon^2h_{z,1}^\varepsilon(v_S).
\]
The first is forced by the conormal boundary mismatch of the fluctuation
\(w=u-v\) and the second is the interior slow-manifold correction generated by the
invariance equation.

\begin{remark}[Flat reference case]
If \(g\) is constant, then \(R_g=0\) and \(\Lambda_\varepsilon^1=0\). Hence
\[
F_{\varepsilon,0}(v_S)
=
(I-M)f(v_S).
\]
If, in addition, the reaction term is independent of the transverse variable
and \(v_S\) is extended constantly in \(Y\), then
\[
(I-M)f(v_S)=0.
\]
Therefore
\[
h_{z,1}^\varepsilon(v_S)=0.
\]
A flat thin domain with a transversely homogeneous reaction term has no first
transverse slow-manifold correction. Nontrivial transverse corrections arise
from variable geometry, encoded by \(R_g\) and \(\Lambda_\varepsilon^1\), or
from a reaction term with genuine \(Y\)-dependence, as in the heterogeneous
Schnakenberg example of \Cref{subsec:transverse-schnakenberg}.
\end{remark}

\subsubsection*{Transition to the finite-order theorem}

The preceding discussion identifies the first correction and explains how the
abstract invariance-equation expansion appears in the thin-domain variables.
It does not by itself prove a finite-time approximation estimate, because the
fast operator is not expanded as a fixed strong operator identity on a common
high-regularity domain. Instead, the thin-domain problem carries an
\(\varepsilon\)-dependent conormal boundary condition.

\Cref{subsec:thin-domain-approximation-theorem} supplies the missing step. It
uses the variational identity
\[
\mathfrak a_\varepsilon
=
\mathfrak a_0+\varepsilon^2\mathfrak a_1
\]
and a finite-order conormal boundary hierarchy to construct corrected graphs
\[
h_{\varepsilon,N,z}^{\rm app}(v_S),
\qquad
h_{\varepsilon,N,F}^{\rm app}(v_S),
\]
with an invariance defect of order \(O(\varepsilon^{2N})\). The theorem then
applies the mild tracking argument in
\(\mathcal Z_\theta^\varepsilon\times\mathcal Y_1\times\mathcal Y_1\) and
finally transfers the estimate back from \(z^\varepsilon\) to the original
fluctuation \(w^\varepsilon\) through
\[
w^\varepsilon
=
z^\varepsilon+\Lambda_\varepsilon^1(v_S^\varepsilon+v_F^\varepsilon).
\]
Thus \Cref{subsec:thin-domain-approximation-theorem} should be read as the
rigorous finite-order version of the formal first-order calculation above.

\subsection{A finite-order thin-domain approximation theorem}
\label{subsec:thin-domain-approximation-theorem}

We now turn the formal expansion from the previous subsection into a finite-time
approximation theorem for the thin-domain system. The key point is that the
expansion of the singular fast operator must be understood through the
variational form introduced in \Cref{subsec:thin-domain-application}. We do not
use a strong decomposition of the conormal operator. Instead, we use the form
identity
\[
\mathfrak a_\varepsilon
=
\mathfrak a_0+\varepsilon^2\mathfrak a_1
\]
and the conormal boundary hierarchy induced by the full
\(\varepsilon\)-dependent boundary condition.

Throughout this subsection we work in the one-dimensional thin-strip setting of
\Cref{subsec:thin-domain-application}. Thus
\[
\Omega=(0,L),
\qquad
Q=\Omega\times(0,1),
\qquad
\mathcal X=\ker M,
\qquad
\mathcal V=H^1(Q)\cap\mathcal X.
\]
The weighted inner product on \(\mathcal X\) is denoted by
\[
(u,\varphi)_g
:=
\int_Q u(x,Y)\varphi(x,Y)g(x)\,dx\,dY.
\]

\subsubsection*{The form expansion and the conormal hierarchy}

Recall that the fast form is
\[
\mathfrak a_\varepsilon(u,\varphi)
=
\int_Q
\frac1{g(x)^2}
\partial_Yu\,\partial_Y\varphi\,g(x)\,dx\,dY
+
\varepsilon^2
\int_Q
\mathcal Du\,\mathcal D\varphi\,g(x)\,dx\,dY.
\]
We set
\[
\mathfrak a_0(u,\varphi)
:=
\int_Q
\frac1{g(x)^2}
\partial_Yu\,\partial_Y\varphi\,g(x)\,dx\,dY
\]
and
\[
\mathfrak a_1(u,\varphi)
:=
\int_Q
\mathcal Du\,\mathcal D\varphi\,g(x)\,dx\,dY.
\]
Let \(A_0\) be the operator associated with \(-\mathfrak a_0\) on the
mean-zero space \(\mathcal X\). Equivalently, \(A_0\) is the vertical Neumann
operator
\[
A_0u=g(x)^{-2}\partial_Y^2u
\]
on each fiber, restricted to the mean-zero subspace. Hence \(A_0\) is
invertible on \(\mathcal X\).

For \(N\in\mathbb N\), set
\[
s_N:=2N+6.
\]
For \(0\le j\le N\), define the anisotropic coefficient spaces
\[
\mathcal X_{N,j}
:=
H^{s_N-2j}\bigl(\Omega;L^2(0,1)\bigr)\cap\mathcal X
\]
and
\[
\mathcal W_{N,j}
:=
H^{s_N-2j}\bigl(\Omega;H^2(0,1)\bigr)\cap\mathcal X.
\]
The loss of two tangential derivatives at each step reflects the fact that the
tangential form \(\mathfrak a_1\) contains the first-order operator
\(\mathcal D\), and its associated distribution contains two tangential
derivatives.

We shall use the following finite-order trace lifting, recorded in
Appendix~A.4. If
\[
\beta\in H^{s_N-2j}(\Omega)
\]
satisfies the endpoint compatibility conditions imposed by the lateral
conormal boundary condition, let
\[
\mathscr B_j\beta\in\mathcal W_{N,j}
\]
denote the chosen mean-zero lifting satisfying
\[
\partial_Y\mathscr B_j\beta=0
\qquad\text{on }Y=0,
\]
\[
\partial_Y\mathscr B_j\beta=\beta
\qquad\text{on }Y=1,
\]
and
\[
\mathcal D\mathscr B_j\beta=0
\qquad\text{on }\partial\Omega\times(0,1).
\]
Moreover,
\[
\|\mathscr B_j\beta\|_{\mathcal W_{N,j}}
\le
C\|\beta\|_{H^{s_N-2j}(\Omega)}.
\]
The endpoint condition \(g'(0)=g'(L)=0\) is used to ensure compatibility at the
corners.

\begin{lemma}[Form expansion with conormal boundary hierarchy]
\label{lem:form-expansion-conormal-hierarchy}
Let \(p_1,\dots,p_N\) be coefficient functions such that
\[
p_j\in\mathcal W_{N,j-1},
\qquad
j=1,\dots,N,
\]
and suppose that they satisfy the conormal hierarchy
\[
\partial_Yp_1=0
\qquad\text{on }Y=\{0,1\},
\]
\[
\partial_Yp_{j+1}=0
\qquad\text{on }Y=0,
\]
\[
\partial_Yp_{j+1}=gg'\mathcal Dp_j
\qquad\text{on }Y=1,
\qquad
j=1,\dots,N-1,
\]
and
\[
\mathcal Dp_j=0
\qquad\text{on }\partial\Omega\times(0,1),
\qquad
j=1,\dots,N.
\]
Set
\[
P_N^\varepsilon
:=
\sum_{j=1}^N\varepsilon^{2j} p_j.
\]
Then the full conormal boundary residual of \(P_N^\varepsilon\) satisfies
\[
\partial_YP_N^\varepsilon-\varepsilon^2 gg'\mathcal DP_N^\varepsilon
=
-\varepsilon^{2(N+1)}gg'\mathcal Dp_N
\qquad\text{on }Y=1,
\]
while
\[
\partial_YP_N^\varepsilon=0
\qquad\text{on }Y=0,
\]
and
\[
\mathcal DP_N^\varepsilon=0
\qquad\text{on }\partial\Omega\times(0,1).
\]
Consequently, there exists a correction
\[
C_N^\varepsilon(P_N^\varepsilon)\in\mathcal W_{N,N}
\]
such that
\[
\widehat P_N^\varepsilon
:=
P_N^\varepsilon+C_N^\varepsilon(P_N^\varepsilon)
\in D(A_\varepsilon),
\]
and
\[
\|C_N^\varepsilon(P_N^\varepsilon)\|_{\mathcal W_{N,N}}
\le
C\varepsilon^{2(N+1)}\|p_N\|_{\mathcal W_{N,N-1}}.
\]
Moreover,
\[
\varepsilon^{-2}A_\varepsilon\widehat P_N^\varepsilon
=
\sum_{m=0}^{N-1}\varepsilon^{2m}\mathcal L_m(p_1,\dots,p_{m+1})
+
\varepsilon^{2N}\mathcal R_{N,A}^\varepsilon,
\]
where
\[
\mathcal L_m(p_1,\dots,p_{m+1})
\in\mathcal X_{N,m+1}
\]
depends only on \(p_1,\dots,p_{m+1}\), and
\[
\|\mathcal R_{N,A}^\varepsilon\|_{\mathcal X}
\le
C\sum_{j=1}^N\|p_j\|_{\mathcal W_{N,j-1}}.
\]
\end{lemma}

\begin{proof}
The boundary residual follows by direct substitution. Indeed,
\[
\begin{aligned}
\partial_YP_N^\varepsilon-\varepsilon^2 gg'\mathcal DP_N^\varepsilon
&=
\sum_{j=1}^N \varepsilon^{2j}\partial_Yp_j
-
\sum_{j=1}^N \varepsilon^{2(j+1)}gg'\mathcal Dp_j .
\end{aligned}
\]
On the boundary at \(Y=1\), the hierarchy gives
\[
\partial_Yp_1=0,
\qquad
\partial_Yp_{j+1}=gg'\mathcal Dp_j.
\]
Hence all intermediate terms cancel telescopically and only the final term
remains:
\[
\partial_YP_N^\varepsilon-\varepsilon^2 gg'\mathcal DP_N^\varepsilon
=
-\varepsilon^{2(N+1)}gg'\mathcal Dp_N.
\]
The lower and lateral boundary conditions follow directly from the assumed
hierarchy.

Let
\[
\beta_N^\varepsilon
:=
\varepsilon^{2(N+1)}gg'\mathcal Dp_N(\cdot,1).
\]
By the trace theorem in the transverse variable, the tangential regularity in
\(\mathcal W_{N,N-1}\), and the fact that \(\mathcal D\) contains at most one
tangential derivative, yields
\[
\|\beta_N^\varepsilon\|_{H^{s_N-2N}(\Omega)}
\le
C\varepsilon^{2(N+1)}\|p_N\|_{\mathcal W_{N,N-1}}.
\]
Define
\[
C_N^\varepsilon(P_N^\varepsilon)
:=
\mathscr B_N\beta_N^\varepsilon.
\]
Then
\[
\widehat P_N^\varepsilon
=
P_N^\varepsilon+C_N^\varepsilon(P_N^\varepsilon)
\]
satisfies the full homogeneous conormal boundary conditions of the fast
operator. Therefore
\[
\widehat P_N^\varepsilon\in D(A_\varepsilon)
\]
by the variational domain identification from
\Cref{subsec:thin-domain-application}.
It remains to identify the expansion. We verify the boundary cancellation
explicitly. For \(\varphi\in\mathcal V\), we compute
\[
\mathfrak a_\varepsilon(P_N^\varepsilon,\varphi)
=
\sum_{j=1}^N
\varepsilon^{2j}
\Bigl[
\mathfrak a_0(p_j,\varphi)
+
\varepsilon^2\,\mathfrak a_1(p_j,\varphi)
\Bigr].
\]
First, integrating \(\mathfrak a_0(p_j,\varphi)\) by parts in \(Y\), and
using
\[
\partial_Yp_j=0
\qquad\text{on }Y=0,
\]
gives
\[
\mathfrak a_0(p_j,\varphi)
=
-
\int_Q
\frac1{g^2}\partial_Y^2p_j\,\varphi\,g\,dx\,dY
+
\int_\Omega
\frac1{g^2}\partial_Yp_j\big|_{Y=1}\,
\varphi(\cdot,1)\,g\,dx.
\]
Thus the upper-boundary contribution from
\(\mathfrak a_0(p_j,\varphi)\), appearing at order \(\varepsilon^{2j}\), is
\[
\int_\Omega
\frac{\partial_Yp_j|_{Y=1}}{g}\,
\varphi(\cdot,1)\,dx.
\]

Next, consider
\[
\mathfrak a_1(p_j,\varphi)
=
\int_Q
\mathcal Dp_j\,\mathcal D\varphi\,g\,dx\,dY,
\qquad
\mathcal D=\partial_x-Yq\partial_Y,
\qquad
q=\frac{g'}{g}.
\]
Using
\[
\mathcal Dp_j=0
\qquad\text{on }\partial\Omega\times(0,1),
\]
and integrating by parts in \(x\) and \(Y\), the upper-boundary contribution
at \(Y=1\) is
\[
-
\int_\Omega
g'(x)\,\mathcal Dp_j\big|_{Y=1}\,
\varphi(\cdot,1)\,dx.
\]
This term is multiplied by \(\varepsilon^{2(j+1)}\) in
\(\mathfrak a_\varepsilon(P_N^\varepsilon,\varphi)\).

We now collect the upper-boundary terms at a fixed power
\(\varepsilon^{2m}\), with \(2\le m\le N\). They are
\[
\int_\Omega
\left[
\frac{\partial_Yp_m|_{Y=1}}{g}
-
g'\,\mathcal Dp_{m-1}\big|_{Y=1}
\right]
\varphi(\cdot,1)\,dx.
\]
By the conormal hierarchy,
\[
\partial_Yp_m\big|_{Y=1}
=
gg'\,\mathcal Dp_{m-1}\big|_{Y=1}.
\]
Therefore
\[
\frac{\partial_Yp_m|_{Y=1}}{g}
=
g'\,\mathcal Dp_{m-1}\big|_{Y=1},
\]
and the bracket vanishes. At order \(\varepsilon^2\), only the boundary term from
\(\mathfrak a_0(p_1,\varphi)\) is present, and this vanishes because
\[
\partial_Yp_1=0
\qquad\text{on }Y=1.
\]
Hence all intermediate upper-boundary contributions cancel telescopically.

The only surviving upper-boundary term is the order-\(\varepsilon^{2(N+1)}\) term coming
from
\[
-\mathfrak a_1(p_N,\varphi),
\]
namely
\[
-\varepsilon^{2(N+1)}
\int_\Omega
g'\,\mathcal Dp_N\big|_{Y=1}\,
\varphi(\cdot,1)\,dx.
\]
This is precisely the boundary residual removed by the correction
\[
C_N^\varepsilon(P_N^\varepsilon).
\]

Consequently, the bulk terms at each order \(\varepsilon^{2(m+1)}\),
\(m=0,\dots,N-1\), are represented by
\[
\mathcal L_m(p_1,\dots,p_{m+1})\in\mathcal X_{N,m+1}.
\]
The operator \(\mathcal L_m\) collects the vertical fiber operator applied to
\(p_{m+1}\), the tangential form contribution of \(p_m\), and lower-order
coefficient terms depending on \(g\) and \(q=g'/g\). Each occurrence of the
tangential form contribution costs two tangential derivatives.

The only remaining boundary contribution is, as shown above, the final top-boundary residual of
order \(\varepsilon^{2(N+1)}\), which is removed by \(C_N^\varepsilon(P_N^\varepsilon)\). 
Since this correction is \(O(\varepsilon^{2(N+1)})\) in \(\mathcal W_{N,N}\), applying
\(\varepsilon^{-2}A_\varepsilon\)
to it contributes \(O(\varepsilon^{2N})\) in \(\mathcal X\). The same order is obtained
from the bulk terms beyond the retained order. This gives the claimed
remainder estimate.
\end{proof}

\subsubsection*{Finite-order regularity closure}

We next show that the coefficient recursion closes in the above scale. This is
the point where higher regularity of the retained slow variable is needed. The
tracking theorem below will still be measured in the lower regularity topology
\[
\mathcal Z_\theta^\varepsilon\times\mathcal Y_1.
\]

\begin{lemma}[Uniform smoothness of the lifted nonlinearities in the high-order scale]
\label{lem:thin-domain-uniform-smoothness}
Let \(N\in\mathbb N\), \(s_N=2N+6\), and assume
\[
g\in C^{s_N+2}([0,L]),
\qquad
0<g_*\le g\le g^*,
\qquad
g'(0)=g'(L)=0.
\]
Assume also that
\[
f\in C_b^{s_N+N+2}(\mathbb R),
\]
or that \(f\) has been smoothly cut off on a bounded absorbing set. Then the
lifted nonlinearities
\[
F_\varepsilon,
\qquad
G_\varepsilon
\]
are \(C^{N+1}\) on bounded subsets of the high-order coefficient spaces
generated below, with bounds uniform in
\(0<\varepsilon\le\varepsilon_0\). In particular, every coefficient obtained by
applying derivatives of \(F_\varepsilon\) or \(G_\varepsilon\) up to order
\(N+1\) to elements bounded in
\[
\mathcal W_{N,j}
\quad\text{and}\quad
\mathcal Y_{s_N-2j}
\]
belongs to the corresponding space with the expected loss of derivatives.
\end{lemma}

\begin{proof}
The lifted nonlinearities are compositions of the Nemytskii operator generated
by \(f\), the averaging map \(M\), the extension map \(\mathcal E\), the trace
operator entering \(R_g\), and the liftings
\[
\Lambda_\varepsilon^1,
\qquad
\Lambda_\varepsilon^0.
\]
The high-order trace and lifting estimates from
\Cref{subsec:thin-domain-application} give uniform bounds for these linear
operators in the Sobolev scale used here. Since \(\Omega\subset\mathbb R\),
the spaces
\[
\mathcal Y_{s_N-2j}
\]
are algebras for all \(0\le j\le N\), and the anisotropic spaces
\[
\mathcal W_{N,j}
=
H^{s_N-2j}(\Omega;H^2(0,1))\cap\mathcal X
\]
are closed under the products appearing in the Nemytskii expansion. The choice
\[
s_N=2N+6
\]
ensures that after \(N\) losses of two tangential derivatives, at least six
tangential derivatives remain. Thus all traces, products, and compositions
which occur in the recursion are controlled.

The standard Sobolev-Moser estimates for Nemytskii maps generated by
\(C_b^{s_N+N+2}\)-functions imply uniform \(C^{N+1}\)-bounds on bounded sets (see \cite{han1997elliptic}).
The coefficients \(g\), \(g'\), and \(q=g'/g\) have the required regularity
because \(g\in C^{s_N+2}\) and \(g\) is bounded away from zero. Combining these
bounds with the uniform estimates for the liftings gives the claim.
\end{proof}

\begin{lemma}[Uniform inverse on the slow fast-tail in the high-order scale]
\label{lem:uniform-slow-tail-inverse-high}
Let
\[
\mathcal Y=\mathcal Y_F^\varepsilon\oplus\mathcal Y_S^\varepsilon
\]
be the spectral splitting from \Cref{subsec:thin-domain-application} with
\[
\zeta=\varepsilon^2.
\]
Set
\[
\widetilde B_{g,\varepsilon}
:=
\varepsilon^2B_g|_{\mathcal Y_F^\varepsilon}.
\]
Then, for every Sobolev index \(\sigma\ge0\) in the range used below,
\[
\widetilde B_{g,\varepsilon}^{-1}:
\mathcal Y_F^\varepsilon\cap\mathcal Y_\sigma
\to
\mathcal Y_F^\varepsilon\cap\mathcal Y_{\sigma+2}
\]
is bounded uniformly in \(0<\varepsilon\le\varepsilon_0\).
\end{lemma}

\begin{proof}
Let
\[
0=\lambda_0<\lambda_1\le\lambda_2\le\cdots
\]
be the eigenvalues of \(-B_g\), with corresponding eigenfunctions \(e_n\).
Since \(B_g\) is a regular one-dimensional Sturm--Liouville operator,
\[
\lambda_n\sim Cn^2.
\]
The fast tail \(\mathcal Y_F^\varepsilon\) consists of modes satisfying
\[
\lambda_n\gtrsim \varepsilon^{-2}.
\]
On such a mode,
\[
\widetilde B_{g,\varepsilon}e_n
=
-\varepsilon^2\lambda_n e_n.
\]
Hence
\[
|\varepsilon^2\lambda_n|\ge c>0
\]
uniformly on the tail. Moreover,
\[
(1+\lambda_n)^{(\sigma+2)/2}
|\varepsilon^2\lambda_n|^{-1}
\le
C(1+\lambda_n)^{\sigma/2}
\]
on the tail, since
\[
\varepsilon^2\lambda_n\ge c.
\]
Therefore
\[
\|\widetilde B_{g,\varepsilon}^{-1}y\|_{\mathcal Y_{\sigma+2}}
\le
C\|y\|_{\mathcal Y_\sigma},
\]
with \(C\) independent of \(\varepsilon\).
\end{proof}

\begin{proposition}[Finite-order closure of the coefficient recursion]
\label{prop:finite-order-coefficient-closure}
Fix \(N\in\mathbb N\), let
\[
s_N=2N+6,
\]
and assume
\[
g\in C^{s_N+2}([0,L]),
\qquad
0<g_*\le g\le g^*,
\qquad
g'(0)=g'(L)=0.
\]
Assume also that
\[
f\in C_b^{s_N+N+2}(\mathbb R),
\]
or that \(f\) has been smoothly cut off on a bounded absorbing set. Let the
retained slow variable satisfy
\[
v_S\in
B_{\mathcal Y_S^\varepsilon\cap\mathcal Y_{s_N}}(0,r).
\]
Suppose that coefficients
\[
p_j:
B_{\mathcal Y_S^\varepsilon\cap\mathcal Y_{s_N}}(0,r)
\to
\mathcal W_{N,j-1},
\]
and
\[
q_j:
B_{\mathcal Y_S^\varepsilon\cap\mathcal Y_{s_N}}(0,r)
\to
\mathcal Y_F^\varepsilon\cap\mathcal Y_{s_N-2j}
\]
have been constructed for \(j=1,\dots,m\), satisfy the conormal hierarchy, and
are \(C^1\)-bounded uniformly in these spaces. Then the next coefficient source terms generated by the invariance equation satisfy
\[
Q_{m,z}^\varepsilon(v_S)\in \mathcal X_{N,m+1},
\]
and
\[
Q_{m,F}^\varepsilon(v_S)
\in
\mathcal Y_F^\varepsilon\cap\mathcal Y_{s_N-2m-2}.
\]
Moreover, the maps
\[
v_S\mapsto Q_{m,z}^\varepsilon(v_S),
\qquad
v_S\mapsto Q_{m,F}^\varepsilon(v_S)
\]
are \(C^1\)-bounded uniformly on
\[
B_{\mathcal Y_S^\varepsilon\cap\mathcal Y_{s_N}}(0,r).
\]
Finally, the equations defining the next coefficients admit solutions
\[
p_{m+1}\in\mathcal W_{N,m},
\qquad
q_{m+1}\in
\mathcal Y_F^\varepsilon\cap\mathcal Y_{s_N-2m-2},
\]
with the required conormal hierarchy.
\end{proposition}

\begin{proof}
We argue by induction. The slow space
\[
\mathcal Y_{s_N}
\]
is an algebra, because \(\Omega\subset\mathbb R\) and \(s_N\ge2\). The
Nemytskii estimates from \Cref{lem:thin-domain-uniform-smoothness} imply that
all nonlinear coefficient expressions generated by \(F_\varepsilon\) and
\(G_\varepsilon\) are \(C^1\)-bounded in the stated spaces.

Assume that the coefficients up to order \(m\) have the asserted regularity.
The source \(Q_{m,z}^\varepsilon\) is a finite sum of terms formed from
\[
p_1,\dots,p_m,
\qquad
q_1,\dots,q_m,
\qquad
v_S,
\]
their first Fréchet derivatives with respect to \(v_S\), the operators
\[
\mathcal D,\quad R_g,\quad M,\quad \Lambda_\varepsilon^1,\quad
\Lambda_\varepsilon^0,
\]
and derivatives of the Nemytskii map generated by \(f\) up to order \(m+1\).
Each occurrence of the tangential form contribution costs two tangential
derivatives. Hence after \(m+1\) recursive steps, the source belongs to
\[
\mathcal X_{N,m+1}
=
H^{s_N-2m-2}(\Omega;L^2(0,1))\cap\mathcal X.
\]
This proves
\[
Q_{m,z}^\varepsilon(v_S)\in \mathcal X_{N,m+1}.
\]

The same argument applies to the slow-fast source. Since \(B_g\) is a
self-adjoint Sturm--Liouville operator and the projections
\[
\Pi_F^\varepsilon,\qquad \Pi_S^\varepsilon
\]
are spectral projections, they are bounded on every \(\mathcal Y_k\). Therefore
\[
Q_{m,F}^\varepsilon(v_S)
\in
\mathcal Y_F^\varepsilon\cap\mathcal Y_{s_N-2m-2}.
\]
The \(C^1\)-bounds follow from the \(C^{N+1}\)-bounds on the nonlinearities,
the induction hypothesis, and the boundedness of the liftings.

It remains to solve for the next coefficients. The fast equation at order
\(\varepsilon^{2m}\) is a fiberwise mean-zero Neumann equation with an upper boundary
datum fixed by the conormal hierarchy:
\[
\partial_Yp_{m+1}=0
\qquad\text{on }Y=0,
\]
and
\[
\partial_Yp_{m+1}=gg'\mathcal Dp_m
\qquad\text{on }Y=1,
\]
with the convention that this datum is zero for \(m=0\). We first lift this
boundary datum by \(\mathscr B_m\), subtract the lifting, and solve the
remaining homogeneous mean-zero vertical Neumann problem. Since the vertical
Neumann operator is invertible on the mean-zero fiber and preserves tangential
Sobolev regularity, this yields
\[
p_{m+1}\in\mathcal W_{N,m}.
\]

The slow-fast coefficient is determined by
\[
\widetilde B_{g,\varepsilon}q_{m+1}
=
Q_{m,F}^\varepsilon.
\]
By \Cref{lem:uniform-slow-tail-inverse-high},
\[
q_{m+1}
=
\widetilde B_{g,\varepsilon}^{-1}Q_{m,F}^\varepsilon
\in
\mathcal Y_F^\varepsilon\cap\mathcal Y_{s_N-2m}.
\]
This is stronger than the stated target regularity and closes the induction.
\end{proof}

\subsubsection*{Recursive graph construction and defect estimate}

\begin{proposition}[Recursive thin-domain graph construction with conormal hierarchy]
\label{prop:recursive-thin-domain-graph-construction-form}
Assume the hypotheses of
\Cref{prop:finite-order-coefficient-closure}. Then there exist coefficient
maps
\[
p_j:
B_{\mathcal Y_S^\varepsilon\cap\mathcal Y_{s_N}}(0,r)
\to
\mathcal W_{N,j-1},
\]
and
\[
q_j:
B_{\mathcal Y_S^\varepsilon\cap\mathcal Y_{s_N}}(0,r)
\to
\mathcal Y_F^\varepsilon\cap\mathcal Y_{s_N-2j},
\]
for \(j=1,\dots,N\), such that the \(p_j\) satisfy the conormal hierarchy from
\Cref{lem:form-expansion-conormal-hierarchy}. Define
\[
P_N^\varepsilon(v_S)
:=
\sum_{j=1}^N\varepsilon^{2j} p_j(v_S),
\]
and
\[
Q_N^\varepsilon(v_S)
:=
\sum_{j=1}^N\varepsilon^{2j} q_j(v_S).
\]
Set
\[
h_{\varepsilon,N,z}^{\rm app}(v_S)
:=
P_N^\varepsilon(v_S)
+
C_N^\varepsilon(P_N^\varepsilon(v_S)),
\]
and
\[
h_{\varepsilon,N,F}^{\rm app}(v_S)
:=
Q_N^\varepsilon(v_S).
\]
Then
\[
h_{\varepsilon,N,z}^{\rm app}(v_S)\in D(A_\varepsilon)
\]
for every
\[
v_S\in B_{\mathcal Y_S^\varepsilon\cap\mathcal Y_{s_N}}(0,r).
\]
Moreover, if
\[
\mathcal R_{\varepsilon,N}^{\rm td}(v_S)
:=
B_gv_S
+
\Pi_S^\varepsilon
G_\varepsilon\Big(
h_{\varepsilon,N,z}^{\rm app}(v_S),
h_{\varepsilon,N,F}^{\rm app}(v_S)+v_S
\Big),
\]
and
\[
\begin{aligned}
\mathcal I_{\varepsilon,N,z}^{\rm td}(v_S)
:={}&
Dh_{\varepsilon,N,z}^{\rm app}(v_S)
\mathcal R_{\varepsilon,N}^{\rm td}(v_S)
-
\varepsilon^{-2}A_\varepsilon h_{\varepsilon,N,z}^{\rm app}(v_S)
\\
&-
F_\varepsilon\Big(
h_{\varepsilon,N,z}^{\rm app}(v_S),
h_{\varepsilon,N,F}^{\rm app}(v_S)+v_S
\Big),
\end{aligned}
\]
and
\[
\begin{aligned}
\mathcal I_{\varepsilon,N,F}^{\rm td}(v_S)
:={}&
Dh_{\varepsilon,N,F}^{\rm app}(v_S)
\mathcal R_{\varepsilon,N}^{\rm td}(v_S)
-
\varepsilon^{-2}\widetilde B_{g,\varepsilon}
h_{\varepsilon,N,F}^{\rm app}(v_S)
\\
&-
\Pi_F^\varepsilon
G_\varepsilon\Big(
h_{\varepsilon,N,z}^{\rm app}(v_S),
h_{\varepsilon,N,F}^{\rm app}(v_S)+v_S
\Big),
\end{aligned}
\]
then
\[
\sup_{\|v_S\|_{\mathcal Y_{s_N}}\le r}
\left(
\|\mathcal I_{\varepsilon,N,z}^{\rm td}(v_S)\|_{\mathcal X}
+
\|\mathcal I_{\varepsilon,N,F}^{\rm td}(v_S)\|_{\mathcal Y_\delta}
\right)
\le
C\varepsilon^{2N}.
\]
\end{proposition}

\begin{proof}
We construct the coefficients inductively. Suppose that
\[
p_1,\dots,p_m,
\qquad
q_1,\dots,q_m
\]
have already been constructed and satisfy the asserted regularity and boundary
hierarchy. Insert
\[
P_m^\varepsilon
=
\sum_{j=1}^m\varepsilon^{2j}p_j,
\qquad
Q_m^\varepsilon
=
\sum_{j=1}^m \varepsilon^{2j} q_j
\]
into the two invariance equations and collect the coefficient of order
\(\varepsilon^{2m}\). Denote the resulting coefficient sources by
\[
Q_{m,z}^\varepsilon(v_S),
\qquad
Q_{m,F}^\varepsilon(v_S).
\]
By \Cref{prop:finite-order-coefficient-closure},
\[
Q_{m,z}^\varepsilon(v_S)\in\mathcal X_{N,m+1},
\]
and
\[
Q_{m,F}^\varepsilon(v_S)
\in
\mathcal Y_F^\varepsilon\cap\mathcal Y_{s_N-2m-2}.
\]

The next fast coefficient \(p_{m+1}\) is defined as the solution of the
order-\(\varepsilon^{2m}\) fast coefficient equation, with boundary data prescribed by
the conormal hierarchy:
\[
\partial_Yp_{m+1}=0
\qquad\text{on }Y=0,
\]
\[
\partial_Yp_{m+1}=gg'\mathcal Dp_m
\qquad\text{on }Y=1,
\]
where the right-hand side is interpreted as zero for \(m=0\), and
\[
\mathcal Dp_{m+1}=0
\qquad\text{on }\partial\Omega\times(0,1).
\]
The construction in the proof of
\Cref{prop:finite-order-coefficient-closure} gives
\[
p_{m+1}\in\mathcal W_{N,m}.
\]
This choice cancels the order-\(\varepsilon^{2m}\) coefficient in the fast defect.

Similarly, \(q_{m+1}\) is defined by
\[
\widetilde B_{g,\varepsilon}q_{m+1}
=
Q_{m,F}^\varepsilon.
\]
By the uniform tail inverse estimate,
\[
q_{m+1}\in
\mathcal Y_F^\varepsilon\cap\mathcal Y_{s_N-2m}.
\]
This choice cancels again the order-\(\varepsilon^{2m}\) coefficient in the slow-fast defect.

Starting with \(m=0\), this procedure constructs the coefficients up to order
\(N\), and all coefficients of both defects through order \(\varepsilon^{2(N-1)}\)
vanish. The polynomial \(P_N^\varepsilon\) satisfies the conormal boundary
condition up to a boundary residual of order \(\varepsilon^{2(N+1)}\). By
\Cref{lem:form-expansion-conormal-hierarchy}, the correction
\[
C_N^\varepsilon(P_N^\varepsilon)
\]
removes this residual and has size \(O(\varepsilon^{2(N+1)})\). Therefore
\[
h_{\varepsilon,N,z}^{\rm app}
=
P_N^\varepsilon+C_N^\varepsilon(P_N^\varepsilon)
\]
belongs to \(D(A_\varepsilon)\), and the correction contributes only
\(O(\varepsilon^{2N})\) to
\[
\varepsilon^{-2}A_\varepsilon h_{\varepsilon,N,z}^{\rm app}.
\]

It remains to estimate the defect. All terms in both defects through order
\(\varepsilon^{2(N-1)}\) cancel by construction. It remains to bound the order-\(\varepsilon^{2N}\)
contributions.

\medskip
\noindent\emph{Estimate of
\(\|\mathcal I_{\varepsilon,N,z}^{\rm td}(v_S)\|_{\mathcal X}\).}
By \Cref{lem:form-expansion-conormal-hierarchy}, the fast operator term has
the expansion
\[
\varepsilon^{-2}
A_\varepsilon h_{\varepsilon,N,z}^{\rm app}
=
\sum_{m=0}^{N-1}
\varepsilon^{2m}
\mathcal L_m(p_1,\dots,p_{m+1})
+
\varepsilon^{2N}\mathcal R_{N,A}^\varepsilon,
\]
with
\[
\|\mathcal R_{N,A}^\varepsilon\|_{\mathcal X}
\le
C\sum_{j=1}^N
\|p_j\|_{\mathcal W_{N,j-1}}.
\]
The correction
\[
C_N^\varepsilon(P_N^\varepsilon)\in\mathcal W_{N,N}
\]
has size \(O(\varepsilon^{2(N+1)})\), and therefore contributes only \(O(\varepsilon^{2N})\) to
\[
\varepsilon^{-2}A_\varepsilon h_{\varepsilon,N,z}^{\rm app}.
\]
Its contribution to
\[
Dh_{\varepsilon,N,z}^{\rm app}(v_S)\,
\mathcal R_{\varepsilon,N}^{\rm td}(v_S)
\]
and to the nonlinear term is of order \(O(\varepsilon^{2(N+1)})\), hence also harmless
for the \(O(\varepsilon^{2N})\)-estimate.

The remaining terms in
\(\mathcal I_{\varepsilon,N,z}^{\rm td}\) are the order-\(\varepsilon^{2N}\) remainders
from the Taylor expansion of
\[
F_\varepsilon\Big(
h_{\varepsilon,N,z}^{\rm app}(v_S),
h_{\varepsilon,N,F}^{\rm app}(v_S)+v_S
\Big)
\]
and from the product
\[
Dh_{\varepsilon,N,z}^{\rm app}(v_S)
\mathcal R_{\varepsilon,N}^{\rm td}(v_S).
\]
By \Cref{lem:thin-domain-uniform-smoothness}, the lifted nonlinearity
\(F_\varepsilon\) is \(C^{N+1}\) with bounds uniform in \(\varepsilon\) on
bounded subsets of the high-order coefficient spaces. The same bounds apply to
the reduced vector field
\[
\mathcal R_{\varepsilon,N}^{\rm td}.
\]
Therefore these remainders are \(O(\varepsilon^{2N})\) in \(\mathcal X\). Consequently,
\[
\|\mathcal I_{\varepsilon,N,z}^{\rm td}(v_S)\|_{\mathcal X}
\le
C\varepsilon^{2N},
\]
uniformly for
\[
\|v_S\|_{\mathcal Y_{s_N}}\le r.
\]

\medskip
\noindent\emph{Estimate of
\(\|\mathcal I_{\varepsilon,N,F}^{\rm td}(v_S)\|_{\mathcal Y_\delta}\).}
The defect
\(\mathcal I_{\varepsilon,N,F}^{\rm td}\) contains the order-\(\varepsilon^{2N}\)
remainders from the Taylor expansion of
\[
G_\varepsilon\Big(
h_{\varepsilon,N,z}^{\rm app}(v_S),
h_{\varepsilon,N,F}^{\rm app}(v_S)+v_S
\Big)
\]
and from the product
\[
Dh_{\varepsilon,N,F}^{\rm app}(v_S)
\mathcal R_{\varepsilon,N}^{\rm td}(v_S).
\]
The coefficients \(q_j\) were constructed using the inverse
\(\widetilde B_{g,\varepsilon}^{-1}\)
only to cancel all coefficients through order \(\varepsilon^{2(N-1)}\).
After this cancellation, the remaining terms are the Taylor remainders just described.

By \Cref{lem:thin-domain-uniform-smoothness}, the order-\(\varepsilon^{2N}\) remainder
of \(G_\varepsilon\), and the corresponding remainder coming from
\(Dh_{\varepsilon,N,F}^{\rm app}\mathcal R_{\varepsilon,N}^{\rm td}\), are
bounded in the space
\[
\mathcal Y_{s_N-2N}.
\]
Since
\[
s_N-2N=6,
\]
this means the remainder is uniformly bounded in
\[
\mathcal Y_6\simeq H^6(\Omega).
\]
On the other hand,
\[
0<\delta<\theta-\frac34<\frac14,
\]
so
\[
2\delta<\frac12<6.
\]
Hence the Sobolev embedding
\[
H^6(\Omega)\hookrightarrow H^{2\delta}(\Omega)
\simeq \mathcal Y_\delta
\]
gives
\[
\|\mathcal I_{\varepsilon,N,F}^{\rm td}(v_S)\|_{\mathcal Y_\delta}
\le
C\varepsilon^{2N},
\]
uniformly for
\[
\|v_S\|_{\mathcal Y_{s_N}}\le r.
\]

Combining the two estimates yields
\[
\sup_{\|v_S\|_{\mathcal Y_{s_N}}\le r}
\left(
\|\mathcal I_{\varepsilon,N,z}^{\rm td}(v_S)\|_{\mathcal X}
+
\|\mathcal I_{\varepsilon,N,F}^{\rm td}(v_S)\|_{\mathcal Y_\delta}
\right)
\le
C\varepsilon^{2N}.
\]
\end{proof}

\begin{proposition}[Persistence of the high-order slow ball]
\label{prop:high-order-slow-ball-persistence}
Fix \(T>0\) and \(N\in\mathbb N\), and let
\[
s_N:=2N+6.
\]
Assume the hypotheses of
\Cref{prop:recursive-thin-domain-graph-construction-form}. 
Let
\(\bar v_S^\varepsilon \)
solve the approximate reduced equation
\[
\partial_t\bar v_S^\varepsilon
=
\mathcal R_{\varepsilon,N}^{\rm td}(\bar v_S^\varepsilon),
\qquad
\bar v_S^\varepsilon(0)=\bar v_{S,0}^\varepsilon,
\]
where
\[
\mathcal R_{\varepsilon,N}^{\rm td}(\bar v_S^\varepsilon)
=
B_g \bar v_S^\varepsilon
+
\Pi_S^\varepsilon
G_\varepsilon\Big(
h_{\varepsilon,N,z}^{\rm app}(\bar v_S^\varepsilon),
h_{\varepsilon,N,F}^{\rm app}(\bar v_S^\varepsilon)+\bar v_S^\varepsilon
\Big).
\]
Assume that
\[
\|\bar v_{S,0}^\varepsilon\|_{\mathcal Y_{s_N}}\le r_0.
\]
Then there exists a constant
\[
C_T=C(T,r_0,N,g,f)
\]
independent of \(0<\varepsilon\le\varepsilon_0\) such that
\[
\sup_{0\le t\le T}
\|\bar v_S^\varepsilon(t)\|_{\mathcal Y_{s_N}}
\le C_T.
\]
Consequently, if the coefficient construction in
\Cref{prop:recursive-thin-domain-graph-construction-form} is carried out on a
ball
\[
B_{\mathcal Y_S^\varepsilon\cap\mathcal Y_{s_N}}(0,R)
\]
with
\[
R>C_T,
\]
then the reduced trajectory remains in the domain of the approximate graph for
all \(0\le t\le T\).
\end{proposition}

\begin{proof}
Since \(B_g\) is self-adjoint and nonpositive on
\(\mathcal Y=L^2(\Omega,g\,dx)\), the operator
\[
I-B_g
\]
is positive and self-adjoint. We write
\[
\|v\|_{\mathcal Y_{s_N}}
=
\|(I-B_g)^{s_N/2}v\|_{\mathcal Y}.
\]
The spectral projection \(\Pi_S^\varepsilon\) commutes with \(B_g\), and hence
\[
\|\Pi_S^\varepsilon v\|_{\mathcal Y_{s_N}}
\le
\|v\|_{\mathcal Y_{s_N}}
\]
uniformly in \(\varepsilon\).

Set
\[
A_s:=(I-B_g)^{s_N/2}.
\]
Testing the reduced equation against
\[
A_s^2\bar v_S^\varepsilon
\]
gives
\[
\frac12
\frac{d}{dt}
\|\bar v_S^\varepsilon\|_{\mathcal Y_{s_N}}^2
=
(B_g\bar v_S^\varepsilon,\bar v_S^\varepsilon)_{\mathcal Y_{s_N}}
+
\Big(
\Pi_S^\varepsilon
G_\varepsilon(
h_{\varepsilon,N,z}^{\rm app}(\bar v_S^\varepsilon),
h_{\varepsilon,N,F}^{\rm app}(\bar v_S^\varepsilon)
+\bar v_S^\varepsilon),
\bar v_S^\varepsilon
\Big)_{\mathcal Y_{s_N}}.
\]
The first term is nonpositive because \(B_g\) is nonpositive and commutes with
\(A_s\):
\[
(B_g\bar v_S^\varepsilon,\bar v_S^\varepsilon)_{\mathcal Y_{s_N}}
=
(B_gA_s\bar v_S^\varepsilon,A_s\bar v_S^\varepsilon)_{\mathcal Y}
\le0.
\]
For the nonlinear term, the finite-order closure and the high-order smoothness
of the lifted nonlinearities imply that, on bounded subsets of
\(\mathcal Y_{s_N}\),
\[
\left\|
G_\varepsilon\Big(
h_{\varepsilon,N,z}^{\rm app}(v_S),
h_{\varepsilon,N,F}^{\rm app}(v_S)+v_S
\Big)
\right\|_{\mathcal Y_{s_N}}
\le
C_R\bigl(1+\|v_S\|_{\mathcal Y_{s_N}}\bigr)
\]
whenever
\[
\|v_S\|_{\mathcal Y_{s_N}}\le R.
\]
Here \(C_R\) is independent of \(\varepsilon\). Therefore, as long as
\(\bar v_S^\varepsilon\) remains in such a ball,
\[
\frac{d}{dt}
\|\bar v_S^\varepsilon\|_{\mathcal Y_{s_N}}
\le
C_R\bigl(1+\|\bar v_S^\varepsilon\|_{\mathcal Y_{s_N}}\bigr).
\]
By Gronwall's inequality,
\[
\|\bar v_S^\varepsilon(t)\|_{\mathcal Y_{s_N}}
\le
\bigl(\|\bar v_{S,0}^\varepsilon\|_{\mathcal Y_{s_N}}+1\bigr)e^{C_Rt}-1.
\]
A standard continuation argument now yields a bound on \([0,T]\). Namely, if
the graph construction is performed on a ball of radius \(R\), choose \(R\)
larger than the right-hand side evaluated at \(t=T\). Then the solution cannot
leave this ball before time \(T\). This gives
\[
\sup_{0\le t\le T}
\|\bar v_S^\varepsilon(t)\|_{\mathcal Y_{s_N}}
\le C_T,
\]
with \(C_T\) independent of \(\varepsilon\).
\end{proof}

\subsubsection*{Finite-time tracking and reconstruction}

\begin{assumption}[Dissipative lower-order bound for the lifted system]
\label{ass:dissipative-lifted-bound}
We assume that the reaction structure is dissipative in the following sense.
There exist constants \(a>0\) and \(b\ge0\), independent of
\(0<\varepsilon\le\varepsilon_0\), such that every sufficiently regular
solution of the lifted system satisfies
\[
\frac{d}{dt}
\left(
\|z^\varepsilon(t)\|_{\mathcal X}^2
+
\|v^\varepsilon(t)\|_{\mathcal Y}^2
\right)
\le
-a
\left(
\|z^\varepsilon(t)\|_{\mathcal X}^2
+
\|v^\varepsilon(t)\|_{\mathcal Y}^2
\right)
+b.
\]
Moreover, on bounded subsets of
\(\mathcal X\times\mathcal Y\), the lifted nonlinearities are bounded in the
spaces in which they enter the mild formulation:
\[
F_\varepsilon(z,v)\in\mathcal X,
\qquad
G_\varepsilon(z,v)\in\mathcal Y_\delta,
\]
with bounds independent of \(\varepsilon\). This assumption is satisfied, for
example, when the original reaction term has a standard dissipative
reaction-diffusion Lyapunov estimate and the liftings
\(\Lambda_\varepsilon^1,\Lambda_\varepsilon^0\) satisfy the uniform bounds from
\Cref{subsec:thin-domain-application}. See for example \cite{Rothe1984,hollis1987global,Morgan1989}.
\end{assumption}

\begin{proposition}[Uniform finite-time bounds and absorbing ball]
\label{prop:full-lifted-uniform-bound}
Assume the hypotheses of \Cref{subsec:thin-domain-application} and
\Cref{ass:dissipative-lifted-bound}. Let
\[
(z^\varepsilon(t),v^\varepsilon(t))
\]
be a mild solution of the lifted thin-domain system on \([0,T]\). Suppose that
the initial data satisfy
\[
\|z_0^\varepsilon\|_{\mathcal Z_\theta^\varepsilon}
+
\|v_0^\varepsilon\|_{\mathcal Y_1}
\le R_0
\]
uniformly in \(0<\varepsilon\le\varepsilon_0\). Then there exists
\[
R_T=R_T(T,R_0)>0,
\]
independent of \(\varepsilon\), such that
\[
\sup_{0\le t\le T}
\left(
\|z^\varepsilon(t)\|_{\mathcal Z_\theta^\varepsilon}
+
\|v^\varepsilon(t)\|_{\mathcal Y_1}
\right)
\le R_T.
\]
If, in addition, \(T\) is allowed to be large, then there exist constants
\[
R_*>0,
\qquad
t_*=t_*(R_0)>0,
\]
independent of \(\varepsilon\), such that
\[
\|z^\varepsilon(t)\|_{\mathcal Z_\theta^\varepsilon}
+
\|v^\varepsilon(t)\|_{\mathcal Y_1}
\le R_*
\qquad
\text{for all }t\ge t_*.
\]
Thus the lifted semiflow possesses a bounded absorbing set in
\[
\mathcal Z_\theta^\varepsilon\times\mathcal Y_1
\]
with radius independent of \(\varepsilon\).
\end{proposition}

\begin{proof}
We first prove the lower-order dissipative bound. By
\Cref{ass:dissipative-lifted-bound},
\[
\frac{d}{dt}
\left(
\|z^\varepsilon(t)\|_{\mathcal X}^2
+
\|v^\varepsilon(t)\|_{\mathcal Y}^2
\right)
\le
-a
\left(
\|z^\varepsilon(t)\|_{\mathcal X}^2
+
\|v^\varepsilon(t)\|_{\mathcal Y}^2
\right)
+b.
\]
Gronwall's inequality gives
\[
\|z^\varepsilon(t)\|_{\mathcal X}^2
+
\|v^\varepsilon(t)\|_{\mathcal Y}^2
\le
e^{-at}
\left(
\|z_0^\varepsilon\|_{\mathcal X}^2
+
\|v_0^\varepsilon\|_{\mathcal Y}^2
\right)
+
\frac ba.
\]
Hence the solution is uniformly bounded in
\(\mathcal X\times\mathcal Y\) on every finite time interval, and it enters a
bounded absorbing ball in this lower topology after a time depending only on
the initial radius.

We now upgrade the estimate to
\[
\mathcal Z_\theta^\varepsilon\times\mathcal Y_1.
\]
The lifted mild formulation gives
\[
z^\varepsilon(t)
=
e^{t\varepsilon^{-2}A_\varepsilon}z_0^\varepsilon
+
\int_0^t
e^{(t-s)\varepsilon^{-2}A_\varepsilon}
F_\varepsilon(z^\varepsilon(s),v^\varepsilon(s))\,ds,
\]
and
\[
v^\varepsilon(t)
=
e^{tB_g}v_0^\varepsilon
+
\int_0^t
e^{(t-s)B_g}
G_\varepsilon(z^\varepsilon(s),v^\varepsilon(s))\,ds.
\]
The trace-compatible fast smoothing estimate yields
\[
\|e^{t\varepsilon^{-2}A_\varepsilon}\|_{\mathcal B(\mathcal X,\mathcal Z_\theta^\varepsilon)}
\le
Ct^{-\theta}e^{-ct/\varepsilon^2},
\qquad t>0.
\]
Moreover, the analytic semigroup generated by \(B_g\) satisfies
\[
\|e^{tB_g}\|_{\mathcal B(\mathcal Y_\delta,\mathcal Y_1)}
\le
Ct^{-(1-\delta)},
\qquad t>0,
\]
and
\[
\|e^{tB_g}\|_{\mathcal B(\mathcal Y_1,\mathcal Y_1)}
\le C_T,
\qquad 0\le t\le T.
\]
Using the lower-order absorbing bound and the boundedness of the nonlinearities
on bounded lower-order sets, we obtain, for \(0<t\le T\),
\[
\|z^\varepsilon(t)\|_{\mathcal Z_\theta^\varepsilon}
\le
C_T\|z_0^\varepsilon\|_{\mathcal Z_\theta^\varepsilon}
+
C
\int_0^t
(t-s)^{-\theta}
\|F_\varepsilon(z^\varepsilon(s),v^\varepsilon(s))\|_{\mathcal X}
\,ds,
\]
and
\[
\|v^\varepsilon(t)\|_{\mathcal Y_1}
\le
C_T\|v_0^\varepsilon\|_{\mathcal Y_1}
+
C
\int_0^t
(t-s)^{-(1-\delta)}
\|G_\varepsilon(z^\varepsilon(s),v^\varepsilon(s))\|_{\mathcal Y_\delta}
\,ds.
\]
Both kernels are integrable because
\[
0<\theta<1,
\qquad
\delta>0.
\]
Therefore
\[
\sup_{0\le t\le T}
\left(
\|z^\varepsilon(t)\|_{\mathcal Z_\theta^\varepsilon}
+
\|v^\varepsilon(t)\|_{\mathcal Y_1}
\right)
\le
R_T,
\]
with \(R_T\) independent of \(\varepsilon\).

For the absorbing estimate, fix a small \(\tau>0\). For \(t\ge\tau\), write the
mild solution from the shifted initial time \(t-\tau\):
\[
z^\varepsilon(t)
=
e^{\tau \varepsilon^{-2}A_\varepsilon}z^\varepsilon(t-\tau)
+
\int_{t-\tau}^t
e^{(t-s)\varepsilon^{-2}A_\varepsilon}
F_\varepsilon(z^\varepsilon(s),v^\varepsilon(s))\,ds,
\]
and similarly
\[
v^\varepsilon(t)
=
e^{\tau B_g}v^\varepsilon(t-\tau)
+
\int_{t-\tau}^t
e^{(t-s)B_g}
G_\varepsilon(z^\varepsilon(s),v^\varepsilon(s))\,ds.
\]
Once the lower-order solution has entered the absorbing ball in
\(\mathcal X\times\mathcal Y\), the nonlinear terms are uniformly bounded in
\(\mathcal X\) and \(\mathcal Y_\delta\). The same smoothing estimates on the
fixed interval of length \(\tau\) therefore give
\[
\|z^\varepsilon(t)\|_{\mathcal Z_\theta^\varepsilon}
+
\|v^\varepsilon(t)\|_{\mathcal Y_1}
\le R_*
\]
for all sufficiently large \(t\), where \(R_*\) is independent of
\(\varepsilon\). This proves the absorbing-ball statement.
\end{proof}

\begin{theorem}[Thin-domain approximation theorem in lifted variables]
\label{thm:thin-domain-approximation}
Assume the hypotheses of
\Cref{prop:recursive-thin-domain-graph-construction-form}. Let
\[
(z^\varepsilon(t),v^\varepsilon(t))
\]
be a mild solution of the lifted thin-domain system from
\Cref{subsec:thin-domain-application} on \([0,T]\). Write
\[
v^\varepsilon(t)=v_F^\varepsilon(t)+v_S^\varepsilon(t)
\]
according to the splitting
\[
\mathcal Y
=
\mathcal Y_F^\varepsilon
\oplus
\mathcal Y_S^\varepsilon.
\]
Let \(\bar v_S^\varepsilon(t)\) solve the approximate reduced equation
\[
\partial_t\bar v_S^\varepsilon
=
\mathcal R_{\varepsilon,N}^{\rm td}(\bar v_S^\varepsilon),
\qquad
\bar v_S^\varepsilon(0)=v_S^\varepsilon(0).
\]
Assume that the initial data satisfy
\[
\|z_0^\varepsilon\|_{\mathcal Z_\theta^\varepsilon}
+
\|v_0^\varepsilon\|_{\mathcal Y_1}
\le R_0,
\]
and that
\[
\|v_{S,0}^\varepsilon\|_{\mathcal Y_{s_N}}\le r_0
\]
uniformly in \(\varepsilon\). Choose the coefficient construction radius
\(R_{\rm coeff}\) so large that it contains both the high-order reduced
trajectory from
\Cref{prop:high-order-slow-ball-persistence}
and the lower-order full trajectory bound from
\Cref{prop:full-lifted-uniform-bound} on \([0,T]\).
Then there is
a constant \(C_T>0\), independent of \(0<\varepsilon\le\varepsilon_0\), such
that
\[
\begin{aligned}
&\sup_{0\le t\le T}
\Big(
\|z^\varepsilon(t)
-
h_{\varepsilon,N,z}^{\rm app}(\bar v_S^\varepsilon(t))
\|_{\mathcal Z_\theta^\varepsilon}
\\
&\qquad\qquad
+
\|v_F^\varepsilon(t)
-
h_{\varepsilon,N,F}^{\rm app}(\bar v_S^\varepsilon(t))
\|_{\mathcal Y_1}
+
\|v_S^\varepsilon(t)-\bar v_S^\varepsilon(t)\|_{\mathcal Y_1}
\Big)
\\
&\qquad\le
C_T
\Big(
\|z_0^\varepsilon
-
h_{\varepsilon,N,z}^{\rm app}(v_{S,0}^\varepsilon)
\|_{\mathcal Z_\theta^\varepsilon}
+
\|v_{F,0}^\varepsilon
-
h_{\varepsilon,N,F}^{\rm app}(v_{S,0}^\varepsilon)
\|_{\mathcal Y_1}
+
\varepsilon^{2N}
\Big).
\end{aligned}
\]
In particular, if the initial data lie on the approximate graph
\[
\mathcal M_{\varepsilon,N}^{\rm td,app}
:=
\left\{
\left(
h_{\varepsilon,N,z}^{\rm app}(v_S),
h_{\varepsilon,N,F}^{\rm app}(v_S),
v_S
\right):
v_S\in
B_{\mathcal Y_S^\varepsilon\cap\mathcal Y_{s_N}}(0,r)
\right\},
\]
then the lifted solution remains within \(O(\varepsilon^{2N})\) of this graph
on every fixed time interval \([0,T]\).
\end{theorem}

\begin{proof}
Set
\[
H_z:=h_{\varepsilon,N,z}^{\rm app},
\qquad
H_F:=h_{\varepsilon,N,F}^{\rm app},
\qquad
\bar v_S:=\bar v_S^\varepsilon.
\]
Define
\[
e_z(t):=z^\varepsilon(t)-H_z(\bar v_S(t)),
\]
\[
e_F(t):=v_F^\varepsilon(t)-H_F(\bar v_S(t)),
\]
and
\[
e_S(t):=v_S^\varepsilon(t)-\bar v_S(t).
\]
The lifted split system gives
\[
\partial_t z^\varepsilon
=
\varepsilon^{-2}A_\varepsilon z^\varepsilon
+
F_\varepsilon(z^\varepsilon,v_F^\varepsilon+v_S^\varepsilon),
\]
\[
\partial_t v_F^\varepsilon
=
\varepsilon^{-2}\widetilde B_{g,\varepsilon}v_F^\varepsilon
+
\Pi_F^\varepsilon
G_\varepsilon(z^\varepsilon,v_F^\varepsilon+v_S^\varepsilon),
\]
and
\[
\partial_t v_S^\varepsilon
=
B_gv_S^\varepsilon
+
\Pi_S^\varepsilon
G_\varepsilon(z^\varepsilon,v_F^\varepsilon+v_S^\varepsilon).
\]
The reduced equation is
\[
\partial_t\bar v_S
=
B_g\bar v_S
+
\Pi_S^\varepsilon
G_\varepsilon(H_z(\bar v_S),H_F(\bar v_S)+\bar v_S).
\]
Using the definitions of the defects, we obtain
\[
\begin{aligned}
\partial_t e_z
={}&
\varepsilon^{-2}A_\varepsilon e_z
\\
&+
\Big[
F_\varepsilon(z^\varepsilon,v_F^\varepsilon+v_S^\varepsilon)
-
F_\varepsilon(H_z(\bar v_S),H_F(\bar v_S)+\bar v_S)
\Big]
-
\mathcal I_{\varepsilon,N,z}^{\rm td}(\bar v_S),
\end{aligned}
\]
and
\[
\begin{aligned}
\partial_t e_F
={}&
\varepsilon^{-2}\widetilde B_{g,\varepsilon}e_F
\\
&+
\Pi_F^\varepsilon
\Big[
G_\varepsilon(z^\varepsilon,v_F^\varepsilon+v_S^\varepsilon)
-
G_\varepsilon(H_z(\bar v_S),H_F(\bar v_S)+\bar v_S)
\Big]
-
\mathcal I_{\varepsilon,N,F}^{\rm td}(\bar v_S).
\end{aligned}
\]
For the retained slow modes,
\[
\begin{aligned}
\partial_t e_S
={}&
B_ge_S
\\
&+
\Pi_S^\varepsilon
\Big[
G_\varepsilon(z^\varepsilon,v_F^\varepsilon+v_S^\varepsilon)
-
G_\varepsilon(H_z(\bar v_S),H_F(\bar v_S)+\bar v_S)
\Big].
\end{aligned}
\]
Moreover,
\[
e_S(0)=0.
\]

Define
\[
E(t)
:=
\|e_z(t)\|_{\mathcal Z_\theta^\varepsilon}
+
\|e_F(t)\|_{\mathcal Y_1}
+
\|e_S(t)\|_{\mathcal Y_1}.
\]
The local Lipschitz estimates for \(F_\varepsilon\) and \(G_\varepsilon\) in
the lower dynamical topology give
\[
\begin{aligned}
&
\|F_\varepsilon(z^\varepsilon,v_F^\varepsilon+v_S^\varepsilon)
-
F_\varepsilon(H_z(\bar v_S),H_F(\bar v_S)+\bar v_S)\|_{\mathcal X}
\\
&\qquad
+
\|G_\varepsilon(z^\varepsilon,v_F^\varepsilon+v_S^\varepsilon)
-
G_\varepsilon(H_z(\bar v_S),H_F(\bar v_S)+\bar v_S)\|_{\mathcal Y_\delta}
\\
&\qquad\le
C E(t).
\end{aligned}
\]
We now write the error equations in mild form. For compactness, define
\[
\mathcal N_z^\varepsilon(t)
:=
F_\varepsilon(z^\varepsilon(t),v_F^\varepsilon(t)+v_S^\varepsilon(t))
-
F_\varepsilon(H_z(\bar v_S(t)),H_F(\bar v_S(t))+\bar v_S(t)),
\]
and
\[
\mathcal N_G^\varepsilon(t)
:=
G_\varepsilon(z^\varepsilon(t),v_F^\varepsilon(t)+v_S^\varepsilon(t))
-
G_\varepsilon(H_z(\bar v_S(t)),H_F(\bar v_S(t))+\bar v_S(t)).
\]
By the local Lipschitz estimates for \(F_\varepsilon\) and \(G_\varepsilon\),
\[
\|\mathcal N_z^\varepsilon(t)\|_{\mathcal X}
+
\|\mathcal N_G^\varepsilon(t)\|_{\mathcal Y_\delta}
\le
C E(t),
\]
where
\[
E(t)
:=
\|e_z(t)\|_{\mathcal Z_\theta^\varepsilon}
+
\|e_F(t)\|_{\mathcal Y_1}
+
\|e_S(t)\|_{\mathcal Y_1}.
\]

For \(e_z\), the mild formulation and the smoothing estimate from \Cref{lem:trace-compatible-fast-smoothing} give
\[
\begin{aligned}
\|e_z(t)\|_{\mathcal Z_\theta^\varepsilon}
\le{}&
C\|e_z(0)\|_{\mathcal Z_\theta^\varepsilon}
\\
&+
C\int_0^t
(t-s)^{-\theta}e^{-c(t-s)/\varepsilon^2}
\left(
\|\mathcal N_z^\varepsilon(s)\|_{\mathcal X}
+
\|\mathcal I_{\varepsilon,N,z}^{\rm td}(\bar v_S(s))\|_{\mathcal X}
\right)\,ds .
\end{aligned}
\]
By \Cref{prop:recursive-thin-domain-graph-construction-form},
\[
\|\mathcal I_{\varepsilon,N,z}^{\rm td}(\bar v_S(s))\|_{\mathcal X}
\le
C\varepsilon^{2N}.
\]
Since
\[
e^{-c(t-s)/\varepsilon^2}\le 1
\]
and \((t-s)^{-\theta}\) is integrable on \((0,T)\), we obtain
\[
\|e_z(t)\|_{\mathcal Z_\theta^\varepsilon}
\le
C\|e_z(0)\|_{\mathcal Z_\theta^\varepsilon}
+
C\int_0^t
(t-s)^{-\theta}E(s)\,ds
+
C_T\varepsilon^{2N}.
\]

For \(e_F\), note that on the fast tail
\[
\varepsilon^{-2}\widetilde B_{g,\varepsilon}
=
B_g|_{\mathcal Y_F^\varepsilon}.
\]
Hence
\[
e^{t\varepsilon^{-2}\widetilde B_{g,\varepsilon}}
=
e^{tB_g}\Pi_F^\varepsilon .
\]
By the slow spectral splitting from \Cref{lem:slow-spectral-splitting-estimates}, the spectrum of
\(B_g\) on \(\mathcal Y_F^\varepsilon\) is bounded to the left by
\(-c\varepsilon^{-2}\).
Therefore the fast-tail semigroup satisfies
\[
\left\|
e^{t\varepsilon^{-2}\widetilde B_{g,\varepsilon}}
\right\|_{\mathcal B(\mathcal Y_\delta,\mathcal Y_1)}
\le
Ct^{\delta-1}e^{-c't/\varepsilon^2},
\qquad t>0.
\]
The mild formulation for \(e_F\) gives
\[
\begin{aligned}
\|e_F(t)\|_{\mathcal Y_1}
\le{}&
C\|e_F(0)\|_{\mathcal Y_1}
\\
&+
C\int_0^t
(t-s)^{\delta-1}e^{-c'(t-s)/\varepsilon^2}
\left(
\|\Pi_F^\varepsilon\mathcal N_G^\varepsilon(s)\|_{\mathcal Y_\delta}
+
\|\mathcal I_{\varepsilon,N,F}^{\rm td}(\bar v_S(s))\|_{\mathcal Y_\delta}
\right)\,ds .
\end{aligned}
\]
Using
\[
e^{-c'(t-s)/\varepsilon^2}\le1,
\]
the boundedness of \(\Pi_F^\varepsilon\) on \(\mathcal Y_\delta\), and
\[
\|\mathcal I_{\varepsilon,N,F}^{\rm td}(\bar v_S(s))\|_{\mathcal Y_\delta}
\le
C\varepsilon^{2N},
\]
we obtain
\[
\|e_F(t)\|_{\mathcal Y_1}
\le
C\|e_F(0)\|_{\mathcal Y_1}
+
C\int_0^t
(t-s)^{\delta-1}E(s)\,ds
+
C_T\varepsilon^{2N}.
\]

For \(e_S\), we use \(e_S(0)=0\). The mild formulation gives
\[
  e_S(t)
  =
  \int_0^t
  e^{(t-s)B_g}\Pi_S^\varepsilon N_G^\varepsilon(s)\,ds .
\]
Since \(\Pi_S^\varepsilon\) is a spectral projection for the self-adjoint Sturm-Liouville operator \(B_g\), it commutes with \(e^{tB_g}\) and is uniformly bounded on the interpolation scale \(\mc Y_\rho\). 
Hence the analytic smoothing estimate for \(B_g\) gives, for \(t>s\),
\[
  \bigl\|
    e^{(t-s)B_g}\Pi_S^\varepsilon
  \bigr\|_{\mathcal B(\mc Y_\delta,\mc Y_1)}
  \le
  C(t-s)^{\delta-1},
\]
with \(C\) independent of \(\varepsilon\). Therefore
\[
  \|e_S(t)\|_{\mc Y_1}
  \le
  C\int_0^t
  (t-s)^{\delta-1}
  \|\Pi_S^\varepsilon \mc N_G^\varepsilon(s)\|_{\mc Y_\delta}\,ds
  \le
  C\int_0^t
  (t-s)^{\delta-1}E(s)\,ds .
\]
Combining this estimate with the bounds for \(e_z\) and \(e_F\), we obtain
\[
  E(t)
  \le
  C_T\bigl(E(0)+\varepsilon^{2N}\bigr)
  +
  C_T\int_0^t
  \Bigl[(t-s)^{-\theta}+(t-s)^{\delta-1}\Bigr]E(s)\,ds .
\]
Since \(0<\theta<1\) and \(\delta>0\), both kernels are integrable on \((0,T)\). 
The singular Gronwall inequality \cite[Lemma~7.1.1]{Henry1981}
therefore yields
\[
  \sup_{0\le t\le T}E(t)
  \le
  C_T\bigl(E(0)+\varepsilon^{2N}\bigr).
\]

Finally,
\[
E(0)
=
\|z_0^\varepsilon-H_z(v_{S,0}^\varepsilon)\|_{\mathcal Z_\theta^\varepsilon}
+
\|v_{F,0}^\varepsilon-H_F(v_{S,0}^\varepsilon)\|_{\mathcal Y_1},
\]
because \(e_S(0)=0\). This proves the theorem.
\end{proof}

\begin{corollary}[Approximation in the original fluctuation variable]
\label{cor:thin-domain-original-variable-approximation}
Under the hypotheses of \Cref{thm:thin-domain-approximation}, define
\[
h_{\varepsilon,N,w}^{\rm app}(v_S)
:=
h_{\varepsilon,N,z}^{\rm app}(v_S)
+
\Lambda_\varepsilon^1
\bigl(
v_S+h_{\varepsilon,N,F}^{\rm app}(v_S)
\bigr).
\]
Then
\[
\begin{aligned}
&\sup_{0\le t\le T}
\left\|
w^\varepsilon(t)
-
h_{\varepsilon,N,w}^{\rm app}(\bar v_S^\varepsilon(t))
\right\|_{H^{2\theta-\delta}(Q)}
\\
&\qquad\le
C_T
\Big(
\|z_0^\varepsilon
-
h_{\varepsilon,N,z}^{\rm app}(v_{S,0}^\varepsilon)
\|_{\mathcal Z_\theta^\varepsilon}
+
\|v_{F,0}^\varepsilon
-
h_{\varepsilon,N,F}^{\rm app}(v_{S,0}^\varepsilon)
\|_{\mathcal Y_1}
+
\varepsilon^{2N}
\Big).
\end{aligned}
\]
Consequently, the full rescaled thin-domain solution
\[
u^\varepsilon
=
v_S^\varepsilon+v_F^\varepsilon+w^\varepsilon
\]
is approximated on \([0,T]\) by
\[
u_{\varepsilon,N}^{\rm app}(t)
:=
\bar v_S^\varepsilon(t)
+
h_{\varepsilon,N,F}^{\rm app}(\bar v_S^\varepsilon(t))
+
h_{\varepsilon,N,w}^{\rm app}(\bar v_S^\varepsilon(t)).
\]
\end{corollary}

\begin{proof}
By definition of the lifted homogeneous fast variable,
\[
z^\varepsilon
=
w^\varepsilon-\Lambda_\varepsilon^1v^\varepsilon.
\]
Hence
\[
w^\varepsilon
=
z^\varepsilon+\Lambda_\varepsilon^1(v_S^\varepsilon+v_F^\varepsilon).
\]
On the approximate graph,
\[
h_{\varepsilon,N,w}^{\rm app}(\bar v_S^\varepsilon)
=
h_{\varepsilon,N,z}^{\rm app}(\bar v_S^\varepsilon)
+
\Lambda_\varepsilon^1
\bigl(
\bar v_S^\varepsilon
+
h_{\varepsilon,N,F}^{\rm app}(\bar v_S^\varepsilon)
\bigr).
\]
Subtracting gives
\[
\begin{aligned}
w^\varepsilon
-
h_{\varepsilon,N,w}^{\rm app}(\bar v_S^\varepsilon)
={}&
z^\varepsilon
-
h_{\varepsilon,N,z}^{\rm app}(\bar v_S^\varepsilon)
\\
&+
\Lambda_\varepsilon^1
\Big(
v_S^\varepsilon-\bar v_S^\varepsilon
+
v_F^\varepsilon
-
h_{\varepsilon,N,F}^{\rm app}(\bar v_S^\varepsilon)
\Big).
\end{aligned}
\]
The norm of \(\mathcal Z_\theta^\varepsilon\) controls
\(H^{2\theta-\delta}(Q)\), and the lifting estimate from
\Cref{subsec:thin-domain-application} gives
\[
\|\Lambda_\varepsilon^1\eta\|_{H^2(Q)}
\le
C\varepsilon^2\|\eta\|_{\mathcal Y_1}
\le
C\|\eta\|_{\mathcal Y_1}.
\]
Therefore
\[
\begin{aligned}
&
\left\|
w^\varepsilon
-
h_{\varepsilon,N,w}^{\rm app}(\bar v_S^\varepsilon)
\right\|_{H^{2\theta-\delta}(Q)}
\\
&\qquad\le
C
\Big(
\|z^\varepsilon
-
h_{\varepsilon,N,z}^{\rm app}(\bar v_S^\varepsilon)\|_{\mathcal Z_\theta^\varepsilon}
+
\|v_S^\varepsilon-\bar v_S^\varepsilon\|_{\mathcal Y_1}
\\
&\qquad\qquad
+
\|v_F^\varepsilon
-
h_{\varepsilon,N,F}^{\rm app}(\bar v_S^\varepsilon)\|_{\mathcal Y_1}
\Big).
\end{aligned}
\]
The claim follows from \Cref{thm:thin-domain-approximation}.
\end{proof}

\begin{remark}[Status of the finite-order theorem]
The construction above shows that the thin-domain approximation theorem is a
consequence of the actual variational structure from
\Cref{subsec:thin-domain-application}. The singular fast operator is expanded
through the form identity
\[
\mathfrak a_\varepsilon
=
\mathfrak a_0+\varepsilon^2\mathfrak a_1,
\]
while the \(\varepsilon\)-dependent conormal boundary condition is handled by
the coefficient hierarchy and the final correction
\[
C_N^\varepsilon(P_N^\varepsilon).
\]
Thus the defect estimate is not assumed; it is proved by recursive
cancellation of all coefficients through order \(\varepsilon^{2N-2}\). The
only additional requirement compared with the basic slow-manifold construction
is higher regularity of \(g\), \(f\), and the retained slow data, needed to
close the finite-order coefficient recursion.
\end{remark}

\section{A canonical example and the numerical comparison}\label{sec:example}

In this section we apply the existence and approximation results of Sections \ref{sec:manifolds} and \ref{sec:reduced} to the Schnakenberg system and study the effects of the thin domain both analytically and numerically.

The classical Schnakenberg system \cite{schnakenberg1979simple,maini1997spatial} is a reaction-diffusion system of the form 
\begin{equation}\label{eq:schnak}
\begin{aligned}
\partial_t u &= D_u\Delta u + a-u+u^2v,\\
\partial_t v &= D_v\Delta v + b-u^2v.
\end{aligned}
\end{equation}
Here $a,b>0$ and $D_v\gg D_u>0$ are chosen so that the homogeneous equilibrium
\[
(u_*,v_*)=\Bigl(a+b,\frac{b}{(a+b)^2}\Bigr)
\]
is stable for the reaction system but destabilized by diffusion in an appropriate parameter regime.
For an analytic overview and the existence of patterns corresponding to the various parameter regimes we refer to \cite{ward2002existence,iron2004stability}.

This system is a natural test problem for several reasons.
First, it is a standard Turing-type model in which the full thin-domain dynamics can exhibit pattern formation, while the analytical results are well understood.
Moreover, in the thin domain the averaged reduced problem is simply the one-dimensional Schnakenberg system on the base interval, which makes comparison with the full two-dimensional thin-domain PDE especially transparent.
The first manifold correction from \Cref{sec:reduced} can be computed explicitly enough to test whether the manifold-corrected reduced model improves the approximation of the patterns.

\subsection{A heterogeneous Schnakenberg model}\label{subsec:transverse-schnakenberg}

The classical Schnakenberg system on a flat thin domain with Neumann boundary conditions is an important reference case, but it is not yet an appropriate test case of the asymptotic correction procedure from \Cref{sec:reduced}. 
In a flat thin domain the geometry term \(R\) vanishes. Moreover, if the
reaction term is independent of the transverse variable \(Y\), then the
mean-zero fast equation has no forcing when evaluated on \(Y\)-independent
states. In the notation of \Cref{prop: asymptotic exp}, this gives
\(F_0(v)=0\), and hence the first transverse correction \(h_{w,1}\) vanishes.
In fact, \(w\equiv0\) is then an exactly invariant subspace for the
\(Y\)-homogeneous flat problem. Thus all transverse \(w\)-corrections vanish
in this setting. This statement concerns only the transverse mean-zero
component; if one additionally splits the averaged one-dimensional variable
into retained and fast spatial modes, the auxiliary high-mode correction
\(h_{v_F}\) may still be nonzero because the nonlinear one-dimensional
reaction term can generate unresolved spatial modes.
To obtain a genuinely nontrivial correction while keeping the same flat geometry and the same homogeneous Neumann boundary conditions on the thin faces, it is therefore natural to introduce a transverse modulation into the nonlinear reaction term.
In the literature such effects have been studied as heterogeneous perturbations to the Schnakenberg model \cite{ishii2021effect,wong2021spot}.

Let
\[
Q_\varepsilon := \Omega \times (-\varepsilon,\varepsilon), \qquad \Omega=(0,L),
\]
and pass to the rescaled variable $Y=y/\varepsilon \in (-1,1)$. We impose homogeneous Neumann conditions on all boundaries and modify the auto-catalytic term by a mean-zero transverse profile. A convenient choice is
\[
\rho(Y):=\cos(\pi Y), \qquad M\rho = \frac12\int_{-1}^1 \rho(Y)\,dY = 0,
\qquad \partial_Y\rho(\pm 1)=0.
\]
For $|\eta|<1$, so that $1+\eta\rho(Y)$ remains strictly positive, we consider the rescaled Schnakenberg system
\begin{equation}\label{eq:schnakenberg rescaled}
\begin{aligned}
\partial_t u &= D_u\Bigl(\partial_x^2 u + \frac{1}{\varepsilon^2}\partial_Y^2 u\Bigr)
+ a-u + \bigl(1+\eta\rho(Y)\bigr)u^2v,
\\
\partial_t v &= D_v\Bigl(\partial_x^2 v + \frac{1}{\varepsilon^2}\partial_Y^2 v\Bigr)
+ b- \bigl(1+\eta\rho(Y)\bigr)u^2v,
\end{aligned}
\end{equation}
subject to
\[
\partial_x u = \partial_x v =0 \quad \text{on } x=0,L,
\qquad
\partial_Y u = \partial_Y v =0 \quad \text{on } Y=\pm 1.
\]
Here $a,b>0$ and $D_v \gg D_u>0$ are chosen in the classical Turing regime for the averaged one-dimensional Schnakenberg system.

Let
\[
Mf(x):=\frac12\int_{-1}^1 f(x,Y)\,dY
\]
denote the averaging operator in the thin direction, and decompose the solution as
\[
u = U+\phi, \qquad v = V+\psi,
\qquad U:=Mu, \quad V:=Mv, \quad M\phi=M\psi=0.
\]
Next, let
\[
e_0(x):=L^{-1/2},\qquad
e_k(x):=\sqrt{\frac{2}{L}}\cos\!\left(\frac{k\pi x}{L}\right),
\qquad k\ge 1,
\]
be the orthonormal Neumann eigenbasis of \(L^2(0,L)\). Thus
\[
\partial_x^2 e_k=-\mu_k e_k,\qquad
\mu_k:=\left(\frac{k\pi}{L}\right)^2,
\qquad k\in\mathbb N_0.
\]
For a fixed \(\varepsilon>0\), choose \(N=N(\varepsilon)\) according to the
slow-fast spectral splitting used in \Cref{sec:manifolds,sec:reduced}; in
particular, the modes in the range of \(Q_N\) are the modes whose diffusion
decay is of order \(\varepsilon^{-2}\). Define
\[
P_N Z := \sum_{k=0}^N \langle Z,e_k\rangle_{L^2(0,L)} e_k,
\qquad
Q_N:=\mathrm{Id}-P_N.
\]
We apply \(P_N\) and \(Q_N\) componentwise to the averaged pair
\(Z:=(U,V)^\top\).
This yields the splitting
\[
Z = Z_S + Z_F,
\qquad
Z_S := P_N Z = (U_S,V_S)^\top,
\qquad
Z_F := Q_N Z = (U_F,V_F)^\top.
\]
Identifying \(Z_S\) and \(Z_F\) with their \(Y\)-independent extensions to
\(Q:=(0,L)\times(-1,1)\), the full state can be written as
\[
(u,v)^\top = Z_S+Z_F+w,
\qquad
w:=(\phi,\psi)^\top,
\qquad Mw=0.
\]
This is the decomposition used in \Cref{sec:manifolds,sec:reduced}: \(w\) is
the mean-zero transverse fast variable, \(Z_F\) is the fast decaying block of
the averaged variable, and \(Z_S\) is the retained slow component.

It is convenient to introduce the reaction operator
\[
\mathcal{N}(U,V,\phi,\psi;Y) := (1 + \eta \rho(Y))(U+\phi)^2(V+\psi).
\]

Then the averaged and mean-zero equations become
\begin{equation}\label{eq:Schnakenberg split}
\begin{aligned}
\partial_t w &= \frac{1}{\varepsilon^2} A_\varepsilon w + F(w,Z_F,Z_S), \\
\partial_t Z_F &= B Z_F + Q_N G(w,Z_F,Z_S), \\
\partial_t Z_S &= B Z_S + P_N G(w,Z_F,Z_S),
\end{aligned}
\end{equation}
where
\[
A_\varepsilon =
\begin{pmatrix}
D_u(\partial_Y^2 + \varepsilon^2 \partial_x^2) & 0 \\
0 & D_v(\partial_Y^2 + \varepsilon^2 \partial_x^2)
\end{pmatrix},
\qquad
B =
\begin{pmatrix}
D_u \partial_x^2 & 0 \\
0 & D_v \partial_x^2
\end{pmatrix},
\]
and
\[
F(w,Z_F,Z_S) :=
\begin{pmatrix}
-\phi + (I-M)\mathcal{N}(U,V,\phi,\psi;Y) \\
-(I-M)\mathcal{N}(U,V,\phi,\psi;Y)
\end{pmatrix},
\]
\[
G(w,Z_F,Z_S) :=
\begin{pmatrix}
a - U + M\mathcal{N}(U,V,\phi,\psi;Y) \\
b - M\mathcal{N}(U,V,\phi,\psi;Y)
\end{pmatrix},
\]
with
\[
U = U_S + U_F
\qquad \text{and} \qquad
V = V_S + V_F.
\]

\begin{remark}
The flat geometry implies
\[
A_\varepsilon=A_0+\varepsilon^2A_1,
\]
with
\[
A_0=
\begin{pmatrix}
D_u\partial_Y^2 & 0\\
0 & D_v\partial_Y^2
\end{pmatrix},
\qquad
A_1=
\begin{pmatrix}
D_u\partial_x^2 & 0\\
0 & D_v\partial_x^2
\end{pmatrix}.
\]
The operator \(A_0\) is considered on the mean-zero subspace in the
\(Y\)-variable with homogeneous Neumann boundary conditions at \(Y=\pm1\).
On this space it is invertible, since the constant transverse mode has been
removed. Together with the \(\varepsilon\)-dependent slow spectral splitting
described above, this places the example in the formal framework of
\Cref{prop: asymptotic exp}. Since \(M\rho=0\), the leading averaged equation
is still the standard one-dimensional Schnakenberg system on the retained
modes. The role of the transverse heterogeneity is to generate a nontrivial
first correction in the mean-zero \(w\)-component.
\end{remark}

\begin{remark}[A priori bounds for the split Schnakenberg example]
For the Schnakenberg system \eqref{eq:schnakenberg rescaled} with homogeneous Neumann boundary conditions and nonnegative initial data, one expects global nonnegative solutions and global-in-time a priori bounds. 
Indeed, the reaction terms are quasi-positive, so non-negativity is preserved.
Moreover, the first component satisfies
\[
\partial_t u - D_u\Bigl(\partial_x^2 + \frac{1}{\varepsilon^2}\partial_Y^2\Bigr)u
= a-u+\bigl(1+\eta\rho(Y)\bigr)u^2v \ge a-u,
\]
and therefore the parabolic comparison principle yields
\[
u(x,Y,t)\ge a(1-e^{-t}) \qquad \text{for all } (x,Y)\in Q,\ t\ge 0.
\]
Hence, for every fixed $\tau>0$,
\[
u(x,Y,t)\ge c_\tau:=a(1-e^{-\tau})>0 \qquad \text{for all } t\ge \tau.
\]
Since $\rho(Y)\in[-1,1]$ and $|\eta|<1$, we also have
\[
1+\eta\rho(Y)\ge 1-|\eta|>0.
\]
Using the lower bound for $u$ in the second equation of \eqref{eq:schnakenberg rescaled}, we obtain for $t\ge \tau$
\[
\partial_t v - D_v\Bigl(\partial_x^2 + \frac{1}{\varepsilon^2}\partial_Y^2\Bigr)v
= b-\bigl(1+\eta\rho(Y)\bigr)u^2v
\le b-(1-|\eta|)c_\tau^2\,v.
\]
Another comparison argument then gives a uniform $L^\infty$-bound for $v$ on $[\tau,\infty)$.
Together with standard global boundedness results for two-component reaction-diffusion systems and parabolic smoothing on bounded domains, this yields a bounded absorbing set for the present Schnakenberg semiflow.
For more details we refer to \cite{hollis1987global, pierre2010global}.

Consequently, for this example the polynomial nonlinearities need not be assumed globally Lipschitz a priori. 
Instead, we may work with local $C^1$- or $C^\infty$-regularity of the reaction terms together with the forward-time global bounds established above.
Only for the backward trajectories required in the Lyapunov-Perron fixed-point argument; cf. \Cref{lemma: lyapunov perron}, do we introduce a cut-off operator to enforce global bounds.
This, however, does not alter the qualitative behavior of the system as on the bounded region visited
by the solutions under consideration, the cut-off system agrees with the original Schnakenberg system, while the nonlinearities become globally Lipschitz as required by the Lyapunov-Perron argument.

\end{remark}

We now evaluate the coefficient equations from \Cref{prop: asymptotic exp} and \Cref{cor: formal reduced field} for the present example. 
Since
\[
Z_S = (U_S,V_S)^\top
\]
is independent of \(Y\), one obtains
\[
F_0(Z_S) = F(0,0,Z_S) = \eta \rho(Y)U_S^2V_S
\begin{pmatrix}
1\\
-1
\end{pmatrix},
\qquad
G_0(Z_S) = G(0,0,Z_S) =
\begin{pmatrix}
a-U_S+U_S^2V_S\\
b-U_S^2V_S
\end{pmatrix}.
\]
The first coefficient \(h_{w,1}\) is therefore explicit, whereas \(h_{v_F,1}\) is determined by the high-mode projection of the averaged reaction term.
\begin{proposition}[First transverse Schnakenberg correction]\label{prop: schnakenberg 1}
    For the split Schnakenberg system \eqref{eq:Schnakenberg split}, the first coefficients in the expansion from \Cref{prop: asymptotic exp}
    \[
    h_w^\varepsilon(Z_S) \sim \varepsilon^2 h_{w,1}(Z_S) + \varepsilon^4 h_{w,2}(Z_S) + \cdots,
    \qquad
    h_{v_F}^\varepsilon(Z_S) \sim \varepsilon^2 h_{v_F,1}(Z_S) + \varepsilon^4 h_{v_F,2}(Z_S) + \cdots
    \]
    are given by
    \begin{align} \label{eq: h1x correction}
      & h_{w,1}(U_S,V_S;Y) = \frac{\eta U_S^2V_S}{\pi^2}\rho(Y) \begin{pmatrix}
    D_u^{-1}\\
    -D_v^{-1} \end{pmatrix} ,
    \intertext{and}\label{eq: h1y correction}
    & h_{v_F,1}(Z_S) = -\tilde{B}^{-1}Q_N \begin{pmatrix}
    a-U_S+U_S^2V_S\\
    b-U_S^2V_S \end{pmatrix},
    \end{align}
   where \(\tilde{B}\) denotes the rescaled fast operator introduced in \Cref{sec:reduced}.  
\end{proposition}

\begin{proof}
By \eqref{eq:first-order-hw-corrected}, the first transverse coefficient satisfies
\[
A_0h_{w,1}(Z_S)+F_0(Z_S)=0.
\]
For the present system this becomes
\[
\begin{pmatrix}
D_u\partial_Y^2 & 0\\
0 & D_v\partial_Y^2
\end{pmatrix}
h_{w,1}
+
\eta\rho(Y)U_S^2V_S
\begin{pmatrix}
1\\
-1
\end{pmatrix}
=0.
\]
The right-hand side is mean-zero in \(Y\), since \(M\rho=0\). Hence the
equation can be inverted uniquely on the mean-zero Neumann subspace. Because
\(\partial_Y^2\rho=-\pi^2\rho\), this gives \eqref{eq: h1x correction}.

Similarly, \eqref{eq:first-order-hF-corrected} gives
\[
\widetilde B h_{v_F,1}(Z_S)+Q_NG_0(Z_S)=0.
\]
Using
\[
G_0(Z_S)=
\begin{pmatrix}
a-U_S+U_S^2V_S\\
b-U_S^2V_S
\end{pmatrix},
\]
and the invertibility of \(\widetilde B\) on the high-mode block, we obtain
\eqref{eq: h1y correction}.
\end{proof}

\begin{remark}
    The two correction mechanisms are conceptually different. The coefficient \(h_{w,1}\) is created by the transverse heterogeneity and measures the first \(Y\)-dependent correction to the manifold. 
    The coefficient \(h_{v_F,1}\) is created by the splitting of the averaged variable and measures how the nonlinear reaction term generates high spatial modes that are slaved to the retained block \(Z_S\). 
\end{remark}

We next insert the first coefficients into the reduced vector field from Equation \ref{eq:asymp-slow-dynamics-corrected}. Let
\[
J(U,V):=
\begin{pmatrix}
-1+2UV & U^2\\
-2UV & -U^2
\end{pmatrix}
\]
denote the Jacobian of the Schnakenberg reaction term
\[
R(U,V):=
\begin{pmatrix}
a-U+U^2V\\
b-U^2V
\end{pmatrix}.
\]

Then \(D_2G_0(Z_S)\xi = J(U_S,V_S)\xi\) for \(\xi \in Y_F \cap \mc Y_1\). The \(w\)-correction can also be evaluated explicitly.
\begin{corollary}[Corrected Schnakenberg reduced dynamics] \label{cor: schnakenberg 2}
  The reduced dynamics on the retained slow block \(Z_S=(U_S,V_S)^\top\) satisfy the formal expansion
  \begin{align}\label{eq: slow dyn schnakenberg}
      \partial_t Z_S=BZ_S + P_NR(Z_S)+ \varepsilon^2 P_N \left[\Theta(U_S,V_S)
    \begin{pmatrix}
    1\\-1\end{pmatrix}
    + J(U_S,V_S)h_{v_F,1}(Z_S) \right] + O(\varepsilon^4),
  \end{align}
   where
   \begin{align}\label{eq: theta term schnakenberg}
      \Theta(U_S,V_S)=\frac{\eta^2}{2\pi^2} \left(\frac{2U_S^3V_S^2}{D_u} -\frac{U_S^4V_S}{D_v}\right).
   \end{align}
Equivalently,
\begin{align*}
  \partial_t U_S
&=
D_u\partial_x^2 U_S
+ P_N(a-U_S+U_S^2V_S)
+ \varepsilon^2 P_N\bigl(\Theta(U_S,V_S)+[J(U_S,V_S)h_{v_F,1}(Z_S)]_1\bigr)
+ O(\varepsilon^4),  \\
\partial_t V_S
&=
D_v\partial_x^2 V_S
+ P_N(b-U_S^2V_S)
+ \varepsilon^2 P_N\bigl(-\Theta(U_S,V_S)+[J(U_S,V_S)h_{v_F,1}(Z_S)]_2\bigr)
+ O(\varepsilon^4).
\end{align*}
\end{corollary}

\begin{proof}
By Equation \eqref{eq:asymp-slow-dynamics-corrected},
\[
\partial_t Z_S
=
BZ_S + P_NG_0(Z_S)
+ \varepsilon^2 P_N\bigl(D_1G_0(Z_S)h_{w,1}(Z_S)+D_2G_0(Z_S)h_{v_F,1}(Z_S)\bigr)
+ O(\varepsilon^4).
\]
The zeroth-order term is \(P_NR(Z_S)\). For the first correction, the derivative with respect to the averaged fast block gives
\[
D_2G_0(Z_S)h_{v_F,1}(Z_S)=J(U_S,V_S)h_{v_F,1}(Z_S).
\]
On the other hand,
\[
D_1G_0(Z_S)[\hat{\phi},\hat{\psi}]
=
M\Bigl[(1+\eta\rho)(2U_SV_S\hat{\phi}+U_S^2\hat{\psi})\Bigr]
\begin{pmatrix}
1\\
-1
\end{pmatrix}.
\]
Substituting \eqref{eq: h1x correction} and using \(M(\rho)=0\) together with
\[
M(\rho^2)=\frac12\int_{-1}^1 \rho(Y)^2\,dY=\frac12,
\]
yields
\[
D_1G_0(Z_S)h_{w,1}(Z_S)
=
\Theta(U_S,V_S)
\begin{pmatrix}
1\\
-1
\end{pmatrix},
\]
with \(\Theta\) given by \eqref{eq: theta term schnakenberg}. This proves \eqref{eq: slow dyn schnakenberg} and its component form.
\end{proof}

\begin{remark}
\Cref{cor: schnakenberg 2} separates the two \(O(\varepsilon^2)\) effects predicted by the theory of \Cref{sec:reduced}.
The term \(\Theta(U_S,V_S)\) is the transverse-manifold correction already visible in the model without the splitting.
The additional term \(J(U_S,V_S)h_{v_F,1}(Z_S)\) is new: it comes from the splitting of the averaged variable and measures the feedback of the fast decaying spatial modes onto the retained modes.
\end{remark}

\begin{remark}[Pattern-forming interpretation]
The correction above has a natural interpretation in the classical theory of diffusion-driven instability. In the thin-domain limit the transverse Neumann modes are shifted to the fast scale by the factor \(\varepsilon^{-2}\partial_Y^2\). 
Hence the leading Turing instability is carried by the \(Y\)-independent component and is governed, to leading order, by the one-dimensional averaged Schnakenberg system on the base interval.
This is consistent with the classical Turing mechanism and with the standard activator-inhibitor interpretation of reaction--diffusion pattern formation.

The present thin-domain setting adds a further effect.
The mean-zero transverse heterogeneity has no leading-order contribution to the averaged
reaction kinetics, but it generates the transverse profile
\[
  u_\perp \sim
  \varepsilon^2
  \frac{\eta U_S^2V_S}{\pi^2D_u}\rho(Y),
  \qquad
  v_\perp \sim
  -\varepsilon^2
  \frac{\eta U_S^2V_S}{\pi^2D_v}\rho(Y).
\]
Thus the thin direction does not create an independent transverse Turing
instability; rather, it produces a small transverse pattern whose amplitude
is modulated by the dominant one-dimensional Turing mode. After insertion
into the averaged equation this profile also produces an
\(O(\varepsilon^2)\) correction to the effective reaction kinetics.
Therefore the averaged model predicts the leading Turing threshold and dominant
wavelength, while the slow-manifold-corrected model predicts the first
thin-domain correction to the bifurcation structure and reconstructs the
transverse residual missed by the purely averaged reduction.

This perspective is closely related to the literature on pattern formation
in heterogeneous reaction--diffusion systems.
Spatial heterogeneity is known to modify classical Turing thresholds, localize unstable modes, and sometimes produce patterns outside the standard homogeneous Turing regime.
See, for example the results in \cite{PageMainiMonk2003} and \cite{KrauseKlikaWoolleyGaffney2020}.
The difference here is that the heterogeneity acts in the thin transverse
direction and is filtered through the fast stable transverse dynamics. The
slow-manifold expansion therefore gives an explicit reduced description of
how this hidden transverse heterogeneity feeds back into the observable
one-dimensional pattern.
\end{remark}

\subsection{Numerical Comparison}

We now complement the analytical observations above by a direct numerical
comparison of the full thin-domain system, the leading averaged reduction,
and the first-order manifold-corrected reconstruction. Throughout this
subsection we use
\[
a=0.10,\qquad b=0.90,\qquad D_u=1.0,\qquad D_v=10.0,\qquad L=6.0,\qquad
\varepsilon=0.15,\qquad \eta=0.8.
\]
We also set \(\rho(Y)=\cos(\pi Y)\) in
\eqref{eq:schnakenberg rescaled}. The initial conditions are perturbations
of the homogeneous equilibrium \((u_*,v_*)\) along unstable spatial modes.

The computations are carried out on a finite spatial grid. For the numerical
tests we do not implement the auxiliary slow-variable splitting from the
Lyapunov-Perron construction. Instead, we work directly with the averaged
variables and set the split fast averaged component equal to zero,
\(v_F\equiv0\). Thus the simulations should be interpreted as a numerical
test of the asymptotic transverse reconstruction from \Cref{sec:reduced} and
\Cref{prop: schnakenberg 1}, rather than as a direct numerical realization
of the exact split slow manifold from \Cref{sec:manifolds}.

We evaluate the profiles at final time \(T=250\), by which the numerical
solutions have reached a steady state for the chosen parameters. To isolate
the genuinely transverse effect, we subtract the cross-sectional averages and
study
\[
u_\perp(x,Y,t):=u(x,Y,t)-Mu(x,t),\qquad
v_\perp(x,Y,t):=v(x,Y,t)-Mv(x,t).
\]
For the leading averaged model, the reconstructed fields are \(Y\)-independent,
so these residuals vanish identically. For the first-order reconstruction we
use the transverse coefficient from \Cref{prop: schnakenberg 1} and define
\[
u_{\mathrm{app}}(x,Y,t)
=
U(x,t)+\varepsilon^2 h_{1,u}(U,V;Y),
\qquad
v_{\mathrm{app}}(x,Y,t)
=
V(x,t)+\varepsilon^2 h_{1,v}(U,V;Y),
\]
where now \(U,V\) denote the averaged numerical solution used in the
no-splitting comparison. Since \(Mh_{1,u}=Mh_{1,v}=0\), the reconstructed
residuals are
\[
u_\perp^{\mathrm{app}}(x,Y,t)
=
\varepsilon^2h_{1,u}(U,V;Y),
\qquad
v_\perp^{\mathrm{app}}(x,Y,t)
=
\varepsilon^2h_{1,v}(U,V;Y).
\]

For the full system, the equilibrium solution (at time $T=250$) has a dominant pattern in the $x$-direction and a small pattern in the transverse $Y$-direction; see \Cref{fig:full1}.
\begin{figure}[ht]
    \centering
    \includegraphics[width=.9\textwidth]{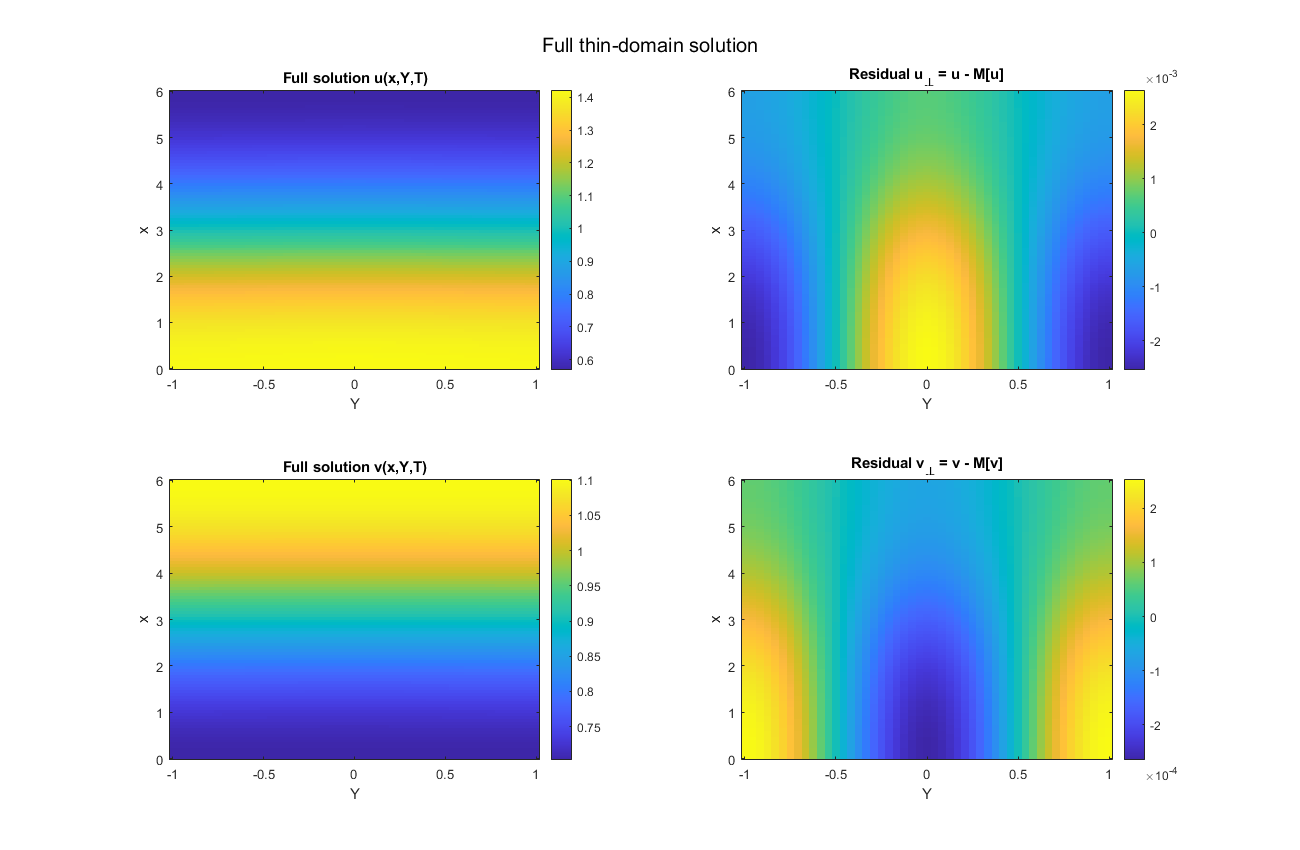}
    \caption{Equilibrium profile of the full Schnakenberg system}
    \label{fig:full1}
\end{figure}

The next question is whether we can recover the small transverse pattern in the two reduced models.

For the fully reduced model on the critical manifold, these residuals vanish identically by construction; see \Cref{fig:reduced1}.
\begin{figure}[H]
    \centering
    \includegraphics[width=.9\textwidth]{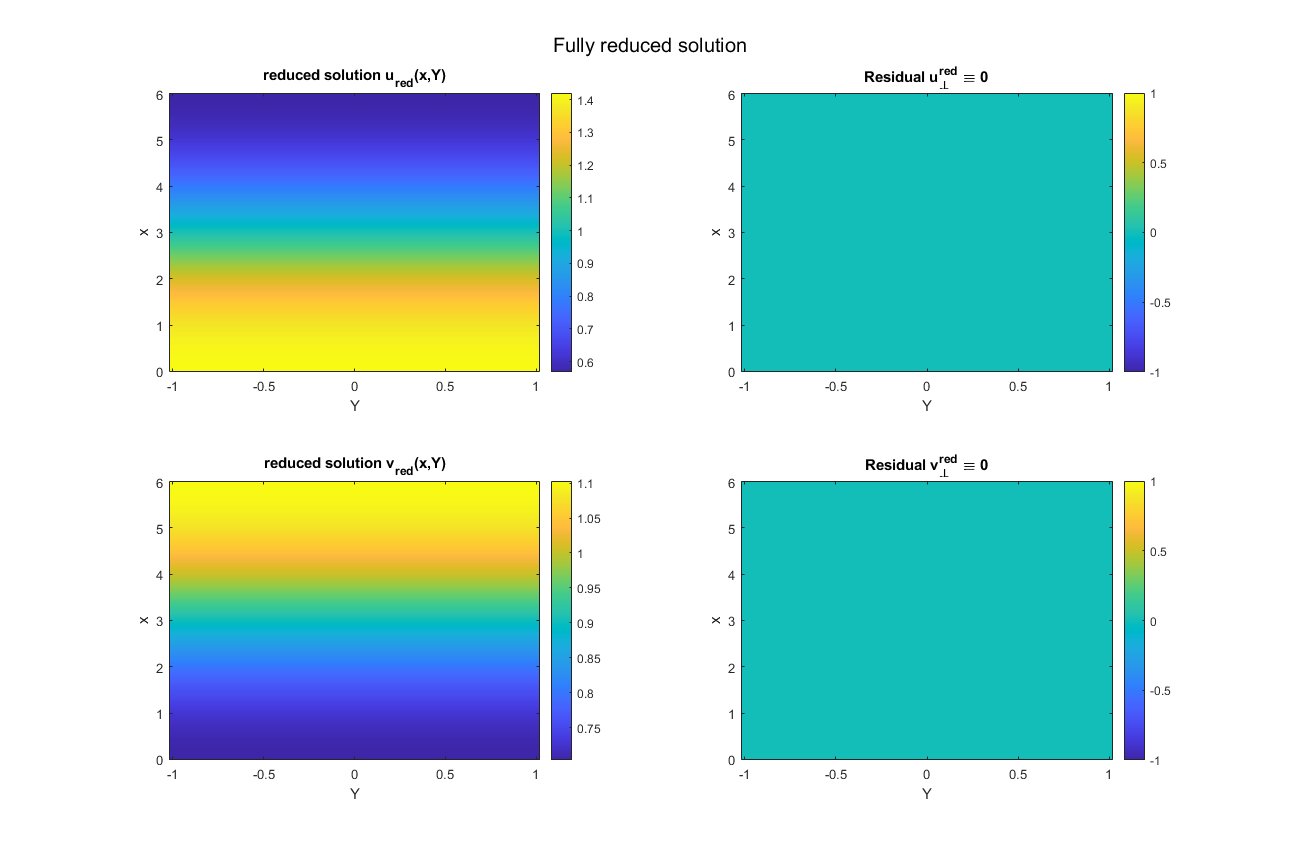}
    \caption{Equilibrium profile of averaged system}
    \label{fig:reduced1}
\end{figure}

For the first-order approximation we reconstruct
\[
u_{\mathrm{app}}(x,Y,t)=U(x,t)+\varepsilon^2 h_{1,u}(U,V;Y),\qquad
v_{\mathrm{app}}(x,Y,t)=V(x,t)+\varepsilon^2 h_{1,v}(U,V;Y),
\]
so that
\[
u_{\perp}^{\mathrm{app}}=\varepsilon^2 h_{1,u}(U,V;Y),\qquad
v_{\perp}^{\mathrm{app}}=\varepsilon^2 h_{1,v}(U,V;Y).
\]
\begin{figure}[H]
    \centering
    \includegraphics[width=.9\textwidth]{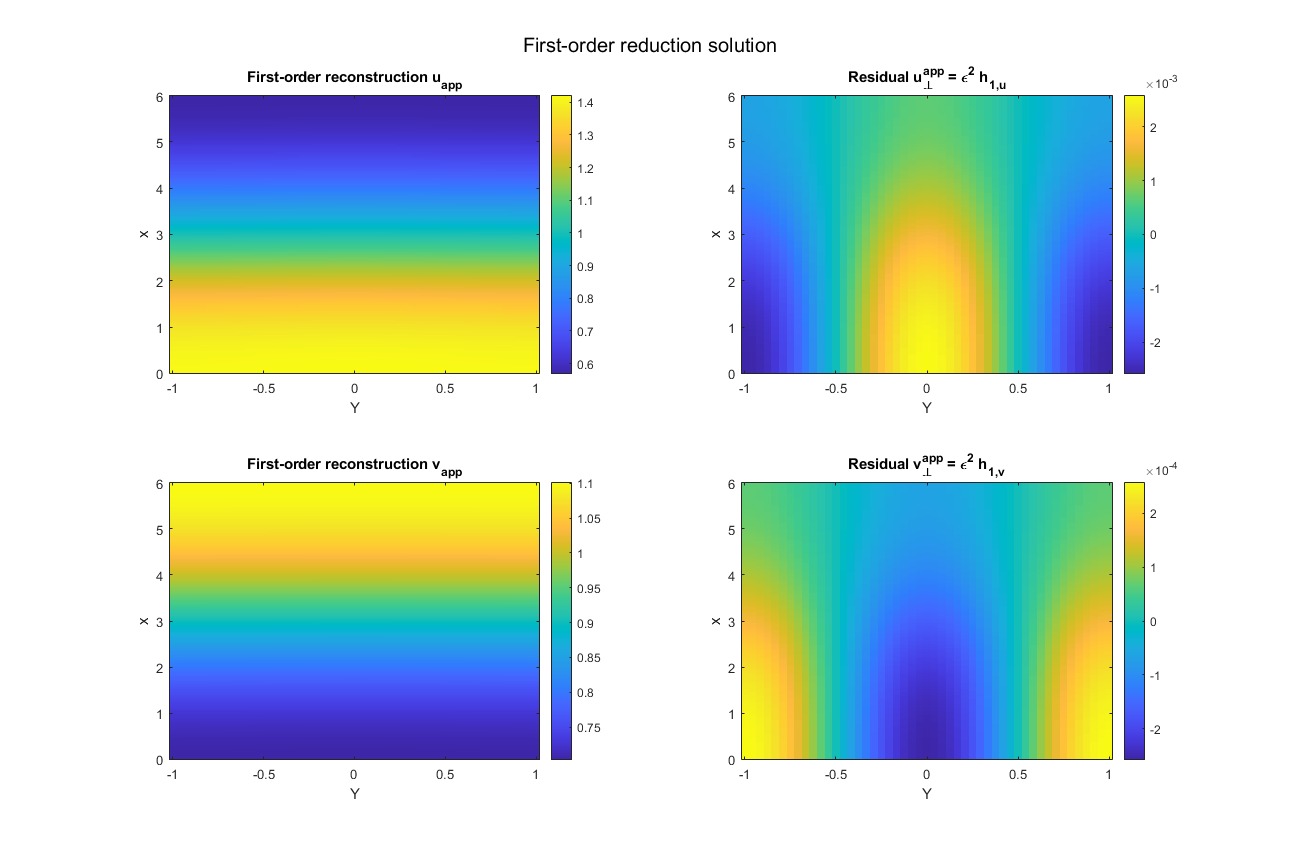}
    \caption{Equilibrium profile of first-order system}
    \label{fig:firstorder1}
\end{figure}

The numerical results show a clear separation between the dominant
one-dimensional pattern and the small transverse correction; see
\Cref{fig:firstorder1,fig:compare}. First, the averaged profile captures the
leading \(x\)-dependence of the full solution very well: the full thin-domain
fields are dominated by a slow pattern in \(x\), while the dependence on \(Y\)
appears only as a small correction. This is precisely the regime predicted by
the reduction theory.

Second, after subtracting the average, the full solution exhibits a coherent
residual profile in the thin direction. For the \(u\)-component this residual
is positive near \(Y=0\) and negative near \(Y=\pm1\), while the
\(v\)-component has the opposite sign and a substantially smaller amplitude.
This agrees with \eqref{eq: h1x correction}: since
\(\rho(Y)=\cos(\pi Y)\), the \(u\)-correction is proportional to
\(D_u^{-1}\rho(Y)\), whereas the \(v\)-correction is proportional to
\(-D_v^{-1}\rho(Y)\).

The residual field of the pure averaged model is identically zero and
therefore misses the transverse structure completely. By contrast, the
first-order reconstruction reproduces both the shape and the size of the
residual accurately. In particular, it captures the sign reversal between
\(u_\perp\) and \(v_\perp\). As \Cref{fig:compare} shows, the \(Y\)-profiles
of the full and first-order residuals nearly overlap, whereas the averaged
model has relative error of order one in the residual because it predicts no
transverse correction.

A second important observation is that the transverse residual amplitude is
not spatially uniform in \(x\). It is largest near the part of the domain
where the slow activator pattern is largest and decreases along the profile.
This is again consistent with \Cref{prop: schnakenberg 1}, because the first
correction is proportional to the slow factor \(U^2V\). Hence the transverse
heterogeneity is slaved to the dominant one-dimensional pattern: the
thin-direction profile is fixed by \(\rho(Y)\), while its amplitude is
modulated by the slow Schnakenberg variables along the base interval.

\begin{figure}[H]
\centering
\begin{subfigure}[b]{1\textwidth}
  \includegraphics[width=1\linewidth]{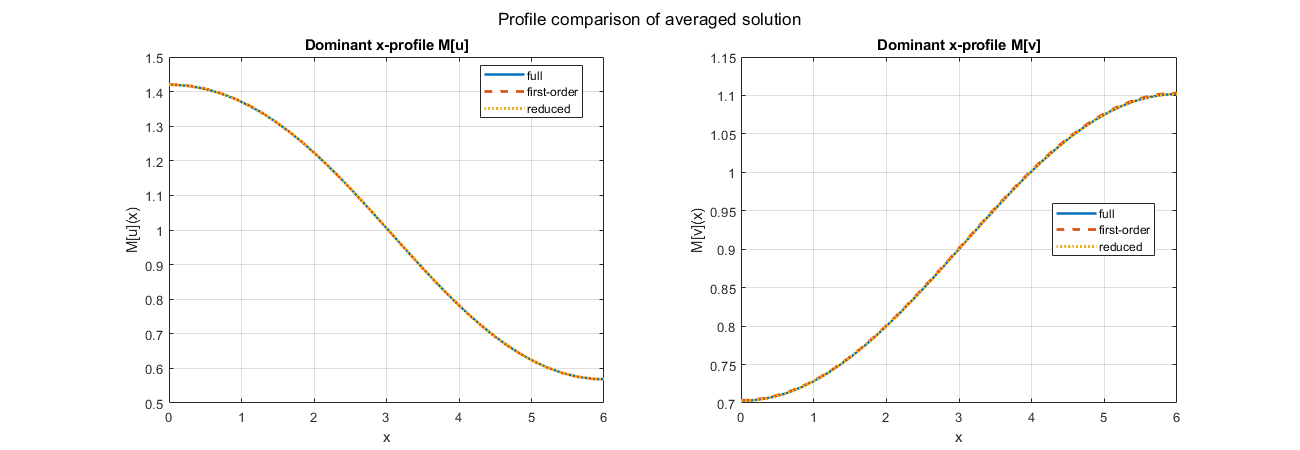}
  \label{fig:avg} 
\end{subfigure}

\medskip 
\begin{subfigure}[b]{1\textwidth}
  \includegraphics[width=1\linewidth]{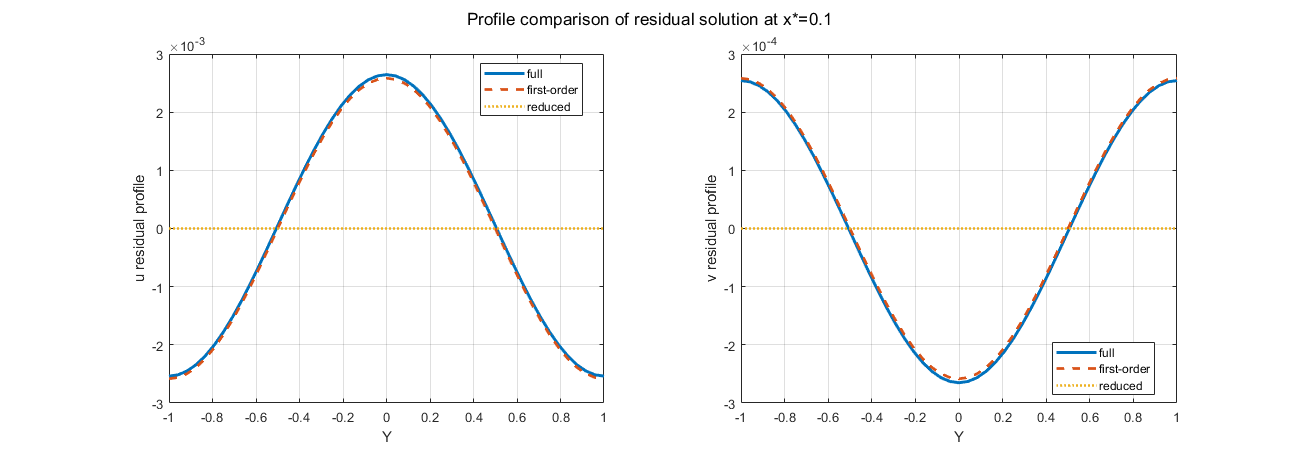}
  \label{fig:res}
\end{subfigure}

\caption{Comparison of the full, first-order, and averaged profiles}
\label{fig:compare}
\end{figure}

\subsection{Conclusion}

Taken together, the asymptotic expansion of the slow manifold and the
numerical simulations support the following interpretation. The leading-order
averaged model captures the dominant slow pattern along the base domain, but
it does not recover the small transverse structure generated by the
heterogeneous reaction term. This missing information is supplied by the
first transverse manifold correction.

In the parameter regime considered here, the correction has only a minor
effect on the dominant \(x\)-profile, but it substantially improves the
reconstruction of the residual \(Y\)-dependence. This is fully consistent
with the asymptotic theory: the first-order slow-manifold term is small in
amplitude, yet it contains the leading-order information needed to recover
the transverse profile of the full thin-domain solution.

From the perspective of pattern formation, this suggests the following
interpretation. The large transverse diffusion filters out independent
transverse Turing modes, so the dominant instability is governed to leading
order by the averaged one-dimensional Schnakenberg system. The transverse
heterogeneity is nevertheless not lost: it reappears through the
slow-manifold correction as a small transverse residual whose amplitude is
slaved to the base-domain Turing pattern. Thus the thin-domain reduction
separates two effects which are often intertwined in heterogeneous
reaction--diffusion systems: the leading diffusion-driven instability along
the extended direction, and the heterogeneity-induced modulation in the
confined direction. This connects the present example to the broader
literature on pattern formation in heterogeneous Turing systems, where
spatial heterogeneity can shift instability thresholds, localize modes, or
generate patterns outside the classical homogeneous Turing regime.

The numerical tests deliberately suppress the auxiliary split fast component
of the averaged variable by imposing \(v_F=0\). Thus they should not be read
as a direct computation of the exact split slow manifold from
\Cref{sec:manifolds}. Rather, they test the transverse reconstruction
predicted by the asymptotic expansion in \Cref{sec:reduced}. In this sense,
the simulations show that, for the transversely heterogeneous Schnakenberg
system, the first-order approximation is not merely a higher-order
perturbation of the averaged equation. It is the leading nontrivial
approximation of the transverse profile of the full solution.

\section{Application to thin tubular domains}\label{sec:tubular domains}
Another natural application and extension of the framework is to thin tubular neighborhoods of curves or hypersurfaces.
We refer to \cite{yanagida1990existence,PrizziRybakowski2002} and the references therein for an overview and to \cite{miura2026thin,DeGaspari2026} for some very recent applications.
If the thin domain is written as a tubular neighborhood of a reference set \(\Gamma\), then after passing to normal coordinates and rescaling the normal variable, the Laplacian takes a Jacobian-weighted form in which the fast transverse diffusion still appears on the \(\varepsilon^{-2}\)-scale, but the geometry is now encoded by the curvature of \(\Gamma\). 
In particular, even for constant thickness, curvature produces additional lower-order drift and tangential correction terms, and the natural averaging operator becomes the cross-sectional average weighted by the metric Jacobian. 

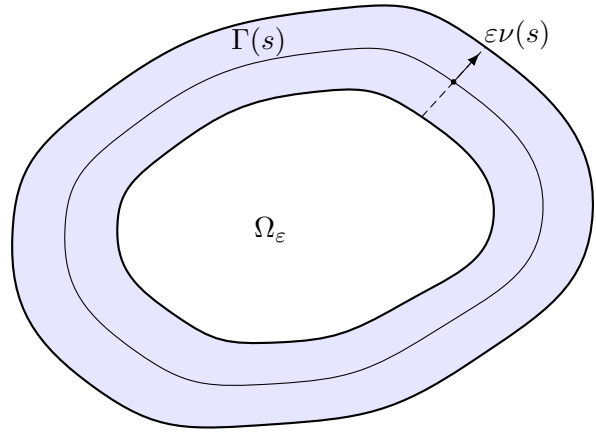
\begin{figure}[H]
\centering
\begin{tikzpicture}[scale=1.05, line cap=round, line join=round, >=Latex]


\fill[blue!10,even odd rule]
plot[smooth cycle, tension=0.85] coordinates {
  (3.90,0.12)
  (2.55,2.20)
  (0.40,2.65)
  (-2.20,1.65)
  (-3.40,-0.18)
  (-2.35,-2.10)
  (-0.10,-2.55)
  (2.35,-1.78)
}
plot[smooth cycle, tension=0.85] coordinates {
  (2.65,0.02)
  (1.75,1.32)
  (0.28,1.62)
  (-1.35,0.95)
  (-2.08,-0.12)
  (-1.35,-1.18)
  (-0.05,-1.52)
  (1.48,-1.05)
};

\draw[thick]
plot[smooth cycle, tension=0.85] coordinates {
  (3.90,0.12)
  (2.55,2.20)
  (0.40,2.65)
  (-2.20,1.65)
  (-3.40,-0.18)
  (-2.35,-2.10)
  (-0.10,-2.55)
  (2.35,-1.78)
};

\draw[thick]
plot[smooth cycle, tension=0.85] coordinates {
  (2.65,0.02)
  (1.75,1.32)
  (0.28,1.62)
  (-1.35,0.95)
  (-2.08,-0.12)
  (-1.35,-1.18)
  (-0.05,-1.52)
  (1.48,-1.05)
};

\draw[thin]
plot[smooth cycle, tension=0.85] coordinates {
  (3.27,0.07)
  (2.15,1.76)
  (0.34,2.12)
  (-1.78,1.28)
  (-2.74,-0.15)
  (-1.85,-1.63)
  (-0.08,-2.04)
  (1.88,-1.40)
};

\node at (-0.15,-0.10) {$\Omega_\varepsilon$};



\coordinate (G) at (2.15,1.76);

\coordinate (Gout) at (2.55,2.20);
\coordinate (Gin)  at (1.75,1.32);

\draw[densely dashed] (Gin) -- (Gout);

\fill (G) circle (1.pt);

\draw[->] (G) -- ($(G)!0.82!(Gout)$);
\node[above right] at ($(G)!0.65!(Gout)$) {$\varepsilon\nu(s)$};


\node[above] at (-0.3,1.96) {$\Gamma(s)$};



\end{tikzpicture}
\caption{A thin tubular neighborhood \(\Omega_\varepsilon\) of a closed reference curve \(\Gamma\).}
\label{fig:tubular-neighborhood}
\end{figure}

\begin{proposition}[Fast-slow splitting on a thin tubular neighborhood]\label{prop:tubular}
Let \(I=\mathbb T_L\) be the one-dimensional torus of length \(L\), and let
\(\Gamma:I\to\mathbb R^2\) be a \(C^2\) arclength parametrization of a closed
regular curve. Let \(\nu(s)\) denote the unit normal and let \(\kappa(s)\) be
the signed curvature, chosen so that
\[
\partial_s\nu(s)=-\kappa(s)\Gamma'(s).
\]
For \(\varepsilon>0\) sufficiently small, assume
\[
a_\varepsilon(s,Y):=1-\varepsilon Y\kappa(s)>0
\qquad\text{for all }(s,Y)\in I\times(-1,1),
\]
and define
\[
\Omega_\varepsilon
:=
\{\Gamma(s)+\varepsilon Y\nu(s):s\in I,\ Y\in(-1,1)\},
\qquad
\Omega:=I\times(-1,1).
\]
Consider
\[
\partial_t\widehat u=\Delta \widehat u+f(\widehat u)
\qquad\text{in }\Omega_\varepsilon,
\]
with homogeneous Neumann boundary conditions on the two tubular boundary
components. Set
\[
u(s,Y,t):=\widehat u(\Gamma(s)+\varepsilon Y\nu(s),t).
\]
Then the rescaled problem on \(Q\) is
\[
\partial_t u
=
\frac{1}{\varepsilon^2a_\varepsilon}
\partial_Y\!\bigl(a_\varepsilon\partial_Yu\bigr)
+
\frac{1}{a_\varepsilon}
\partial_s\!\left(\frac{1}{a_\varepsilon}\partial_su\right)
+
f(u),
\]
with periodic boundary conditions in \(s\) and homogeneous Neumann boundary
conditions
\[
\partial_Yu(s,\pm1,t)=0.
\]
Define the weighted averaging operator
\[
M_\varepsilon u(s)
:=
\frac12\int_{-1}^1 u(s,Y)a_\varepsilon(s,Y)\,dY.
\]
Since
\[
\frac12\int_{-1}^1 a_\varepsilon(s,Y)\,dY=1,
\]
\(M_\varepsilon\) is the weighted projection onto functions independent of
\(Y\). Let
\[
(\mathcal E r)(s,Y):=r(s)
\]
denote the \(Y\)-independent extension of a function on \(I\), and write
\[
v:=M_\varepsilon u,\qquad
w:=u-\mathcal E v.
\]
Then \(v\) is independent of \(Y\) and \(M_\varepsilon w=0\). Moreover, the
rescaled problem can be written as
\[
\partial_t v
=
\partial_s\!\bigl(J_\varepsilon(s)\partial_s v\bigr)
+
R_\varepsilon w
+
M_\varepsilon f(v+w),
\]
\[
\partial_t w
=
\frac{1}{\varepsilon^2a_\varepsilon}
\partial_Y\!\bigl(a_\varepsilon\partial_Yw\bigr)
+
\frac{1}{a_\varepsilon}
\partial_s\!\left(\frac{1}{a_\varepsilon}\partial_sw\right)
-
\mathcal E(R_\varepsilon w)
+
\mathcal Q_\varepsilon v
+
(I-\mathcal E M_\varepsilon)f(v+w),
\]
where
\[
J_\varepsilon(s):=\frac12\int_{-1}^1
\frac{1}{a_\varepsilon(s,Y)}\,dY,
\]
\[
R_\varepsilon w
:=
M_\varepsilon\!\left[
\frac{1}{a_\varepsilon}
\partial_s\!\left(\frac{1}{a_\varepsilon}\partial_s w\right)
\right],
\]
and
\[
\mathcal Q_\varepsilon v
:=
(I-\mathcal E M_\varepsilon)\!\left[
\frac{1}{a_\varepsilon}
\partial_s\!\left(\frac{1}{a_\varepsilon}\partial_s v\right)
\right].
\]
In particular, if \(\kappa\equiv0\), then \(a_\varepsilon\equiv1\),
\(J_\varepsilon\equiv1\), \(R_\varepsilon w=0\), and
\(\mathcal Q_\varepsilon v=0\). Thus the system reduces to the flat-strip
splitting.
\end{proposition}

\begin{proof}
Consider the change of variables
\[
\Phi_\varepsilon:\Omega\to\Omega_\varepsilon,
\qquad
\Phi_\varepsilon(s,Y)=\Gamma(s)+\varepsilon Y\nu(s).
\]
Since \(\Gamma\) is parametrized by arclength and
\(\partial_s\nu=-\kappa\Gamma'\), we have
\[
\partial_s\Phi_\varepsilon
=
(1-\varepsilon Y\kappa(s))\Gamma'(s)
=
a_\varepsilon(s,Y)\Gamma'(s),
\qquad
\partial_Y\Phi_\varepsilon
=
\varepsilon\nu(s).
\]
The coordinates are orthogonal, and the metric is
\[
g_\varepsilon
=
a_\varepsilon(s,Y)^2\,ds^2+\varepsilon^2\,dY^2,
\qquad
\sqrt{\det g_\varepsilon}=\varepsilon a_\varepsilon(s,Y).
\]
Hence the Laplace-Beltrami formula gives
\[
\Delta\widehat u
=
\frac{1}{a_\varepsilon}
\partial_s\!\left(\frac{1}{a_\varepsilon}\partial_su\right)
+
\frac{1}{\varepsilon^2a_\varepsilon}
\partial_Y\!\bigl(a_\varepsilon\partial_Yu\bigr).
\]
This proves the rescaled equation. Since the coordinate system is orthogonal
and the thin boundary components are given by \(Y=\pm1\), the homogeneous
Neumann condition on these components becomes
\[
\partial_Yu(s,\pm1,t)=0.
\]

We next apply the weighted average \(M_\varepsilon\). The transverse part drops
out because
\[
M_\varepsilon\!\left[
\frac{1}{\varepsilon^2a_\varepsilon}
\partial_Y(a_\varepsilon\partial_Yu)
\right]
=
\frac{1}{2\varepsilon^2}
\int_{-1}^1
\partial_Y(a_\varepsilon\partial_Yu)\,dY
=
\frac{1}{2\varepsilon^2}
\bigl[a_\varepsilon\partial_Yu\bigr]_{Y=-1}^{Y=1}
=
0.
\]
For the tangential part acting on the \(Y\)-independent function \(v\), we get
\[
M_\varepsilon\!\left[
\frac{1}{a_\varepsilon}
\partial_s\!\left(\frac{1}{a_\varepsilon}\partial_s v\right)
\right]
=
\frac12\int_{-1}^1
\partial_s\!\left(\frac{1}{a_\varepsilon}\partial_s v\right)\,dY
=
\partial_s\!\bigl(J_\varepsilon(s)\partial_s v\bigr).
\]
The remaining averaged tangential contribution is exactly \(R_\varepsilon w\).
This yields the equation for \(v\).

Subtracting the \(Y\)-independent extension of the averaged equation from the
full equation gives the equation for \(w\). The singular transverse term only
acts on \(w\), while the tangential part acting on \(v\) contributes
\(\mathcal Q_\varepsilon v\), and the averaged tangential part of the
\(w\)-equation is subtracted as \(\mathcal E(R_\varepsilon w)\). This gives
the stated system.

If \(\kappa\equiv0\), then \(a_\varepsilon\equiv1\). Therefore
\[
J_\varepsilon\equiv1,\qquad
R_\varepsilon w
=
M_\varepsilon(\partial_s^2w)
=
\partial_s^2(M_\varepsilon w)
=
0,
\]
and
\[
\mathcal Q_\varepsilon v
=
(I-\mathcal E M_\varepsilon)(\partial_s^2v)
=
0.
\]
Thus the flat-strip splitting is recovered.
\end{proof}

\begin{remark}[Open curves and end caps]
The statement of \Cref{prop:tubular} is written for a closed reference curve
in order to avoid additional boundary terms at the ends of the tube. If
\(I=(0,L)\) and the tubular neighborhood has end caps, then the physical
homogeneous Neumann condition also has to be transformed on the boundary
components \(s=0\) and \(s=L\). In the coordinates above this gives
\[
\partial_su=0\qquad\text{on }\{0,L\}\times(-1,1).
\]
After applying the weighted average, this condition induces coupled boundary
conditions for \(v\) and \(w\), because the weight \(a_\varepsilon(s,Y)\)
depends on \(s\). These boundary conditions are analogous to the lateral
boundary conditions in \Cref{sec:general-thin} and should be treated
separately. The periodic formulation avoids this additional boundary
analysis and isolates the transverse fast-slow mechanism.
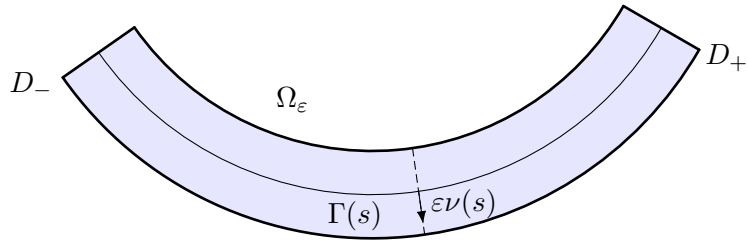
\begin{figure}[ht]
\centering
\begin{tikzpicture}[scale=1.05, line cap=round, line join=round, >=Latex]

    \def\R{4.2}          
    \def\eps{0.55}       
    \def\thA{215}        
    \def\thB{330}        

    \fill[blue!10]
        ({(\R+\eps)*cos(\thA)},{(\R+\eps)*sin(\thA)})
        arc[start angle=\thA,end angle=\thB,radius=\R+\eps]
        --
        ({(\R-\eps)*cos(\thB)},{(\R-\eps)*sin(\thB)})
        arc[start angle=\thB,end angle=\thA,radius=\R-\eps]
        -- cycle;

    \draw[black,line width=1.0pt]
        ({(\R+\eps)*cos(\thA)},{(\R+\eps)*sin(\thA)})
        arc[start angle=\thA,end angle=\thB,radius=\R+\eps];

    \draw[black,line width=1.0pt]
        ({(\R-\eps)*cos(\thA)},{(\R-\eps)*sin(\thA)})
        arc[start angle=\thA,end angle=\thB,radius=\R-\eps];

    \draw[black,line width=1.0pt]
        ({(\R-\eps)*cos(\thA)},{(\R-\eps)*sin(\thA)})
        -- ({(\R+\eps)*cos(\thA)},{(\R+\eps)*sin(\thA)});

    \draw[black,line width=1.0pt]
        ({(\R-\eps)*cos(\thB)},{(\R-\eps)*sin(\thB)})
        -- ({(\R+\eps)*cos(\thB)},{(\R+\eps)*sin(\thB)});

    \draw[black,thin]
        ({\R*cos(\thA)},{\R*sin(\thA)})
        arc[start angle=\thA,end angle=\thB,radius=\R];

    \def\thM{278}
    \coordinate (G) at ({\R*cos(\thM)},{\R*sin(\thM)});
    \coordinate (Gin) at ({(\R-\eps)*cos(\thM)},{(\R-\eps)*sin(\thM)});
    \coordinate (Gout) at ({(\R+\eps)*cos(\thM)},{(\R+\eps)*sin(\thM)});

    \draw[densely dashed] (Gin) -- (Gout);

     \draw[->,thin] (G) -- ($(G)!0.75!(Gout)$);

    \node at (-1.0,-3.0) {$\Omega_\varepsilon$};
    \node at ($(G)+(-0.8,-0.3)$) {$\Gamma(s)$};
    \node at ($(G)!0.55!(Gout)+(0.53,0.15)$) {$\varepsilon \nu(s)$};


    \node[black] at ({(\R+\eps)*cos(\thA)-0.4},{(\R+\eps)*sin(\thA)-0.1}) {$D_-$};
    \node[black] at ({(\R+\eps)*cos(\thB)+0.35},{(\R+\eps)*sin(\thB)-0.1}) {$D_+$};

\end{tikzpicture}
\caption{A thin tubular neighborhood $\Omega_\varepsilon$.}
\label{fig:exact-tubular-domain}
\end{figure}
\end{remark}

\begin{remark}[Space curves in $\mathbb{R}^3$]
The same rescaling idea extends to tubular neighborhoods of curves embedded in $\mathbb{R}^3$, where the transverse variable is now two-dimensional. Let $\Gamma:I\to\mathbb{R}^3$ be an arclength parametrization of a smooth curve and choose an orthonormal normal frame $E_1(s),E_2(s)$. If one uses a parallel transported (Bishop) frame, then a tubular neighborhood with square cross-section can be written as
\[
\Omega_\varepsilon
=
\left\{
\Gamma(s)+\varepsilon\bigl(Y_1E_1(s)+Y_2E_2(s)\bigr):
s\in I,\ (Y_1,Y_2)\in[-1,1]^2
\right\}.
\]
In these coordinates the metric factor becomes
\[
a_\varepsilon(s,Y_1,Y_2):=1-\varepsilon\bigl(\kappa_1(s)Y_1+\kappa_2(s)Y_2\bigr),
\]
where $\kappa_1,\kappa_2$ are the curvature components with respect to the chosen frame, and the rescaled operator takes the form
\[
\partial_t u
=
\frac{1}{\varepsilon^2a_\varepsilon}
\left[
\partial_{Y_1}\!\bigl(a_\varepsilon \partial_{Y_1}u\bigr)
+
\partial_{Y_2}\!\bigl(a_\varepsilon \partial_{Y_2}u\bigr)
\right]
+
\frac{1}{a_\varepsilon}\partial_s\!\left(\frac{1}{a_\varepsilon}\partial_s u\right)
+
f(u).
\]
Thus the fast operator is now the weighted Neumann Laplacian on the two-dimensional cross-section, and the splitting from \Cref{sec:general-thin} carries over after replacing the one-dimensional averaging operator by
\[
M_\varepsilon u(s)
:=
\frac14\int_{-1}^1\int_{-1}^1
u(s,Y_1,Y_2)\,a_\varepsilon(s,Y_1,Y_2)\,dY_1\,dY_2.
\]

If one instead uses a circular cross-section and polar coordinates
\[
\Omega_\varepsilon
=
\left\{
\Gamma(s)+\varepsilon r\bigl(\cos\phi\,E_1(s)+\sin\phi\,E_2(s)\bigr):
s\in I,\ r\in(0,1),\ \phi\in(0,2\pi)
\right\},
\]
then
\[
a_\varepsilon(s,r,\phi)
=
1-\varepsilon r\bigl(\kappa_1(s)\cos\phi+\kappa_2(s)\sin\phi\bigr),
\]
and the rescaled equation becomes
\[
\partial_t u
=
\frac{1}{\varepsilon^2a_\varepsilon r}
\partial_r\!\bigl(a_\varepsilon r\,\partial_r u\bigr)
+
\frac{1}{\varepsilon^2a_\varepsilon r^2}
\partial_\phi\!\bigl(a_\varepsilon \partial_\phi u\bigr)
+
\frac{1}{a_\varepsilon}\partial_s\!\left(\frac{1}{a_\varepsilon}\partial_s u\right)
+
f(u).
\]
At \(r=0\) the usual regularity condition for polar coordinates is imposed,
and the physical Neumann condition on the lateral boundary of the tube becomes
\(\partial_r u=0\) at \(r=1\).
In this setting the natural weighted averaging operator is
\[
M_\varepsilon u(s)
:=
\frac{1}{\pi}
\int_0^{2\pi}\int_0^1
u(s,r,\phi)\,a_\varepsilon(s,r,\phi)\,r\,dr\,d\phi.
\]
Hence the overall structure is unchanged: the transverse diffusion still acts on the $\varepsilon^{-2}$-scale, while the geometry enters through curvature-dependent lower-order corrections. If one uses the Frenet frame instead of a Bishop frame, additional torsion-induced first-order coupling terms appear, but these remain lower-order geometry contributions and do not alter the basic fast-slow mechanism.
\end{remark}

Once the equation is written in rescaled normal coordinates, one again obtains
a fast-slow system with singular transverse diffusion on the
\(\varepsilon^{-2}\)-scale and an \(O(1)\) tangential part carrying
geometry-dependent correction terms. The following lemma records the basic
semigroup estimate for the projected tubular fast operator.

\begin{lemma}[Fast operator on the tubular mean-zero space]
\label{lem:tubular-fast-operator}
Let
\[
H_\varepsilon:=L^2(\Omega,a_\varepsilon\,ds\,dY),
\qquad
\mathcal X_\varepsilon:=\ker M_\varepsilon\subset H_\varepsilon,
\]
and let
\[
\mc V_\varepsilon
:=
\{w\in H^1(\Omega):M_\varepsilon w=0,\ \partial_Yw=0
\text{ on }Y=\pm1\text{ in the weak Neumann sense}\}.
\]
Assume periodic boundary conditions in \(s\). Define the symmetric form
\[
\mathfrak a_\varepsilon(w,z)
:=
-\frac{1}{\varepsilon^2}
\int_\Omega a_\varepsilon\,\partial_Yw\,\partial_Yz\,ds\,dY
-
\int_\Omega \frac{1}{a_\varepsilon}\,\partial_sw\,\partial_sz\,ds\,dY,
\qquad
w,z\in \mc V_\varepsilon.
\]
Then \(\mathfrak a_\varepsilon\) is densely defined, closed, symmetric and
non-positive on \(\mathcal X_\varepsilon\). Hence it generates a
self-adjoint, sectorial operator
\[
\mathcal A_\varepsilon:D(\mathcal A_\varepsilon)\subset
\mathcal X_\varepsilon\to\mathcal X_\varepsilon
\]
and an analytic contraction \(C_0\)-semigroup on \(\mathcal X_\varepsilon\).
For smooth functions in the operator domain,
\[
\mathcal A_\varepsilon w
=
\frac{1}{\varepsilon^2a_\varepsilon}
\partial_Y\!\bigl(a_\varepsilon\partial_Yw\bigr)
+
(I-\mathcal E M_\varepsilon)
\left[
\frac{1}{a_\varepsilon}
\partial_s\!\left(\frac{1}{a_\varepsilon}\partial_sw\right)
\right].
\]
Equivalently,
\[
\mathcal A_\varepsilon w
=
\frac{1}{\varepsilon^2a_\varepsilon}
\partial_Y\!\bigl(a_\varepsilon\partial_Yw\bigr)
+
\frac{1}{a_\varepsilon}
\partial_s\!\left(\frac{1}{a_\varepsilon}\partial_sw\right)
-
\mathcal E(R_\varepsilon w).
\]
Moreover, there exists \(c_0>0\), independent of sufficiently small
\(\varepsilon\), such that
\[
\sigma(\mathcal A_\varepsilon)
\subset
\left\{z\in\mathbb C:\operatorname{Re}z\le
-\frac{c_0}{\varepsilon^2}\right\}.
\]
\end{lemma}

\begin{proof}
For sufficiently small \(\varepsilon\), the weight satisfies
\[
0<c_*\le a_\varepsilon(s,Y)\le c^*<\infty
\qquad\text{uniformly on }\Omega.
\]
Thus \(H_\varepsilon\) is uniformly equivalent to \(L^2(\Omega)\), and
\(\mc V_\varepsilon\) is a closed subspace of \(H^1(\Omega)\). The form
\(\mathfrak a_\varepsilon\) is therefore densely defined, closed, symmetric
and non-positive on \(\mathcal X_\varepsilon\). By the representation theorem
for closed symmetric forms, it defines a self-adjoint non-positive operator
\(\mathcal A_\varepsilon\), and hence \(\mathcal A_\varepsilon\) is sectorial
of angle \(0\) and generates an analytic contraction semigroup.

For smooth \(w,z\in \mc V_\varepsilon\), integration by parts gives
\[
\mathfrak a_\varepsilon(w,z)
=
\int_\Omega
\left[
\frac{1}{\varepsilon^2a_\varepsilon}
\partial_Y(a_\varepsilon\partial_Yw)
+
\frac{1}{a_\varepsilon}
\partial_s\!\left(\frac{1}{a_\varepsilon}\partial_sw\right)
\right]
z\,a_\varepsilon\,ds\,dY.
\]
Since the test function \(z\) has weighted mean zero, adding or subtracting
the \(Y\)-independent function
\[
\mathcal E(R_\varepsilon w)
=
\mathcal E M_\varepsilon
\left[
\frac{1}{a_\varepsilon}
\partial_s\!\left(\frac{1}{a_\varepsilon}\partial_sw\right)
\right]
\]
does not change the pairing with \(z\). This identifies the strong expression
for \(\mathcal A_\varepsilon\) on smooth elements of the operator domain.

It remains to prove the uniform spectral gap. The weighted mean-zero
condition implies a uniform transverse Poincare inequality: there exists
\(C_P>0\), independent of sufficiently small \(\varepsilon\), such that
\[
\|w\|_{H_\varepsilon}^2
\le
C_P\|\partial_Yw\|_{H_\varepsilon}^2,
\qquad
w\in \mc V_\varepsilon.
\]
Indeed, this follows from the usual one-dimensional Neumann Poincare
inequality on \((-1,1)\), the uniform equivalence of the weights
\(a_\varepsilon\), and the fact that \(a_\varepsilon=1+O(\varepsilon)\)
uniformly. Therefore
\[
-\mathfrak a_\varepsilon(w,w)
\ge
\frac{1}{\varepsilon^2}
\int_\Omega a_\varepsilon|\partial_Yw|^2\,ds\,dY
\ge
\frac{1}{C_P\varepsilon^2}
\|w\|_{H_\varepsilon}^2.
\]
The spectral theorem for self-adjoint operators gives the claimed spectral
inclusion with \(c_0=C_P^{-1}\).
\end{proof}

\begin{remark}[Identification with a fixed fast space]
The weighted fast space \(\mathcal X_\varepsilon=\ker M_\varepsilon\) depends
on \(\varepsilon\). Hence the abstract framework from \Cref{sec:manifolds}
does not apply literally without an identification of these spaces. A
convenient fixed reference space is
\[
\mathcal X_0:=\left\{w\in L^2(\Omega):\frac12\int_{-1}^1w(s,Y)\,dY=0
\text{ for a.e. }s\right\}.
\]
For \(z\in\mathcal X_0\), define
\[
(\mathcal T_\varepsilon z)(s,Y)
:=
z(s,Y)-M_\varepsilon z(s).
\]
Then \(\mathcal T_\varepsilon z\in\mathcal X_\varepsilon\), and
\(\mathcal T_\varepsilon:\mathcal X_0\to\mathcal X_\varepsilon\) is an
isomorphism for sufficiently small \(\varepsilon\). Its inverse is
\[
(\mathcal S_\varepsilon w)(s,Y)
:=
w(s,Y)-\frac12\int_{-1}^1w(s,Y')\,dY',
\qquad
w\in\mathcal X_\varepsilon.
\]
Moreover, \(\mathcal T_\varepsilon=I+O(\varepsilon)\) in operator norm on
\(L^2(Q)\), and the corresponding Sobolev-space maps are uniformly bounded
on bounded \(\varepsilon\)-intervals. Thus the tubular fast operators can be
transported to a fixed fast space before applying the abstract
Lyapunov-Perron or asymptotic framework.
\end{remark}

\begin{remark}[Remaining challenges for applying the abstract framework]
\label{rem:tubular-remaining-challenges}
The tubular splitting in \Cref{prop:tubular} has the same formal
fast-slow structure as the thin-domain splitting studied in
\Cref{sec:manifolds}. Nevertheless, applying the abstract results of
\Cref{sec:manifolds} to the tubular setting still requires several additional
verification steps.

First, the fast spaces depend on \(\varepsilon\), since the mean-zero
condition is defined through the weighted average \(M_\varepsilon\). The
identification map \(\mathcal T_\varepsilon:\mathcal X_0\to\mathcal X_\varepsilon\)
therefore has to be used to transport the fast equation to a fixed Hilbert
space. After this transport, one has to verify that the resulting family of
fast operators satisfies the assumptions of \Cref{ass: operators} uniformly,
including the sectorial estimates, the exponential stability, and the
corresponding fractional-domain estimates.

Second, the slow operator is no longer simply the flat one-dimensional
Neumann Laplacian. The leading averaged tangential operator is
\[
B_\varepsilon v
:=
\partial_s\!\bigl(J_\varepsilon(s)\partial_s v\bigr),
\qquad
J_\varepsilon(s)
=
\frac12\int_{-1}^1\frac{1}{a_\varepsilon(s,Y)}\,dY .
\]
Thus one has to choose an \(\varepsilon\)-dependent spectral splitting of the
slow space which is compatible with \(B_\varepsilon\), prove that the
spectral projections commute with the slow operator, and establish the
high-mode resolvent and decay estimates required in
\Cref{ass: splitting slow comp}. Since \(B_\varepsilon=B_0+O(\varepsilon^2)\) for
smooth closed curves, this should follow from perturbation theory for
self-adjoint elliptic operators on the circle, but the details have to be
checked.

Third, the curvature-dependent terms
\[
R_\varepsilon w
=
M_\varepsilon\!\left[
\frac{1}{a_\varepsilon}
\partial_s\!\left(\frac{1}{a_\varepsilon}\partial_s w\right)
\right],
\qquad
\mathcal Q_\varepsilon v
=
(I-\mathcal E M_\varepsilon)
\left[
\frac{1}{a_\varepsilon}
\partial_s\!\left(\frac{1}{a_\varepsilon}\partial_s v\right)
\right],
\]
must be estimated in the same spaces used for the abstract nonlinearities.
In particular, one has to prove that these terms are lower order with respect
to the singular transverse operator and that they are uniformly Lipschitz on
the relevant interpolation spaces. These estimates are the tubular analogue
of the geometry-induced estimates in the variable-thickness case.

Fourth, the weighted projection \(M_\varepsilon\) and the extension
\(\mathcal E\) have to be shown to interact well with the nonlinear reaction
term. For globally Lipschitz nonlinearities this is straightforward, but for
polynomial reaction terms one would again need either a priori bounds or a
cut-off argument, as in the flat and variable-thickness settings.

Finally, to apply the asymptotic expansion results from
\Cref{sec:reduced}, one also has to expand the geometric quantities
\(a_\varepsilon\), \(M_\varepsilon\), \(B_\varepsilon\), \(R_\varepsilon\),
and \(\mathcal Q_\varepsilon\) in powers of \(\varepsilon\), with remainders
controlled in the chosen operator norms. In the planar closed-curve case this
is mainly a curvature-dependent perturbation problem. For space curves or
higher-dimensional tubes, additional torsion and normal-frame terms appear
and have to be included in the same estimates.

Thus \Cref{prop:tubular} and \Cref{lem:tubular-fast-operator} provide the
correct structural and fast-semigroup input, but a full application of the
abstract slow-manifold theorem requires, in addition, a fixed-space
formulation, a compatible slow spectral splitting, uniform bounds for the
curvature-induced coupling terms, and a controlled expansion of the weighted
geometric operators.
\end{remark}

\section{Discussion and outlook}\label{sec:discussion}

The main structural message of the paper is that thin-domain reduction is
governed by the interaction between boundary conditions, transverse spectral
structure, and geometry. Homogeneous Neumann conditions on the thin faces form
the distinguished regime for genuine lower-dimensional reduction, because
they preserve the transverse zero mode. Homogeneous Dirichlet conditions
remove this mode and therefore force decay on the fast transverse time scale,
whereas inhomogeneous boundary data may generate liftings, singular forcing
terms, or nontrivial transverse profiles. Thus the first step in any
thin-domain reduction is to understand which boundary conditions preserve a
slow averaged variable.

For thin domains with variable thickness, the averaging procedure leads to a
fast-slow system in which the averaged variable evolves on the base domain
and the mean-zero transverse component is acted on by a strongly stable
operator on the \(\varepsilon^{-2}\)-scale. The variable geometry does not
merely perturb the flat equation: it also enters through transformed boundary
conditions and trace terms, which produce the geometry-induced coupling in the
averaged equation. The construction in \Cref{sec:general-thin,sec:manifolds}
therefore shows that a valid reduction requires both the transverse spectral
gap and a trace-compatible functional-analytic framework.

On the dynamical side, this work identifies two complementary reduction
mechanisms. The splitting-based Lyapunov-Perron construction gives an exact
slow manifold provided the abstract hypotheses from \Cref{sec:manifolds} are
satisfied. In particular, one needs a uniformly stable fast operator, suitable
sectorial estimates, globally Lipschitz or localized nonlinearities, and a
spectral splitting of the slow variable compatible with the unbounded slow
operator. This last point is the main dynamical obstruction: without such a
splitting, the slow evolution cannot in general be continued backward on the
full slow phase space.

The second mechanism is the asymptotic construction from the invariance
equation. Under the assumptions of \Cref{sec:reduced}, the approximate graph
obtained by solving the coefficient equations gives a controlled approximation
to the full dynamics on finite time intervals. In the thin-strip setting,
\Cref{subsec:thin-domain-approximation-theorem} proves this statement to finite
order in the lifted variables by combining the form expansion
\(\mathfrak a_\varepsilon=\mathfrak a_0+\varepsilon^2\mathfrak a_1\) with the
conormal coefficient hierarchy. More generally, when an exact slow manifold
exists, the same expansion approximates the exact graph in the relevant
topology. What remains open beyond the present theorem is the reliability of
such approximate reduced objects in regimes where neither the exact
Lyapunov--Perron construction nor the finite-order conormal closure is
available.

The heterogeneous Schnakenberg example illustrates the practical role of the
first transverse correction. The leading averaged equation captures the
dominant one-dimensional pattern along the base interval, but it cannot
recover the \(Y\)-dependent residual. The first-order transverse correction
does recover this residual to leading order. In the numerical comparison we
deliberately suppress the auxiliary split fast component of the averaged
variable by setting \(v_F=0\). Hence the computations should be read as a test
of the asymptotic transverse reconstruction, rather than as a direct numerical
realization of the full split Lyapunov--Perron manifold.

Several extensions remain natural. First, the theory should be developed for
higher-dimensional base domains. The formal splitting is similar, but the
analysis requires higher-dimensional versions of the trace estimates,
elliptic regularity results, and spectral decompositions used here. Second,
one can study more general non-flat geometries, including curved or tubular
domains. The tubular calculation in \Cref{sec:tubular domains} shows that the same
fast transverse mechanism persists after replacing the flat average by a
Jacobian-weighted average. A full application of the abstract slow-manifold
results in that setting, however, still requires the fixed-space
identification of the weighted mean-zero spaces, a compatible tangential
spectral splitting, and uniform control of the curvature- and torsion-induced
coupling terms.

A further direction is to incorporate boundary conditions that generate
nontrivial liftings or boundary layers. The analysis in \Cref{sec:boundary}
shows that inhomogeneous Neumann data can create singular averaged forcing
unless an appropriate compatibility scaling is imposed, while inhomogeneous
Dirichlet data may leave an order-one transverse profile. Such regimes are
not merely perturbations of the homogeneous Neumann setting; they may require
a different choice of slow variables or a renormalized limit problem.

Finally, an important open problem is the combination of the present
slow-manifold viewpoint with homogenization in thin domains with oscillatory
thickness, for instance
\[
Q_\varepsilon
=
\{(x,y):x\in\Omega,\ 0<y<\varepsilon g(x,x/\varepsilon)\}.
\]
See \cite{arrieta2011semilinear,arrieta2011homogenization}. In this regime,
the rescaling used in \Cref{sec:general-thin} produces coefficients depending
on both \(x\) and \(x/\varepsilon\). The transformed tangential derivatives
and the transformed Neumann condition then contain additional singular
oscillatory contributions. Consequently, the one-scale averaging argument of
the present paper is no longer sufficient by itself. A satisfactory reduction
theory should combine the fast transverse relaxation mechanism with a
two-scale homogenization analysis. Understanding how invariant-manifold
corrections interact with homogenized geometry-induced terms appears to be a
promising direction for future work.


\appendix

\section{Appendix: Uniform conormal estimates for the projected fast operator and finite-order liftings}
\label{app:uniform-parameter-elliptic-estimate}

In this appendix we isolate the conormal regularity inputs used in
\Cref{subsec:thin-domain-application} and
\Cref{subsec:thin-domain-approximation-theorem}. The first input is the
\(H^2\)-level parameter-elliptic estimate used in
Lemma~\ref{lem:trace-compatible-fast-smoothing}, which yields the uniform
graph-norm estimate for the projected fast operator. The second input is the
finite-order conormal trace lifting used in the coefficient hierarchy of
\Cref{lem:form-expansion-conormal-hierarchy}. The purpose is not to reprove the
full parameter-dependent elliptic regularity theory. Rather, we state the
precise conormal estimates needed here, verify that the geometry-induced
projection term is lower order at the graph-norm level, and record the
finite-order trace right inverses used in
\Cref{subsec:thin-domain-approximation-theorem}.

The estimate used in \Cref{subsec:thin-domain-application} is
\begin{equation}
\label{eq:uniform-parameter-elliptic-estimate-app}
\|u\|_{H^2(Q)}
\le
C\bigl(
\|L_\varepsilon u\|_{\mathcal X}
+
\|u\|_{\mathcal X}
\bigr),
\qquad
u\in D(L_\varepsilon),
\end{equation}
where \(L_\varepsilon:=\varepsilon^{-2}A_\varepsilon\), and where the constant
\(C\) is independent of \(0<\varepsilon\le \varepsilon_0\).

We recall the notation
\[
Q=(0,L)\times(0,1),
\qquad
q(x):=\frac{g'(x)}{g(x)},
\qquad
\mathcal D:=\partial_x-Yq(x)\partial_Y,
\]
and
\[
Mu(x):=\int_0^1u(x,Y)\,dY,
\qquad
\mathcal X:=\ker M.
\]
The projected fast operator has the formal strong expression
\begin{align*}
  L_\varepsilon u
=
\frac1{\varepsilon^2g(x)^2}\partial_Y^2u
+
\mathcal D^2u
-
\mathcal E(R_gu),
\end{align*}
where
\[
R_gu
:=
-\frac1g\partial_x\bigl(g'u(\cdot,1)\bigr),
\qquad
\mathcal E\phi(x,Y):=\phi(x).
\]
The corresponding homogeneous conormal boundary conditions are
\[
\partial_Yu=0
\qquad\text{on }Y=0,
\]
\[
\partial_Yu-\varepsilon^2gg'\mathcal Du=0
\qquad\text{on }Y=1,
\]
and
\[
\mathcal Du=0
\qquad\text{on }\partial\Omega\times(0,1).
\]

\subsection*{A.1. The local conormal problem}

We first remove the projection term. Define the local operator
\begin{align}
\mathscr L_\varepsilon u
:=
\frac1{\varepsilon^2g(x)^2}\partial_Y^2u+\mathcal D^2u.
\label{eq:local-operator-appendix}
\end{align}
Thus
\[
L_\varepsilon u
=
\mathscr L_\varepsilon u-\mathcal E(R_gu).
\]

We shall use the following standard parameter-elliptic regularity theorem.
It is a scalar second-order version of the Agmon-Douglis-Nirenberg
elliptic estimate with parameter, combined with the regularity theory for
polygonal domains with compatible conormal boundary conditions; see
\cite{AgmonDouglisNirenberg1959,AgranovichVishik1964,Grisvard1985,McLean2000}.

The regularity result used below is not new as an abstract elliptic theorem.
It is a standard consequence of parameter-dependent
Agmon-Douglis-Nirenberg estimates for scalar second-order elliptic boundary
value problems with conormal boundary conditions, together with regularity
theory for compatible conormal problems on polygonal domains. The role of the
following discussion is to record why the present rescaled operator
\[
\mathscr L_\varepsilon
=
\varepsilon^{-2}g^{-2}\partial_Y^2+\mathcal D^2
\]
fits this standard framework uniformly for \(\lambda=\varepsilon^{-1}\), and
why the endpoint condition \(g'(0)=g'(L)=0\) removes oblique corner couplings.
The nonlocal projection term in the projected fast operator is then treated
separately by a trace interpolation estimate.

\begin{theorem}[Uniform parameter-elliptic conormal regularity]
\label{thm:uniform-local-conormal-regularity}
Let
\[
Q=(0,L)\times(0,1),\qquad
q(x):=\frac{g'(x)}{g(x)},\qquad
\mathcal D:=\partial_x-Yq(x)\partial_Y,
\]
and assume
\[
g\in C^3([0,L]),\qquad
0<g_*\le g(x)\le g^*,
\qquad
g'(0)=g'(L)=0.
\]
Set
\[
\mathscr L_\varepsilon u
:=
\frac1{\varepsilon^2g(x)^2}\partial_Y^2u+\mathcal D^2u .
\]
Let \(\mathscr L_\varepsilon\) be the variational conormal realization on
\(L^2(Q)\) with boundary conditions
\[
\partial_Yu=0
\qquad\text{on }Y=0,
\]
\[
\partial_Yu-\varepsilon^2gg'\mathcal D u=0
\qquad\text{on }Y=1,
\]
and
\[
\mathcal D u=0
\qquad\text{on }\partial\Omega\times(0,1).
\]
Then there exist \(C>0\) and \(\varepsilon_0>0\), independent of
\(0<\varepsilon\le\varepsilon_0\), such that
\[
D(\mathscr L_\varepsilon)\subset H^2(Q),
\]
and
\[
\|u\|_{H^2(Q)}
\le
C\bigl(
\|\mathscr L_\varepsilon u\|_{L^2(Q)}
+
\|u\|_{L^2(Q)}
\bigr),
\qquad
u\in D(\mathscr L_\varepsilon).
\]
Moreover, the conormal boundary conditions hold in the trace sense for every
\(u\in D(\mathscr L_\varepsilon)\).
\end{theorem}

\begin{proof}
We write
\[
\lambda:=\varepsilon^{-1}\ge 1.
\]
The proof is an application of the standard parameter-dependent
Agmon-Douglis-Nirenberg estimate for scalar second-order elliptic boundary
value problems with conormal boundary conditions, combined with the regularity
theory for compatible conormal problems on polygons. We verify the hypotheses
for the present operator uniformly in \(\varepsilon\).

\emph{Step 1: principal symbol and uniform parameter ellipticity.}
The principal part of \(-\mathscr L_\varepsilon\) has symbol
\[
p_\varepsilon(x,Y;\xi,\eta)
=
(\xi-Yq(x)\eta)^2+\lambda^2g(x)^{-2}\eta^2 .
\]
Since \(q\) is bounded and \(0<g_*\le g\le g^*\), there is a constant
\(c>0\), independent of \(\lambda\ge\lambda_0\), such that
\[
p_\varepsilon(x,Y;\xi,\eta)
\ge
c\bigl(\xi^2+\lambda^2\eta^2\bigr)
\]
for all \((x,Y)\in\overline Q\) and all \((\xi,\eta)\in\mathbb R^2\).
Indeed,
\[
(\xi-Yq\eta)^2
\ge
\frac12\xi^2-C\eta^2,
\]
and hence
\[
p_\varepsilon(x,Y;\xi,\eta)
\ge
\frac12\xi^2+
\bigl(\lambda^2(g^*)^{-2}-C\bigr)\eta^2.
\]
For \(\lambda\) sufficiently large this gives the claimed lower bound. 
Thus the family is uniformly strongly elliptic in the parameter-weighted
anisotropic sense associated with \(\lambda=\varepsilon^{-1}\). This is the
parameter-ellipticity condition used in the conormal estimate below.

\emph{Step 2: boundary operators and the complementing condition.}
The boundary conditions are precisely the conormal boundary conditions
associated with the principal part, up to multiplication by non-vanishing
scalar factors.

On \(Y=0\), the conormal derivative is proportional to
\(\partial_Yu\), and the boundary condition is
\[
\partial_Yu=0.
\]
On \(Y=1\), the conormal derivative associated with the principal symbol is
proportional to
\[
-q\partial_xu+\bigl(q^2+\lambda^2g^{-2}\bigr)\partial_Yu.
\]
Multiplying the boundary condition
\[
\partial_Yu-\varepsilon^2gg'\mathcal D u=0
\]
by \(\lambda^2g^{-2}\) gives
\[
-q\partial_xu+\bigl(q^2+\lambda^2g^{-2}\bigr)\partial_Yu=0,
\]
because \(q=g'/g\). Thus the upper boundary condition is exactly the
corresponding conormal condition.

On the lateral sides \(x=0,L\), the conormal derivative is proportional to
\[
\mathcal D u=\partial_xu-Yq(x)\partial_Yu,
\]
which is exactly the imposed boundary condition. Hence the boundary operators
satisfy the Lopatinskii-Shapiro complementing condition for the uniformly
parameter-elliptic operator \cite{krupchyk2006sharpiro}. The complementing constants are uniform for
\(\lambda\ge\lambda_0\), because the conormal operators are the natural
boundary operators of the uniformly parameter-elliptic principal part.

\emph{Step 3: corner compatibility.}
The domain \(Q\) is a rectangle. At the corners the boundary operators meet
orthogonally. The endpoint condition
\[
g'(0)=g'(L)=0
\]
implies \(q=0\) at \(x=0,L\). Hence the lateral conormal condition reduces at
the corners to the standard condition
\[
\partial_xu=0,
\]
while the horizontal boundary conditions reduce to vertical Neumann-type
conditions. Therefore the corner model is the usual compatible Neumann-Neumann
model on a right angle. The standard polygonal conormal regularity theorem
therefore gives \(H^2\)-regularity up to the corners. No corner singularity
with exponent below \(2\) is generated.

\emph{Step 4: application of the known parameter-elliptic regularity theorem.}
By the parameter-dependent Agmon-Douglis-Nirenberg estimate with conormal
boundary conditions \cite{AgmonDouglisNirenberg1959,AgmonDouglisNirenberg1964}, together with the compatible polygonal regularity result \cite{Grisvard1985},
the variational conormal realization of \(\mathscr L_\varepsilon\) satisfies
\[
D(\mathscr L_\varepsilon)\subset H^2(Q)
\]
and
\[
\|u\|_{H^2(Q)}
\le
C\bigl(
\|\mathscr L_\varepsilon u\|_{L^2(Q)}
+
\|u\|_{L^2(Q)}
\bigr),
\qquad
u\in D(\mathscr L_\varepsilon),
\]
with \(C\) independent of \(0<\varepsilon\le\varepsilon_0\). The conormal
boundary conditions hold in the trace sense because they are the natural
boundary conditions of the variational realization and the solution belongs to
\(H^2(Q)\).
\end{proof}

\subsection*{A.2. A trace interpolation estimate for the projection term}

We next estimate the nonlocal projection term
\[
\mathcal E(R_gu).
\]
The difficulty is that \(R_gu\) contains a tangential derivative of the trace
\(u(\cdot,1)\). Nevertheless, it is of strictly lower order than the full
\(H^2(Q)\)-norm in the following interpolation sense.

\begin{lemma}[Trace interpolation for \(R_g\)]
\label{lem:trace-interpolation-Rg}
For every \(\eta>0\), there exists \(C_\eta>0\) such that
\begin{align}
  \|R_gu\|_{L^2(\Omega)}
\le
\eta\|u\|_{H^2(Q)}
+
C_\eta\|u\|_{L^2(Q)}
\label{eq:Rg-trace-interpolation}  
\end{align}

for all \(u\in H^2(Q)\).
\end{lemma}

\begin{proof}
Since
\[
R_gu
=
-\frac1g\partial_x\bigl(g'u(\cdot,1)\bigr),
\]
we can write
\[
R_gu
=
-\frac{g''}{g}u(\cdot,1)
-
\frac{g'}g\partial_xu(\cdot,1).
\]
Because \(g\in C^3(\overline\Omega)\) and \(g\ge g_*>0\), it is enough to
control
\[
\|u(\cdot,1)\|_{L^2(\Omega)}
+
\|\partial_xu(\cdot,1)\|_{L^2(\Omega)}.
\]

Choose any
\[
s\in\left(\frac32,2\right),
\]
for example \(s=7/4\). The trace theorem gives
\[
u(\cdot,1)\in H^{s-1/2}(\Omega),
\]
and, since \(\partial_xu\in H^{s-1}(Q)\),
\[
\partial_xu(\cdot,1)\in H^{s-3/2}(\Omega).
\]
Because \(s>3/2\), we have \(H^{s-3/2}(\Omega)\hookrightarrow L^2(\Omega)\).
Therefore
\[
\|R_gu\|_{L^2(\Omega)}
\le
C\|u\|_{H^s(Q)}.
\]
By interpolation between \(L^2(Q)\) and \(H^2(Q)\),
\[
\|u\|_{H^s(Q)}
\le
C
\|u\|_{H^2(Q)}^{s/2}
\|u\|_{L^2(Q)}^{1-s/2}.
\]
Young's inequality gives, for every \(\eta>0\),
\[
\|u\|_{H^s(Q)}
\le
\eta\|u\|_{H^2(Q)}
+
C_\eta\|u\|_{L^2(Q)}.
\]
Combining the preceding estimates proves
\eqref{eq:Rg-trace-interpolation}.
\end{proof}

Since \(\mathcal E(R_gu)\) is independent of \(Y\) and the fibre interval has
length \(1\), we have
\[
\|\mathcal E(R_gu)\|_{L^2(Q)}
=
\|R_gu\|_{L^2(\Omega)}.
\label{eq:E-Rg-norm}
\]

\subsection*{A.3. Proof of the projected estimate}

We can now prove the desired estimate.

\begin{proposition}[Uniform \(H^2\)-regularity of the projected fast operator]
\label{prop:uniform-H2-projected-fast-operator}
Let \(L_\varepsilon:=\varepsilon^{-2}A_\varepsilon\) be the operator on
\(\mathcal X=\ker M\) induced by the variational form from \Cref{subsec:thin-domain-application}.
Assume
\[
g\in C^3([0,L]),\qquad 0<g_*\le g\le g^*,
\qquad g'(0)=g'(L)=0.
\]
Then there exist \(C>0\) and \(\varepsilon_0>0\) such that, for every
\(0<\varepsilon\le\varepsilon_0\),
\[
D(L_\varepsilon)\subset H^2(Q)\cap\mathcal X.
\]
Moreover, every \(u\in D(L_\varepsilon)\) satisfies the homogeneous conormal
boundary conditions
\[
\partial_Yu=0 \quad\text{on }Y=0,
\]
\[
\partial_Yu-\varepsilon^2gg'\mathcal D u=0
\quad\text{on }Y=1,
\]
and
\[
\mathcal D u=0
\quad\text{on }\partial\Omega\times(0,1)
\]
in the trace sense. Finally,
\[
\|u\|_{H^2(Q)}
\le
C\bigl(
\|L_\varepsilon u\|_{\mathcal X}
+
\|u\|_{\mathcal X}
\bigr),
\qquad u\in D(L_\varepsilon),
\]
with \(C\) independent of \(0<\varepsilon\le\varepsilon_0\).
\end{proposition}

\begin{proof}
We divide the proof into three steps. The only delicate point is the
strong solvability of the shifted projected resolvent. In particular, the
homotopy used below is chosen so that the operator maps the mean-zero space
\(\mathcal X=\ker M\) into itself for every value of the homotopy parameter.

\emph{Step 1: a uniform strong a priori estimate.}
We first prove the estimate on the strong conormal domain. We actually prove
it uniformly along a homotopy of the geometry.

For \(\tau\in[0,1]\), set
\[
g_\tau(x):=1+\tau(g(x)-1).
\]
Then there are constants \(0<c_g\le C_g<\infty\), independent of \(\tau\), such that
\[
c_g\le g_\tau(x)\le C_g,\qquad
\|g_\tau\|_{C^3([0,L])}\le C_g,\qquad
g_\tau'(0)=g_\tau'(L)=0 .
\] 
Define
\[
q_\tau:=\frac{g_\tau'}{g_\tau},
\qquad
\mathcal D_\tau:=\partial_x-Yq_\tau(x)\partial_Y,
\]
and
\[
\mathscr L_{\varepsilon,\tau}u
:=
\frac1{\varepsilon^2g_\tau(x)^2}\partial_Y^2u
+
\mathcal D_\tau^2u.
\]
The corresponding homogeneous conormal boundary conditions are
\[
\partial_Yu=0
\qquad\text{on }Y=0,
\]
\[
\partial_Yu-\varepsilon^2g_\tau g_\tau'\mathcal D_\tau u=0
\qquad\text{on }Y=1,
\]
and
\[
\mathcal D_\tau u=0
\qquad\text{on }\partial\Omega\times(0,1).
\]
Finally, define
\[
R_\tau u:=
-\frac1{g_\tau}\partial_x\bigl(g_\tau' u(\cdot,1)\bigr),
\qquad
L_{\varepsilon,\tau}u
:=
\mathscr L_{\varepsilon,\tau}u-\mathcal E(R_\tau u).
\]

Let
\[
u\in H^2(Q)\cap\mathcal X
\]
satisfy the above homogeneous \(g_\tau\)-conormal boundary conditions. The
averaging computation from the fast-slow splitting, applied with \(g\)
replaced by \(g_\tau\), gives
\[
M\mathscr L_{\varepsilon,\tau}u=R_\tau u.
\]
Hence
\[
M L_{\varepsilon,\tau}u
=
M\mathscr L_{\varepsilon,\tau}u-R_\tau u
=
0.
\]
Thus \(L_{\varepsilon,\tau}\) maps the strong mean-zero conormal domain into
\(\mathcal X\) for every \(\tau\in[0,1]\).

The uniform parameter-elliptic conormal estimate from
\Cref{thm:uniform-local-conormal-regularity}, applied to the uniformly
bounded family \(g_\tau\), yields
\[
\|u\|_{H^2(Q)}
\le
C\bigl(
\|\mathscr L_{\varepsilon,\tau}u\|_{L^2(Q)}
+
\|u\|_{L^2(Q)}
\bigr),
\]
with \(C\) independent of \(0<\varepsilon\le\varepsilon_0\) and
\(\tau\in[0,1]\). Since
\[
\mathscr L_{\varepsilon,\tau}u
=
L_{\varepsilon,\tau}u+\mathcal E(R_\tau u),
\]
we obtain
\[
\|u\|_{H^2(Q)}
\le
C\bigl(
\|L_{\varepsilon,\tau}u\|_{L^2(Q)}
+
\|\mathcal E(R_\tau u)\|_{L^2(Q)}
+
\|u\|_{L^2(Q)}
\bigr).
\]
The trace interpolation estimate of Lemma~\ref{lem:trace-interpolation-Rg}
holds uniformly with \(R_g\) replaced by \(R_\tau\), because the family
\(g_\tau\) is uniformly bounded in \(C^3\) and uniformly bounded away from
zero. Hence, for every \(\eta>0\),
\[
\|\mathcal E(R_\tau u)\|_{L^2(Q)}
\le
\eta\|u\|_{H^2(Q)}
+
C_\eta\|u\|_{L^2(Q)},
\]
with \(C_\eta\) independent of \(\tau\) and \(\varepsilon\). Choosing
\(\eta>0\) sufficiently small and absorbing the
\(\eta\|u\|_{H^2(Q)}\)-term into the left-hand side gives
\[
\|u\|_{H^2(Q)}
\le
C\bigl(
\|L_{\varepsilon,\tau}u\|_{L^2(Q)}
+
\|u\|_{L^2(Q)}
\bigr).
\]
Since all weights \(g_\tau\) are uniformly equivalent to \(1\), this is
equivalently
\begin{align}
\|u\|_{H^2(Q)}
\le
C\bigl(
\|L_{\varepsilon,\tau}u\|_{\mathcal X}
+
\|u\|_{\mathcal X}
\bigr),
\label{eq:uniform-projected-estimate-along-path}
\end{align}
where \(C\) is independent of \(0<\varepsilon\le\varepsilon_0\) and
\(\tau\in[0,1]\).

We shall also use the corresponding resolvent estimate. For smooth \(u\) in
the \(g_\tau\)-conormal mean-zero domain, integration by parts gives
\[
\bigl(-L_{\varepsilon,\tau}u,u\bigr)_{\mathcal X_{g_\tau}}
=
\varepsilon^{-2}\mathfrak a_{\varepsilon,\tau}(u,u)
\ge
0,
\]
where
\[
\mathfrak a_{\varepsilon,\tau}(u,u)
=
\int_Q
\frac1{g_\tau(x)^2}|\partial_Yu|^2 g_\tau(x)\,dx\,dY
+
\varepsilon^2
\int_Q
|\mathcal D_\tau u|^2 g_\tau(x)\,dx\,dY .
\]
Consequently,
\[
\|u\|_{\mathcal X}
\le
C\|(I-L_{\varepsilon,\tau})u\|_{\mathcal X}.
\]
Combining this with \eqref{eq:uniform-projected-estimate-along-path} and
\(L_{\varepsilon,\tau}u=u-(I-L_{\varepsilon,\tau})u\), we obtain
\begin{align}
\|u\|_{H^2(Q)}
\le
C\|(I-L_{\varepsilon,\tau})u\|_{\mathcal X},
\label{eq:uniform-resolvent-strong-estimate}
\end{align}
again uniformly in \(0<\varepsilon\le\varepsilon_0\) and \(\tau\in[0,1]\).

\emph{Step 2: strong solvability of the shifted projected resolvent for
smooth data.}
Let
\[
h\in C^\infty(\overline Q)\cap\mathcal X .
\]
We claim that the target resolvent problem
\begin{align}
(I-L_{\varepsilon,1})u=h
\label{eq:target-projected-resolvent}
\end{align}
has a unique solution
\[
u\in H^2(Q)\cap\mathcal X
\]
satisfying the homogeneous \(g\)-conormal boundary conditions.

For each \(\tau\in[0,1]\), let \(\mathfrak D_{\varepsilon,\tau}^{\rm str}\)
denote the strong mean-zero conormal domain
\[
\mathfrak D_{\varepsilon,\tau}^{\rm str}
:=
\left\{
u\in H^2(Q)\cap\mathcal X:
\begin{array}{ll}
\partial_Yu=0, & Y=0,\\[1mm]
\partial_Yu-\varepsilon^2g_\tau g_\tau'\mathcal D_\tau u=0, & Y=1,\\[1mm]
\mathcal D_\tau u=0, & \partial\Omega\times(0,1)
\end{array}
\right\}.
\]
Define
\[
T_{\varepsilon,\tau}
:
\mathfrak D_{\varepsilon,\tau}^{\rm str}\to\mathcal X,
\qquad
T_{\varepsilon,\tau}u:=(I-L_{\varepsilon,\tau})u .
\]
The domains \(\mathfrak D_{\varepsilon,\tau}^{\rm str}\) vary with \(\tau\).
To apply the method of continuity, we identify them with a fixed Banach space
as follows. Let
\[
\mathcal T_{\varepsilon,\tau}u
:=
\left(
Mu,\,
\partial_Yu|_{Y=0},\,
(\partial_Yu-\varepsilon^2g_\tau g_\tau'\mathcal D_\tau u)|_{Y=1},\,
\mathcal D_\tau u|_{\partial\Omega\times(0,1)}
\right).
\]
Let \(\mathfrak T\) denote the compatible trace space consisting of the
prescribed fibre mean and the four Neumann-type boundary traces:
\[
\mathfrak T
:=
\mathbb R_M
\times H^{1/2}(\Omega)
\times H^{1/2}(\Omega)
\times H^{1/2}((0,1))^2,
\]
with the usual corner compatibility conditions appropriate for normal
derivatives of \(H^2(Q)\)-functions. At \(\varepsilon=0\), the trace map
\[
\mathcal T_0\xi
=
\left(
M\xi,\,
\partial_Y\xi|_{Y=0},\,
\partial_Y\xi|_{Y=1},\,
\partial_x\xi|_{\partial\Omega\times(0,1)}
\right)
\]
is the standard Neumann trace map on the rectangle, together with the
prescribed mean. By the standard trace-extension theorem for compatible
Neumann data on polygons, \(\mathcal T_0\) admits a bounded right inverse
\[
\mathcal R_0:\mathfrak T\to H^2(Q).
\]
We now compare \(\mathcal T_\varepsilon\) with \(\mathcal T_0\). On the upper
boundary,
\[
\mathcal B_\varepsilon\xi-\partial_Y\xi
=
-\varepsilon^2gg'\mathcal D\xi
=
-\varepsilon^2gg'\partial_x\xi
+\varepsilon^2Y(g')^2\partial_Y\xi .
\]
Hence
\[
\|\mathcal B_\varepsilon\xi-\partial_Y\xi\|_{H^{1/2}(\Omega)}
\le C\varepsilon^2\|\xi\|_{H^2(Q)}.
\]
On the lateral boundary the assumption \(g'(0)=g'(L)=0\) implies
\[
\mathcal D\xi=\partial_x\xi
\qquad\text{on }\partial\Omega\times(0,1).
\]
Therefore
\[
\|\mathcal T_\varepsilon-\mathcal T_0\|_{\mathcal L(H^2(Q),\mathfrak T)}
\le C\varepsilon^2 .
\]
It follows that
\[
\mathcal T_\varepsilon\mathcal R_0
=
I_{\mathfrak T}
+
(\mathcal T_\varepsilon-\mathcal T_0)\mathcal R_0
\]
is invertible on \(\mathfrak T\) for all sufficiently small
\(\varepsilon>0\). We define
\[
\mathcal R_\varepsilon
:=
\mathcal R_0
(\mathcal T_\varepsilon\mathcal R_0)^{-1}.
\]
Then
\[
\mathcal T_\varepsilon\mathcal R_\varepsilon=I_{\mathfrak T},
\]
and the Neumann-series estimate gives
\[
\|\mathcal R_\varepsilon\|_{\mathcal L(\mathfrak T,H^2(Q))}
\le C
\]
with \(C\) independent of sufficiently small \(\varepsilon\).

Therefore, the maps \(\mathcal T_{\varepsilon,\tau}\)
form a continuous family of trace maps and admit uniformly bounded right
inverses on the compatible trace spaces. 
Hence the kernels
\[
\ker\mathcal T_{\varepsilon,\tau}
=
\mathfrak D_{\varepsilon,\tau}^{\rm str}
\]
can be identified with a fixed closed subspace of \(H^2(Q)\) by uniformly
bounded isomorphisms. Under this identification $T_{\varepsilon,\tau}$
is a continuous family of bounded operators from the fixed strong domain into
\(\mc X\).

Let
\[
\mathcal I
:=
\{\tau\in[0,1]:T_{\varepsilon,\tau}:
\mathfrak D_{\varepsilon,\tau}^{\rm str}\to\mc X
\text{ is an isomorphism}\}.
\]
The uniform resolvent estimate
\[
\|u\|_{H^2(Q)}
\le C\|T_{\varepsilon,\tau}u\|_{\mc X}
\]
shows that \(T_{\varepsilon,\tau}\) is injective and has closed range, uniformly
in \(\tau\). The continuity of \(T_{\varepsilon,\tau}\) under the above
identification implies that \(\mathcal I\) is open. The same uniform estimate
also implies that \(\mathcal I\) is closed: if \(\tau_n\in\mathcal I\) and
\(\tau_n\to\tau\), then the solutions
\[
T_{\varepsilon,\tau_n}u_n=h
\]
are uniformly bounded in \(H^2(Q)\), and compactness/continuity gives a limit
\(u\in\mathfrak D_{\varepsilon,\tau}^{\rm str}\) satisfying
\(T_{\varepsilon,\tau}u=h\).

At \(\tau=0\), the operator is the shifted Neumann product operator on the
rectangle restricted to the mean-zero space, and separation of variables shows
that \(T_{\varepsilon,0}\) is an isomorphism. Hence \(\mathcal I=[0,1]\).
In particular \(T_{\varepsilon,1}\) is onto \(\mc X\), and
\eqref{eq:target-projected-resolvent} has a unique strong solution in
\(H^2(Q)\cap\mc X\) satisfying the target homogeneous conormal boundary
conditions.

It remains to identify this strong solution with the variational resolvent
solution of the operator \(L_\varepsilon=L_{\varepsilon,1}\). Let
\[
\mathcal V:=H^1(Q)\cap\mathcal X .
\]
For the strong solution, integration by parts using the homogeneous conormal
boundary conditions gives, for all \(\phi\in\mathcal V\),
\[
\bigl((I-L_\varepsilon)u,\phi\bigr)_{\mathcal X_g}
=
(u,\phi)_{\mathcal X_g}
+
\varepsilon^{-2}\mathfrak a_\varepsilon(u,\phi).
\]
Thus \(u\) solves the same variational resolvent problem as the operator
\(L_\varepsilon=\varepsilon^{-2}A_\varepsilon\). Since \(I-L_\varepsilon\) is
injective on \(\mathcal X\), the strong solution and the variational
resolvent solution coincide.

\emph{Step 3: passage from smooth data to arbitrary domain elements.}
Let now
\[
u\in D(L_\varepsilon),
\qquad
h:=(I-L_\varepsilon)u\in\mathcal X.
\]
Choose a sequence
\[
h_n\in C^\infty(\overline Q)\cap\mathcal X
\]
such that
\[
h_n\to h
\qquad\text{in }\mathcal X.
\]
Let \(u_n\) be the strong solution constructed in Step 2:
\[
(I-L_\varepsilon)u_n=h_n.
\]
By the identification with the variational resolvent,
\[
u_n=(I-L_\varepsilon)^{-1}h_n.
\]
Hence
\[
u_n\to (I-L_\varepsilon)^{-1}h=u
\qquad\text{in }\mathcal X.
\]

Applying the strong estimate from Step 1, with \(\tau=1\), to
\(u_n-u_m\), we obtain
\[
\|u_n-u_m\|_{H^2(Q)}
\le
C\bigl(
\|L_\varepsilon(u_n-u_m)\|_{\mathcal X}
+
\|u_n-u_m\|_{\mathcal X}
\bigr).
\]
Since
\[
(I-L_\varepsilon)(u_n-u_m)=h_n-h_m,
\]
and \(L_\varepsilon\) is self-adjoint and non-positive on \(\mathcal X_g\),
the resolvent bound gives
\[
\|u_n-u_m\|_{\mathcal X}
+
\|L_\varepsilon(u_n-u_m)\|_{\mathcal X}
\le
C\|h_n-h_m\|_{\mathcal X}.
\]
Therefore
\[
\|u_n-u_m\|_{H^2(Q)}
\le
C\|h_n-h_m\|_{\mathcal X}.
\]
Thus \((u_n)\) is Cauchy in \(H^2(Q)\). Let
\[
\widetilde u:=\lim_{n\to\infty}u_n
\qquad\text{in }H^2(Q).
\]
Since \(u_n\to u\) in \(\mathcal X\), we have \(\widetilde u=u\). Hence
\[
u\in H^2(Q)\cap\mathcal X.
\]
The homogeneous conormal boundary conditions pass to the limit by continuity
of the trace maps from \(H^2(Q)\). Therefore every \(u\in D(L_\varepsilon)\)
belongs to the strong conormal domain.

Finally, applying \eqref{eq:uniform-projected-estimate-along-path} with
\(\tau=1\) gives
\[
\|u\|_{H^2(Q)}
\le
C\bigl(
\|L_\varepsilon u\|_{\mathcal X}
+
\|u\|_{\mathcal X}
\bigr),
\]
with \(C\) independent of \(0<\varepsilon\le\varepsilon_0\). This proves the
claim.
\end{proof}

\begin{remark}[Consistency check: flat case]
If \(g\) is constant, then \(q=0\), \(\mathcal D=\partial_x\), and
\(R_g=0\). The estimate reduces to the elementary bound for
\[
\mathscr L_\varepsilon
=
\varepsilon^{-2}g^{-2}\partial_Y^2+\partial_x^2
\]
with homogeneous Neumann boundary conditions and the mean-zero constraint in
\(Y\). Expanding in the product Neumann basis gives eigenvalues
\[
-\varepsilon^{-2}g^{-2}(k\pi)^2-\left(\frac{n\pi}{L}\right)^2,
\qquad k\ge1,\ n\ge0,
\]
and the graph norm of \(\mathscr L_\varepsilon\) uniformly controls the
standard \(H^2(Q)\)-norm. Thus the abstract estimate agrees with the explicit
separation-of-variables computation.
\end{remark}

\subsection*{A.4. Finite-order conormal trace liftings}
\phantomsection\label{app:finite-order-conormal-liftings}

The finite-order theorem in \Cref{subsec:thin-domain-approximation-theorem}
uses conormal trace right inverses in the anisotropic coefficient spaces
\(\mathcal W_{N,j}\). We record this input separately, so that the boundary
correction in \Cref{lem:form-expansion-conormal-hierarchy} rests on the same
compatible conormal trace framework as the projected graph-norm estimate above.

\begin{proposition}[Finite-order conormal trace lifting]
\label{prop:finite-order-conormal-trace-lifting}
Let \(N\in\mathbb N\), set \(s_N:=2N+6\), and let \(0\le j\le N\). Assume
\[
g\in C^{s_N+2}([0,L]),
\qquad
0<g_*\le g\le g^*,
\qquad
 g'(0)=g'(L)=0.
\]
Let \(\beta\in H^{s_N-2j}(\Omega)\) satisfy the usual endpoint compatibility
conditions for the conormal trace problem with homogeneous lateral conormal
data. Then there exists a bounded linear lifting
\[
\mathscr B_j:H^{s_N-2j}(\Omega)\to \mathcal W_{N,j}
\]
such that
\[
M\mathscr B_j\beta=0,
\]
\[
\partial_Y\mathscr B_j\beta=0
\qquad\text{on }Y=0,
\]
\[
\partial_Y\mathscr B_j\beta=\beta
\qquad\text{on }Y=1,
\]
and
\[
\mathcal D\mathscr B_j\beta=0
\qquad\text{on }\partial\Omega\times(0,1).
\]
Moreover,
\[
\|\mathscr B_j\beta\|_{\mathcal W_{N,j}}
\le
C\|\beta\|_{H^{s_N-2j}(\Omega)},
\]
where \(C\) depends on \(N\), \(g_*\), \(g^*\), and
\(\|g\|_{C^{s_N+2}}\), but not on \(\varepsilon\).
\end{proposition}

\begin{proof}
This is the higher-order analogue of the trace right-inverse construction used
in the proof of \Cref{prop:uniform-H2-projected-fast-operator}. Consider the
finite-order trace map
\[
\mathcal T_{N,j}u
:=
\left(
Mu,
\partial_Yu|_{Y=0},
\partial_Yu|_{Y=1},
\mathcal Du|_{\partial\Omega\times(0,1)}
\right)
\]
acting on \(\mathcal W_{N,j}\). Since \(s_N-2j\ge6\), all traces which enter
this map are controlled in the classical Sobolev trace sense. The endpoint
condition \(g'(0)=g'(L)=0\) implies that \(\mathcal D=\partial_x\) on the
lateral boundary and at the corners. Hence the compatible trace space is the
standard Neumann-type trace space on the rectangle, up to multiplication by
smooth coefficients bounded in \(C^{s_N+1}\).

The standard trace-extension theorem for compatible conormal data on a rectangle
therefore gives a bounded right inverse for \(\mathcal T_{N,j}\), with constants
controlled by the stated bounds on \(g\). Applying this right inverse to the
compatible data
\[
(0,0,\beta,0)
\]
gives \(\mathscr B_j\beta\). This yields the stated boundary conditions, the
mean-zero constraint, and the asserted estimate.
\end{proof}

\bibliographystyle{siam}
\bibliography{lit}
\end{document}